\documentclass[11pt,reqno]{amsproc}

\title[]{A paradifferential approach for well-posedness of the Muskat problem}

\author{Huy Q. Nguyen}
\address{Department of Mathematics, Brown University, Providence, RI 02912}
\email{hnguyen@math.brown.edu}

\author{Beno\^it Pausader}
\address{Department of Mathematics, Brown University, Providence, RI 02912}
\email{benoit.pausader@math.brown.edu}

\usepackage[margin=1in]{geometry}
\usepackage{amsmath, amsthm, amssymb, mathrsfs, stmaryrd}
\usepackage{times}
\usepackage{color}
\usepackage{hyperref}
\usepackage{graphicx}
\usepackage{subcaption}
\usepackage{float}

\newcommand{\bq}{\begin{equation}}
\newcommand{\eq}{\end{equation}}
\newcommand{\bqa}{\begin{eqnarray*}}
\newcommand{\eqa}{\end{eqnarray*}}

\theoremstyle{plain}
\newtheorem{theo}{Theorem}[section]
\newtheorem{prop}[theo]{Proposition}
\newtheorem{lemm}[theo]{Lemma}
\newtheorem{coro}[theo]{Corollary}

\newtheorem{defi}[theo]{Definition}
\theoremstyle{definition}
\newtheorem{rema}[theo]{Remark}
\newtheorem{nota}[theo]{Notation}
\DeclareMathOperator{\cnx}{div}
\DeclareMathOperator{\RE}{Re}
\DeclareMathOperator{\dist}{dist}

\DeclareMathOperator{\supp}{supp}

\DeclareSymbolFont{pletters}{OT1}{cmr}{m}{sl}
\DeclareMathSymbol{s}{\mathalpha}{pletters}{`s}

\def\tt{\theta}
\def\eps{\varepsilon}
\def\na{\nabla}
\def\la{\left\lvert}
\def\lA{\left\lVert}
\def\ra{\right\vert}
\def\rA{\right\Vert} 
\def\lb{\llbracket}
\def\rb{\rrbracket}
\def\les{\lesssim}

\def\mez{\frac{1}{2}}
\def\tdm{\frac{3}{2}}

\def\Rr{\mathbb{R}}
\def\T{\mathbb{T}}
\def\Nn{\mathbb{N}}
\def\Cc{\mathbb{C}}

\def\cF{\mathcal{F}}

\def\cN{\mathcal{N}}

\def\ld{\lambda}

\def\p{\partial}
\def\na{\nabla}
\def\wc{\rightharpoonup}
\def\ka{\kappa}
\def\ma{\mathfrak{a}}
\def\mb{\mathfrak{b}}

\def\wt{\widetilde}

\numberwithin{equation}{section}

\date{today}
\begin{document}
\begin{abstract}
We study the Muskat problem for  one fluid or two fluids, with or without viscosity jump, with or without rigid boundaries, and in arbitrary space dimension $d$ of the interface. The Muskat problem is scaling invariant in the Sobolev space $H^{s_c}(\mathbb{R}^d)$ where $s_c=1+\frac{d}{2}$. Employing a paradifferential approach, we prove local well-posedness for large data in any subcritical Sobolev spaces $H^s(\mathbb{R}^d)$, $s>s_c$. Moreover, the rigid boundaries are only required to be Lipschitz and can have arbitrarily large variation.  The Rayleigh-Taylor stability condition is assumed for the case of two fluids with viscosity jump but is proved to be automatically satisfied for the case of one fluid. The starting point of this work is  a reformulation  solely in terms of the Drichlet-Neumann operator. The key elements of proofs are new paralinearization and contraction results for the Drichlet-Neumann operator in rough domains.
\end{abstract}

\keywords{Muskat, Hele-Shaw, Free boundary problems, Regularity, Paradifferential calculus.}

\noindent\thanks{\em{ MSC Classification: 35R35, 35Q35, 35S10, 35S50, 76B03.}}

\maketitle


\section{Introduction}

\subsection{The Muskat problem} In its full generality, the  Muskat problem describes the dynamics of two immiscible fluids  in a porous medium with different densities $\rho^\pm$ and different viscosities $\mu^\pm$. Let us denote the interface between the two fluids by $\Sigma$ and assume that it is the graph of a time-dependent function $\eta(x, t)$, i.e. 
\bq\label{Sigma}
\Sigma_t=\{(x, \eta(t,x)): x\in \Rr^d\}.
\eq
The associated time-dependent fluid domains are then given by
\bq\label{Omega+}
\Omega^+_t=\{(x, y)\in \Rr^d\times \Rr: \eta(t, x)<y<\underline{b}^+(x)\}
\eq
and
\bq\label{Omega-}
\Omega^-_t=\{(x, y)\in \Rr^d\times \Rr: \underline{b}^-(x)<y<\eta(t, x)\}
\eq
where $\underline{b}^\pm$ are the parametrizations of the rigid boundaries
\bq\label{Gamma:pm}
\Gamma^\pm =\{(x, \underline{b}^\pm(x)): x\in \Rr^d\}.
\eq


\begin{figure}[H]
\begin{subfigure}{.5\textwidth}
  \centering
  \includegraphics[width=.6\linewidth]{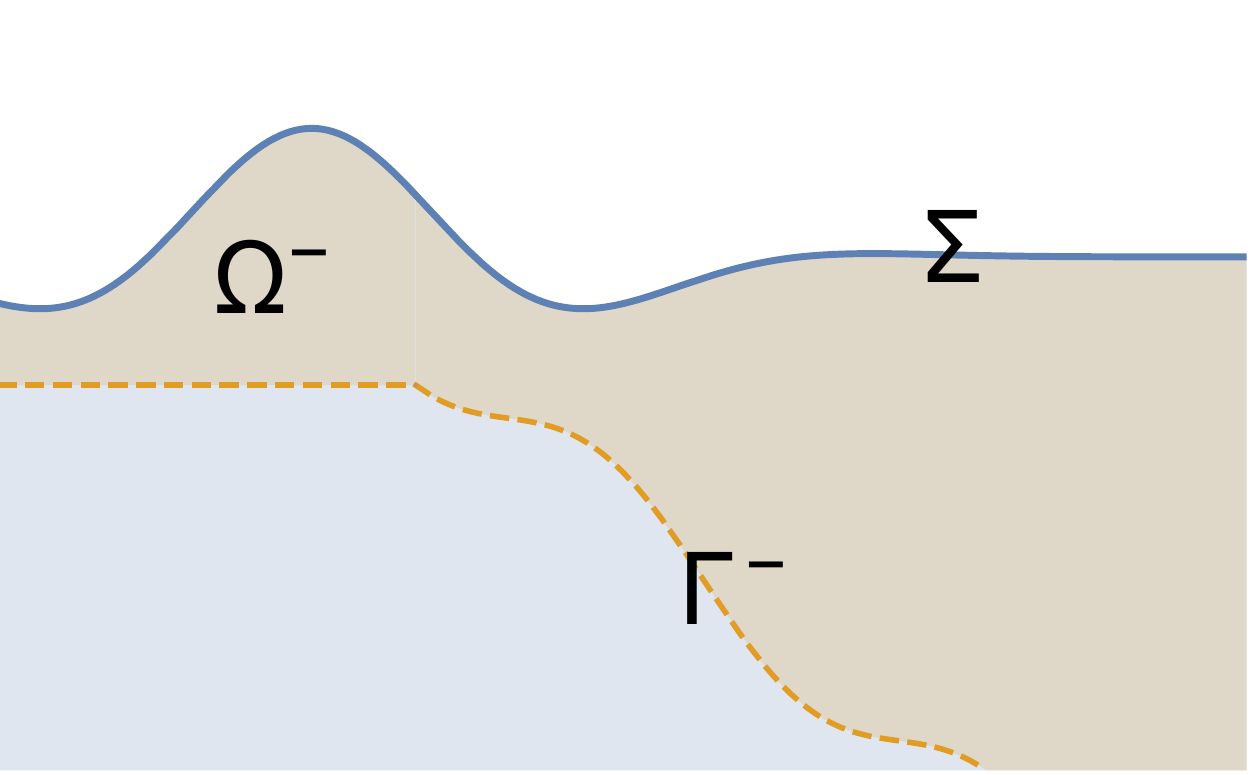}
  \caption{The one-phase problem}
\end{subfigure}%
\begin{subfigure}{.5\textwidth}
  \centering
  \includegraphics[width=.6\linewidth]{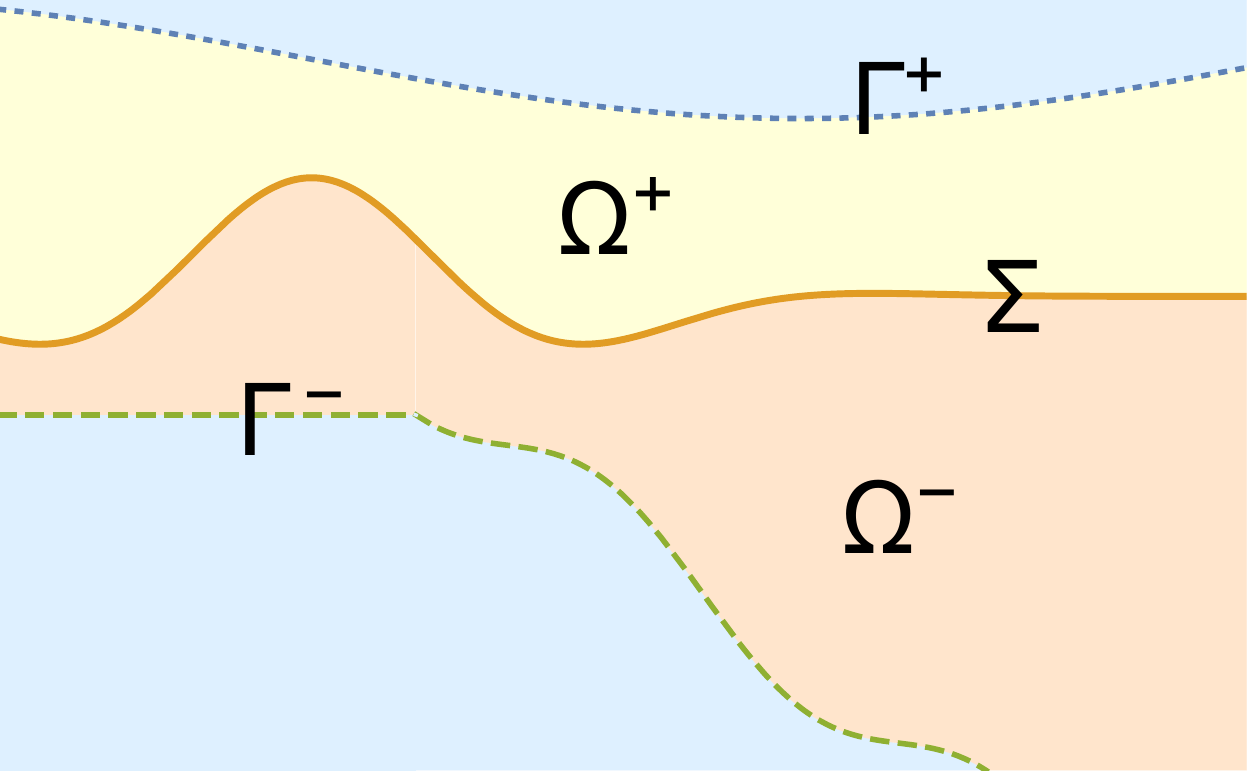}
  \caption{The two-phase problem}
\end{subfigure}
\label{fig:fig}
\end{figure}

The incompressible fluid velocity $u^\pm$ in each region is governed by Darcy's law:
\bq\label{Darcy:pm}
\mu^\pm u^\pm+\nabla_{x, y}p^\pm=-(0, \rho^\pm)\quad \text{in}~\Omega^\pm_t,
\eq
and 
\bq
\cnx_{x, y} u^\pm =0\quad \text{in}~\Omega^\pm_t.
\eq
Note that we have normalized  gravity  to $1$ in \eqref{Darcy:pm}.

At the interface $\Sigma$, the normal velocity is continuous:
\bq\label{u.n:pm}
u^+\cdot n=u^-\cdot n\quad\text{on}~\Sigma_t
\eq
where $n=\frac{1}{\sqrt{1+|\nabla \eta|^2}}(-\nabla \eta, 1)$ is the upward pointing unit normal to $\Sigma_t$. Then, the interface moves with the fluid:
\bq\label{kbc:pm}
\p_t\eta=\sqrt{1+|\nabla \eta|^2}u^-\cdot n\vert_{\Sigma_t}.
\eq
By neglecting the effect of surface tension, the pressure is continuous at the interface: 
\bq\label{p:pm}
p^+=p^-\quad\text{on}~\Sigma_t.
\eq
Finally, at the two rigid boundaries, the no-penetration boundary conditions are imposed:
\bq\label{bc:b:pm}
u^\pm\cdot \nu^\pm=0\quad \text{on}~\Gamma^\pm
\eq
where $\nu^\pm=\pm\frac{1}{\sqrt{1+|\nabla \underline b^\pm|^2}}(-\nabla \underline b^\pm, 1)$ denotes the outward pointing unit normal to $\Gamma^\pm$. We will also consider the case that at least one of $\Gamma^\pm$ is empty (infinite depth); \eqref{bc:b:pm} is  then replaced by the vanishing of $u$ at infinity. 

We shall refer to the system \eqref{Omega+}-\eqref{bc:b:pm} as the two-phase Muskat problem. When the top phase corresponds to vacuum, i.e. $\mu^+=\rho^+=0$, the two-phase Muskat problem reduces to the one-phase Muskat problem and \eqref{p:pm} becomes
\bq
p^-=0\quad\text{on}~\Sigma_t.
\eq

\subsection{Presentation of the main results}

It turns out that the Muskat problem can be recast as a {\it quasilinear} evolution problem of the interface $\eta$ only (see e.g. \cite{AlaMeuSme,CorGan07, CorLaz, GanGarPatStr, Mat}). Moreover, in the case of infinite bottom, if $\eta(t,x)$ is a solution then so is
\begin{equation*}
\eta_\lambda (t,x):=\lambda^{-1}\eta(\lambda t, \lambda x),\qquad \lambda>0
\end{equation*}
and thus the Sobolev space $H^{1+\frac d2}(\Rr^d)$ is scaling invariant. Our main results assert that the Muskat problem in arbitrary dimension is locally well-posed for large data in {\it all subcritical Sobolev spaces} $H^s(\Rr)$, $s>1+\frac{d}{2}$, either in the case of one fluid or the case of two fluids with or without viscosity jump, and when the bottom is either empty or is  the graph of a Lipshitz function with arbitrarily large variation. We state here an informal version of our main results and refer to Theorem \ref{theo:wp1} and Theorem \ref{theo:wp2} for precise statements.

\begin{theo}[Informal version]\label{theo:informal}Let $d\ge1$ and $s>1+\frac{d}{2}$.

$(i)$ (The one-phase problem) Consider $\rho^->\rho^+=0$ and $\mu^->\mu^+=0$. Assume either that the depth is infinite or that the bottom is the graph of a  Lipschitz function that does not touch the surface. Then the one-phase Muskat problem is locally well posed in $H^s(\Rr^d)$.
 
$(ii)$ (The two-phase problem)  Consider  $\rho^->\rho^+>0$ and $\mu^\pm>0$. Assume that the upper and lower boundaries are either empty or graphs of  Lipschitz functions that do not touch the interface. The two-phase Muskat problem is locally well posed in $H^s(\Rr^d)$ in the sense that any initial data in $H^s(\Rr^d)$ satisfying the Rayleigh-Taylor condition leads to a unique solution in $L^\infty([0, T]; H^s(\Rr^d))$ for some $T>0$.
\end{theo}
The starting point of our analysis is the fact that the Muskat problem has a very simple reformulation in terms of the Dirichlet-Neumann map $G$ (see the definition \eqref{defG:intro} below); most strikingly, in the case of one fluid, it is equivalent to
\begin{equation}\label{SimplifiedMuskat}
\partial_t\eta+G(\eta)\eta=0.
\end{equation}
See Proposition \ref{reform:1}. This makes it clear that
\begin{itemize}
\item Any precise result on the continuity of the Dirichlet-Neumann map leads to direct application for the Muskat problem. This is especially relevant in view of the recent intensive work in the context of water-waves  \cite{ABZ1,ABZ3, ThoGuy, PoyNg1, LannesJAMS}.
\item The Muskat problem is the natural parabolic analog of the water-wave problem and as such is a useful toy-model to understand some of the outstanding challenges for the water-wave problem. 
\end{itemize}

The second point above applies to the study of possible splash singularities, see \cite{CasCorFefGan,CasCorFefGanMar}. Another problematic is the question of optimal low-regularity well-posedness for quasilinear problems. This seems a rather formidable problem for water-waves since the mechanism of dispersion is harder to properly pin down in the quasilinear case (see \cite{Ai,ABZ1,ABZ3, ABZ4, PoyNg2,HunIfrTat, Ng, Wu2d,Wu3d}), but becomes much more tractable in the case of the Muskat problem due to its parabolicity. This is the question we consider here.

The Muskat problem exists in many incarnations:  with or without viscosity jump, with or without surface tension, with or without bottom, with or without permeability jump, in $2d$ or $3d$, when the interface are graphs or curves\dots Our main objective is to provide a flexible approach that covers many aspects at the same time and provides almost sharp well-posedness results. The main questions that we do not address here are
\begin{itemize}
\item The case of surface tension or jump in permeability (see e.g. \cite{Mat, Mat2, BerCorGra, GraShk, Per}). This can also be covered by the paradifferential formalism, but we decided to leave it for another work in order to highlight the centrality of the Dirichlet-Neumann operator. 
\item The case the interface is not a graph. We believe that so long as the interface is a graph over {\it some} smooth reference surface, the approach here may be adapted, but this would require substantial additional technicalities.
\item The case of beaches when the bottom and the interface meet. This is again a difficult problem (see e.g. \cite{Poy}).
\item The case of critical regularity. This is a delicate issue, especially for large data, or in the presence of corners. We believe that the approach outlined here could lead to interesting new insights into this question, but the estimates we provide would need to be significantly refined.
\end{itemize}
Finally, let us stress the fact that in our quasilinear case, there is a significant difference between small and large data, even for local existence. Indeed, the solution is created through some scheme which amounts to decomposing
\begin{equation*}
\begin{split}
\partial_t\eta=\mathcal{D}\eta+\Pi
\end{split}
\end{equation*}
where $\partial_t-\mathcal{D}$ can be more or less explicitly integrated, while $\Pi$ contains the perturbative terms. There are two ways the terms can be perturbative in an expansion
\begin{enumerate}
\item because they are small and at the same level of regularity,
\item because they are more regular.
\end{enumerate}
The first possibility allows, in the case of small data to bypass the precise understanding of the terms entailing derivative losses, so long as they are compatible with the regularity of solutions to $(\partial_t-\mathcal{D})\eta=0$. In our case, when considering large data, we need to extract the terms corresponding to the loss of derivatives in \eqref{SimplifiedMuskat} and this is where the {\it paradifferential calculus approach} is particularly useful. 
\subsection{Prior results}
The Muskat problem was introduced in \cite{Mus} and has recently been the subject of intense study, both numerically and analytically. Interestingly, the Muskat problem is mathematically analogous to the Hele-Shaw  problem \cite{Hel1,Hel2} for viscous flows between two closely spaced parallel plates. We will mostly discuss the issue of well-posedness and refer to \cite{CasCorFefGan, CasCorFefGanMar,GomGra} for interesting results on singularity formation and to  \cite{GraLaz,GanSur} for recent reviews on the Muskat problem.  In the case of small data and infinite depth, global strong solutions have been constructed in subcritical spaces \cite{ Cam, Chen, CheGraShk, ConCorGanRodPStr, ConCorGanStr, ConGanShvVic, ConPug, CorGan07, CorLaz, SieCafHow} and in critical spaces \cite{GanGarPatStr}. We note in particular that \cite{CorLaz} allows for interfaces with large slope and \cite{CheGraShk, GanGarPatStr} allow for viscosity jump. Global weak solutions were obtained in \cite{ConCorGanStr, DenLeiLin}.

As noted earlier there is a significant difference between small and large data for this quasilinear problem. We now discuss in detail the issue of local well-posedness for large data. Early results on  local well-posedness for large data in Sobolev spaces  date back to \cite{Chen, EscSim, Yi} and  \cite{Amb0, Amb}. C\'ordoba and Gancedo \cite{CorGan07} introduced  the contour dynamic formulation for the Muskat problem without viscosity jump and with infinite depth, and proved local well-posedness in $H^{d+2}(\Rr^d)$, $d=1, 2$; here the interface is the graph of a function. In \cite{CorCorGan, CorCorGan2}, C\'ordoba, C\'ordoba and Gancedo  extended this result to the case of viscosity jump and nongraph interfaces satisfying the arc-chord and the Rayleigh-Taylor conditions.  Note that in the case with viscosity jump, one needs to invert a highly nonlocal equation to obtain the vorticity as an operator of the interface. Using an ALE (Arbitrary Lagrangian-Eulerian) approach, Cheng,  Granero and Shkoller  \cite{CheGraShk} proved local well-posedness for the one-phase problem with flat bottom when the initial surface $\eta\in H^2(\T)$ which allows for unbounded curvature. This result was then extended by Matioc \cite{Mat2} to the case of viscosity jump (but no bottom). For the case of constant viscosity, using nonlinear lower bounds, the authors in \cite{ConGanShvVic} obtained local well-posedness for $\eta\in W^{2, p}(\Rr)$ with $p\in (1, \infty]$. Note that $W^{2, 1}(\Rr)$ is scaling invariant yet requires $1/2$ more derivative compared to $H^{3/2}(\Rr)$. By rewriting the problem as an abstract parabolic equation in a suitable functional setting, Matioc \cite{Mat} sharpened the local well-posedness theory to $\eta \in H^{3/2+\eps}(\Rr)$ for the case of constant viscosity and infinite depth. This covers all subcritical ($L^2$-based) Sobolev spaces for the given one-dimensional setting. We also note the recent work of Alazard-Lazar \cite{AlaLaz} which extends this result by allowing non $L^2$-data.

Our Theorem \ref{theo:informal} thus confirms local-wellposedness for large data in all subcritical Sobolev spaces for a rather general setting allowing for  viscosity jump, large bottom variations and higher dimensions. A notable feature of our approach is that it is entirely phrased in terms of the Dirichlet-Neumann operator and as a result, once this operator is properly understood, there is no significant difficulty in passing from constant viscosity to viscosity jump. Furthermore, we obtain an explicit quasilinear parabolic form (see \eqref{reduce:intro} and \eqref{reduce:intro2}) of the Muskat problem by extracting the elliptic and the transport part in the nonlinearity. 
\subsection{Organization of the paper}

In Section \ref{SecMainRes}, we reformulate the Muskat problem in terms of the  Dirichlet-Neumann operator and present the main results of the paper. In Section \ref{DefDN}, we properly define the Dirichlet-Neumann operator in our setting and obtain preliminary low-regularity bounds which are then used to obtain paralinearization and contraction estimates in higher norms via a paradifferential approach. These are key technical ingredients for the proof of the main results which are given in  Section \ref{SecMainProof}. Appendix \ref{appendix:trace} gathers trace theorems for homogeneous Sobolev spaces; Appendix \ref{appendixB} is devoted to the proof of \eqref{RT} and \eqref{RT2}; finally, a review of the paradifferential calculus machinery is presented in Appendix \ref{appendix:para}.

 
 \section{Reformulation and main results}\label{SecMainRes}
\subsection{Reformulation}
In order to state our reformulation for the Muskat problem, let us define the Dirichlet-Neumann operators $G^\pm(\eta)$ associated to $\Omega^\pm$. For a given function $f$, if $\phi^\pm$ solve 
\bq\label{elliptic:intro}
\begin{cases}
\Delta_{x, y} \phi^\pm=0\quad\text{in}~\Omega^\pm,\\
\phi^\pm=f\quad\text{on}~\Sigma,\\
\frac{\p\phi^\pm}{\p \nu^\pm}=0\quad\text{on}~\Gamma^\pm.
\end{cases}
\eq
then 
\bq\label{defG:intro}
G^\pm(\eta) f:=\sqrt{1+|\nabla \eta|^2}\frac{\p\phi^\pm}{\p n}.
\eq
The Dirichlet-Neumann operator will be studied in detail in Section \ref{DefDN}. We can now restate the Muskat problem in terms of $G^\pm$.

\begin{prop}\label{reform:1} Let $d\ge 1$.

$(i)$ If $(u, p, \eta)$ solve the one-phase Muskat problem then $\eta:\Rr^d\to \Rr$  obeys the equation 
\bq
\p_t\eta+\kappa G^{-}(\eta)\eta=0,\quad \kappa=\frac{\rho^-}{\mu^-}.\label{eq:eta}
\eq
Conversely, if $\eta$  is a solution of \eqref{eq:eta} then the one-phase Muskat problem has a solution which admits $\eta$ as the free surface.

$(ii)$ If $(u^\pm, p^\pm, \eta)$ is a solution of the two-phase Muskat problem then 
\bq\label{eq:eta2p}
\p_t\eta=-\frac{1}{\mu^-}G^{-}(\eta)f^-,
\eq
where $f^\pm:=p^\pm\vert_{\Sigma}+\rho^\pm \eta$ satisfy
\bq\label{system:fpm}
\begin{cases}
 f^+-f^-= (\rho^+-\rho^-)\eta,\\
\frac{1}{\mu^+}G^+(\eta)f^+-\frac{1}{\mu^-}G^-(\eta)f^-=0.
\end{cases}
\eq
Conversely, if $\eta$ is a solution of  \eqref{eq:eta2p} with $f^\pm$ solution of \eqref{system:fpm} then the two-phase Muskat problem has a solution which admits $\eta$ as the free interface.
\end{prop}

We refer to \cite{AlaMeuSme,ChlGuiSch} for similar reformulations and derivation of a number of interesting properties.

\begin{proof}
$(i)$ Assume first that $(u^{-}, p^{-}, \eta)$ solve the one-phase Muskat problem. Setting $q=p^{-}+\rho^- y$, then $q$ solves the elliptic problem
\bq\label{elliptic:q}
\Delta_{x, y}q=0\quad\text{in}~\Omega^{-}_t,\qquad q=\rho^-\eta\quad\text{on}~\Sigma_t,\qquad \p_{\nu} q=0\quad\text{on}~\Gamma^{-}.
\eq
Since $\sqrt{1+|\nabla \eta|^2}u^{-}\cdot n\vert_{\Sigma_t}=-G^{-}(\eta)(\rho^-\eta)$, \eqref{eq:eta} follows from  \eqref{kbc:pm} and \eqref{Darcy:pm}.

Conversely, if $\eta$ satisfies \eqref{eq:eta}  then the pressure $p^{-}=q-\rho^- y$ is obtained by solving \eqref{elliptic:q}, and the velocity is determined from the Darcy's law \eqref{Darcy:pm}. 

$(ii)$ As before, \eqref{eq:eta2p} follows from \eqref{kbc:pm} and  \eqref{Darcy:pm} for $\Omega^-$. The jump of $f$ in \eqref{system:fpm} is a consequence of the continuity \eqref{p:pm} of the pressure. Lastly, the jump of Dirichlet-Neumann operators is exactly the continuity \eqref{u.n:pm} of  the normal velocity. Conversely, if $\eta$ is known then $(u^\pm, p^\pm)$ can be easily determined. 
\end{proof}

\begin{rema}
For a given function $\eta\in W^{1, \infty}(\Rr^d)\cap H^\mez(\Rr^d)$, we prove in Proposition \ref{prop:fpm} below that there exists a unique pair $f^\pm$ solving \eqref{system:fpm} in a variational sense. 
\end{rema}
\subsection{Main results}

The Rayleigh-Taylor stability condition requires that the pressure is increasing in the normal direction when crossing the interface from the top fluid to the bottom fluid. More precisely, 
\bq
\text{RT}(x, t)=(\na_{x, y}p^+-\na_{x, y}p^-)\cdot n\vert_{\Sigma_t}>0.
\eq
In terms of $\eta$ and $f^\pm$, we have 
\bq\label{RT}
\text{RT}(x, t)=(1+|\na\eta|^2)^\mez \big(\lb \rho\rb-\lb B\rb\big),
\eq
where 
\[
\lb\rho\rb=\rho^{-}-\rho^+,\quad \lb B\rb=B^--B^+
\]
and
\[
B^\pm=\frac{\nabla \eta\cdot \nabla f^\pm+G^\pm(\eta)f^\pm}{1+|\nabla \eta|^2}.
\]
Using the Darcy law \eqref{Darcy:pm} we can write that
\bq\label{RT2}
\text{RT}(x, t)=\lb \rho\rb(1+|\na \eta|^2)^{-\mez}+\lb \mu\rb u\cdot n,\quad \lb \mu\rb=(\mu^--\mu^+).
\eq
See Appendix \ref{appendixB} for the proof of \eqref{RT} and \eqref{RT2}. Let us denote
\bq\label{def:spaceZs}
Z^s(T)=L^\infty([0, T]; H^s(\Rr^d))\cap L^2([0, T]; H^{s+\mez}(\Rr^d)).
\eq

For the one-phase problem, we prove local well-posedness {\it without} assuming the Rayleigh-Taylor stability condition which in fact {\it always holds}, even in finite depth (see Remark \ref{rema:RT}).
\begin{theo}\label{theo:wp1}
 Let $\mu^->0$ and $\rho^->0$. Let $s>1+\frac{d}{2}$ with $d\ge 1$. Consider either $\Gamma^-=\emptyset$ or $\underline{b}^{-}\in  \dot W^{1, \infty}(\Rr^d)$. Let  $\eta_0\in H^s(\Rr^d)$ satisfy 
\bq
 \dist(\eta_0, \Gamma^{-})\ge 2h>0.
 \eq
   Then, there exist a positive time $T$ depending only on  $(s, \kappa)$, $h$ and $\| \eta_0\|_{H^s(\Rr^d)}$, and a unique solution $\eta\in Z^s(T)$  to equation \eqref{eq:eta} such that $\eta\vert_{t=0}=\eta_0$ and
 \bq
\dist(\eta(t), \Gamma^{-})\ge h\quad\forall t\in [0, T].
\eq
Furthermore, the $L^2$ norm of $\eta$ in nonincreasing in time. 
\end{theo}
As for the two-phase problem, we prove local well-posedness in the stable regime ($\rho^+<\rho^-$) for large data satisfying the Rayleigh-Taylor stability condition. 
\begin{theo}\label{theo:wp2}
Let $\mu^\pm>0$ and $\rho^->\rho^+>0$. Let $s>1+\frac{d}{2}$ with $d\ge 1$. Consider  any combination of $\Gamma^\pm=\emptyset$ and $\underline{b}^\pm\in  \dot W^{1, \infty}(\Rr^d)$. Let  $\eta_0\in H^s(\Rr^d)$ satisfy 
\begin{align}
& \dist(\eta_0, \Gamma^\pm)\ge 2h>0,\\
 &\inf_{x\in \Rr^d}\mathrm{RT}(x, 0)>2\ma>0.
 \end{align}
  Then, there exist a positive time $T$ depending only on $(s, \mu^\pm, \lb\rho\rb)$, $(h, \ma)$ and $\| \eta_0\|_{H^s(\Rr^d)}$, and a unique solution 
  $ \eta\in Z^s(T)$ to equations \eqref{eq:eta2p}-\eqref{system:fpm} such that $\eta\vert_{t=0}=\eta_0$,
 \begin{align}
&\dist(\eta(t), \Gamma^\pm)\ge h\quad\forall t\in [0, T],\\
&\inf_{(x, t)\in \Rr^d\times [0, T]}\mathrm{RT}(x, t)> \ma.
\end{align}
Furthermore, the $L^2$ norm of $\eta$ is nonincreasing in time.
\end{theo}
Several remarks on our main results are in order.
\begin{rema}
The solutions constructed in Theorems \ref{theo:wp1} and \ref{theo:wp2} are unique in  $L^\infty_tH^s_x$ and the solution maps are locally Lipschitz in  $L^\infty_tH^{s}_x$ with respect to the topology of $L^\infty_tH^{s-1}_x$. The proof of Theorem \ref{theo:wp2} also provides the following estimate for $f^\pm$
\bq
\| f^\pm\|_{ \wt H^s_\pm(\Rr^d)}\le \cF(\| \eta\|_{H^s(\Rr^d)})\| \eta\|_{H^s(\Rr^d)}
\eq
where the space $ \wt H^s_\pm(\Rr^d)$ is defined by \eqref{def:wtHs}. Modulo some minor modifications, our proofs work equally for the periodic case. 
\end{rema}
\begin{rema}\label{rema:RT}
The Rayleigh-Taylor (RT) condition is ubiquitous in free boundary problems. For irrotational water-waves (one fluid), Wu \cite{Wu2d} proved that this condition is automatically satisfied if there is no bottom. In the presence of a bottom that is the graph of a function, Lannes \cite{LannesJAMS}  proved this condition assuming that the second fundamental form of the bottom is sufficiently small, covering the case of flat bottoms. In the context of the  Muskat problem, there are various scenarios for the stable regime $\rho^+<\rho^-$. When the interface is a general curve/surface, the RT condition was assumed in \cite{Amb0, CorCorGan, CorCorGan2}. On the other hand, when the interface is  a graph, we see from \eqref{RT2} that this condition always holds if there is no viscosity jump but need not be so otherwise. In particular, for the one-phase problem, the local well-posedness result in \cite{CheGraShk} assumes the RT condition for flat bottoms. However, we prove in Proposition \ref{lemm:upperbounG} that the RT condition holds in the one-phase case so long as the bottom is either empty or is the graph of a Lipschitz function which can be unbounded and have large variation. 
\end{rema}
\begin{rema}
When the surface tension effect is taken into account, well-posedness holds without the Rayleigh-Taylor condition. It turns out that $H^{1+\frac{d}{2}}(\Rr^d)$ is also the scaling invariant Sobolev space for the Muskat problem with surface tension. Local well-posedness for all subcritical data in $H^s(\Rr^d)$, $s>1+\frac{d}{2}$, is established in \cite{Ng1}. Furthermore, at the same level of regularity, \cite{FlyNgu} proves that solutions constructed in Theorems \ref{theo:wp1} and \ref{theo:wp2} are limits of solutions to the problem with surface tension as surface tension vanishes. 
\end{rema}
\subsection{Strategy of proof}
Let us briefly explain our strategy for a priori estimates. The main step consists in obtaining  a precise paralinearization  for the Dirichlet-Neumann operator $G^\pm(\eta)f$ when $\eta\in H^s(\Rr^d)$, $s>1+\frac d2$ and $f$ has the maximal regularity $H^s(\Rr^d)$. We prove in Theorem \ref{tame:RDN} that 
\bq\label{paralin:intro}
G^{-}(\eta)f=T_\lambda(f-T_B\eta)-T_V\cdot\nabla \eta+R^-(\eta)f
\eq
where $(B, V)$ are explicit functions (see \eqref{BV}), $\ld$ is an elliptic  first-order symbol (see \eqref{ld}) and the remainder $R^-(\eta)f$ obeys 
\[
\| R^-(\eta)f\|_{H^{s-\mez}}\le \cF(\| \eta\|_{H^s})\big(1+\| \eta\|_{H^{s+\mez-\delta}}\big)\Vert f\Vert_{\widetilde{H}^s_-}
\]
provided that $0<\delta< \min(s-1-\frac d2, \mez)$. Here we note that the term $f-T_B\eta$ comes from the consideration of Alinhac's good unknown. 

1) For the one-phase problem \eqref{eq:eta}, taking $f=\eta$ yields
\bq\label{reduce:intro}
\p_t \eta
= -\ka T_{\lambda(1-B)}\eta+\ka T_V\cdot\nabla \eta+ R_1
\eq
where 
\bq\label{R:intro}
\| R_1\|_{H^{s-\mez}}\le \cF(\| \eta\|_{H^s})\big(1+\| \eta\|_{H^{s+\mez-\delta}}\big)\Vert \eta\Vert_{H^s}.
\eq
We observe that the transport term $T_V\cdot\nabla \eta$ is harmless for energy estimates and the term $-\ka T_{\lambda(1-B)}\eta$ would give the parabolicity if $1-B>0$. Then this latter term entails a gain of $\mez$ derivative when measured in $L^2_t$, compensating the loss of $\mez$ derivative in the remainder $R_1$. Moreover, the fact that the highest order term $\| \eta\|_{H^{s+\mez-\delta}}$ in \eqref{R:intro} appears linearly with a gain of $\delta$ derivative gives room to choose the time $T$ as a small parameter. We thus obtain a closed a priori estimate in $L^\infty_tH^s_x\cap L^2_tH^{s+\mez}_x$. Finally, we prove in Proposition \ref{lemm:upperbounG} that the stability condition $1-B>0$ is automatically satisfied.

2) As for the two-phase problem \eqref{eq:eta2p}-\eqref{system:fpm}, we apply the paralinearization \eqref{paralin:intro} and obtain a reduced equation similar to \eqref{reduce:intro}:
\bq\label{reduce:intro2}
\p_t\eta=-\frac{1}{\mu^++\mu^-}T_{\lambda(\lb\rho\rb-\lb B\rb)}\eta+\frac{1}{\mu^++\mu^-}T_{\lb V\rb}\cdot\nabla \eta+R_2
\eq
where $R_2$ obeys the same bound  \eqref{R:intro} as $R_1$. Consequently, the parabolicity holds if $\lb\rho\rb-\lb B\rb>0$ and  in view of \eqref{RT}, this is equivalent to $\mathrm{RT}>0$. This shows a remarkable link between Alinhac's good unknown and the Rayleigh-Taylor stability condition. 

Finally, we remark that the contraction estimate for the solutions requires a fine contraction estimate for the Dirichlet-Neumann operator, see Theorem \ref{theo:contractionDN2}.

\section{The Dirichlet-Neumann operator: continuity, paralinearization and contraction estimates}\label{DefDN}

This section is devoted to the study of the Dirichlet-Neumann operator. For the two-phase problem \eqref{eq:eta2p}, the function $f^-$ obtained from solving \eqref{system:fpm} is only determined up to additive constants and we need to define $G^-(\eta)f$ for $f$ belonging to a suitable homogeneous space. Since $f$ is the trace of a harmonic function (see \eqref{elliptic:intro}) with bounded gradient in $L^2$,  the trace theory recently developed in \cite{LeoTice} is perfectly suited for this purpose, allowing us to take $f$ in a ``screened'' homogeneous Sobolev space (see \eqref{def:wtHmez}) tailored to the bottom. This is the content of  Subsection \ref{section:contDN}, where we obtain existence and modest regularity of the variational solution to the appropriate Dirichlet problem. 

Next in Subsection \ref{section:paralin}, we obtain a precise paralinearization for $G^-(\eta)f$ by extracting all the first order symbols. This is done when $\eta$ has subcritical regularity $H^s(\Rr^d)$, $s>1+\frac d2$, and $f$ has the maximal regularity $H^s(\Rr^d)$. The error estimate is precise enough to obtain closed a priori estimates afterwards. Finally, in Subsection \ref{section:contraction} we prove a contraction estimate for $G^-(\eta_1)-G^-(\eta_2)$, showing a gain of derivative for $\eta_1-\eta_2$ which will be crucial for the contraction estimate of solutions. 
\subsection{Definition and continuity}\label{section:contDN}
We  study the Dirichlet-Neumann problem associated to the fluid domain $\Omega^-$ underneath the free interface $\Sigma=\{(x,\eta(x)): x\in \Rr^d\}$.  Here and in what follows, the time variable is frozen. We say that a function $g:\Rr^d\to \Rr$ is Lipschitz, $g\in \dot W^{1, \infty}(\Rr^d)$, if $\na g\in L^\infty(\Rr^d)$.  As for the bottom $\Gamma^-$, we assume  that either
\begin{itemize}
 \item $\Gamma^-=\emptyset$ or 
 \item $\Gamma^-=\{(x, \underline b^-(x)): x\in \Rr^d\}$ where $\underline b^-\in \dot W^{1, \infty}(\Rr)$  satisfies
 \bq\label{assume:dist}
 \dist(\Sigma, \Gamma^-)>h>0.
 \eq
 \end{itemize}
 In either case, $ \dist(\Sigma, \Gamma^-)>h>0$. Consider the elliptic problem
\bq\label{eq:elliptic}
\begin{cases}
\Delta_{x, y} \phi=0\quad\text{in}~\Omega^-,\\
\phi=f\quad\text{on}~\Sigma,\\
\frac{\p\phi}{\p \nu^-}=0\quad\text{on}~\Gamma^-,
\end{cases}
\eq
where in the case of infinite depth ($\Gamma^-=\emptyset$), the Neumann condition is replaced by the vanishing of $\na_{x, y} \phi$ as $y\to -\infty$
\bq\label{vanish:dyphi}
\lim_{y\to -\infty} \na_{x, y}\phi=0.
\eq
The Dirichlet-Neuman operator associated to $\Omega^-$ is formally defined by
\bq\label{def:DN}
G^-(\eta) f=\sqrt{1+|\nabla \eta|^2}\frac{\p\phi}{\p n},
\eq
where we recall that $n$ is the upward-pointing unit normal to $\Sigma$. Similarly, if $\phi$ solves the elliptic problem \eqref{eq:elliptic} with $(\Omega^-, \Gamma^-, \nu^-)$ replaced by $(\Omega^+, \Gamma^+, \nu^+)$ then we define 
\begin{equation*}
G^+(\eta) f=\sqrt{1+|\nabla \eta|^2}\frac{\p\phi}{\p n}.
\end{equation*}
Note that $n$ is inward-pointing for $\Omega^+$. In the rest of this section, we only state results for $G^-(\eta)$ since corresponding results for  $G^+(\eta)$ are completely parallel.

The Dirichlet data $f$ for \eqref{eq:elliptic} will be taken  in the following ``screened'' fractional Sobolev space (\cite{LeoTice})
\bq\label{def:wtHmez}
\wt H^\mez_{\Theta}(\Rr^d)=\left\{ f\in  \mathcal{S}'(\Rr^d)\cap  L^2_{loc}(\Rr^d): \int_{\Rr^d}\int_{B_{\Rr^d}(0, \Theta(x))}\frac{|f(x+k)-f(x)|^2}{|k|^{d+1}}dkdx<\infty\right\}/ \Rr,
\eq
where $\Theta:\Rr^d\to (0, \infty]$ is a given lower semi-continuous function.  We will choose for an arbitrary number $a\in (0, 1)$ that
\bq\label{def:frakd}
\Theta(x)=
\begin{cases}
 \infty\quad\text{when}\quad \Gamma^-=\emptyset,\\
\frak{d}(x):=a(\eta(x)-\underline b^-(x))\quad\text{when}\quad \underline b^-\in \dot W^{1, \infty}(\Rr^d).
\end{cases}
\eq
In view of assumption \eqref{assume:dist}, 
\bq\label{lower:frakd}
\mathfrak{d} \ge ah.
\eq
We also define the slightly-homogeneous Sobolev spaces
\bq\label{def:dotHs}
 H^{1,\sigma}(\Rr^d)=\{f\in \mathcal{S}'(\Rr^d)\cap L^2_{loc}(\Rr^d): \na f\in H^{\sigma-1}(\Rr^d)\}/~\Rr.
\eq
\begin{rema}
 According to Theorem 2.2 b) in \cite{Stri}, $f\in \wt H^\mez_1(\Rr^d)$ ($\Theta\equiv 1$) if and only if $f\in \mathcal{S}'(\Rr^d)\cap L^2_{loc}(\Rr^d)$ and $\hat f$ is locally $L^2$ in the complement of the origin such that
\bq\label{Fourier:mezH1}
\int_{\Rr^d}\min(|\xi|, |\xi|^2) |\hat f(\xi)|^2d\xi<\infty;
\eq
moreover, $\| f\|_{\wt H^\mez_1(\Rr^d)}^2$ is bounded above and below by a multiple of \eqref{Fourier:mezH1} so that
\bq
\wt H^\mez_1(\Rr^d)=\dot H^{1,\frac{1}{2}}(\Rr^d).
\eq
On the other hand,  $f\in \wt H^\mez_\infty(\Rr^d)$ ($\Theta\equiv \infty$) if and only if $f\in L^2_{loc}(\Rr^d)$ and $\hat f$ is locally $L^2$ in the complement of the origin, with 
\bq\label{Fourier:mezHinfty}
\int_{\Rr^d}|\xi||\hat f(\xi)|^2d\xi<\infty;
\eq
moreover, $\| f\|_{\wt H^\mez_\infty(\Rr^d)}^2$ is a constant multiple of \eqref{Fourier:mezHinfty}. Thus, we have the continuous embeddings
\bq\label{wtH:dotH}
\dot{H}^\frac{1}{2}(\mathbb{R}^d)= \wt H^\mez_\infty(\Rr^d)\subset \wt H^\mez_{\mathfrak{d}}(\Rr^d)\subset \wt H^\mez_1(\Rr^d) = H^{1,\frac{1}{2}}(\Rr^d)
 \eq
 upon recalling the lower bound \eqref{lower:frakd} for $\mathfrak{d}$. In addition, under condition \eqref{assume:dist},
 \bq\label{iden:boundedb}
 \text{if}\quad \eta,~\underline b^-\in W^{1, \infty}(\Rr^d)\quad\text{then}\quad  \wt H^\mez_{\mathfrak{d}}(\Rr^d)=\wt H^\mez_1(\Rr^d).
 \eq
See Theorem 3.13 \cite{LeoTice}. To accommodate unbounded bottoms, we have only assumed that $\underline{b}^-\in \dot W^{1, \infty}(\Rr^d)$ and thus \eqref{iden:boundedb} is not applicable. Nevertheless, we have the following proposition. 
\end{rema}
\begin{prop}\label{prop:wtH}
Assume that $\sigma_1,~\sigma_2:\Rr^d\to (0, \infty]$ satisfy 
\begin{align}\label{wt:cd1}
&\inf_{x\in \Rr^d}\{\sigma_1(x), \sigma_2(x)\}>h>0,\quad\| \sigma_1-\sigma_2\|_{L^\infty(\Rr^d)}\le M<\infty.
\end{align}
Then  there exists $C=C(d, h, M)$ such that
\bq\label{equi:wtH}
\| f\|_{\wt H^\mez_{\sigma_1}(\Rr^d)}\le C\| f\|_{\wt H^\mez_{\sigma_2}(\Rr^d)}\quad\forall f\in \wt H^\mez_{\sigma_2}(\Rr^d).
\eq 
\end{prop}
It follows  that for any two surfaces $\eta_1$ and $\eta_2$ in $L^\infty(\Rr^d)$  satisfying \eqref{assume:dist}, the screened Sobolev space $\wt H^\mez_{\frak{d}}(\Rr^d)$, $\frak{d}$ given by \eqref{def:frakd}, is independent of $\eta_j$. The proof of Proposition \ref{prop:wtH} is given in Appendix \ref{appendix:wtH}. 

We will solve \eqref{eq:elliptic} in the homogeneous Sobolev space $\dot H^1(\Omega^-)$ where 
\bq\label{def:dotH}
\dot H^1(U)=\{u\in L^2_{loc}(U): \na u\in L^2(U)\}/~\Rr
\eq
for $U\subset \Rr^N$ connected. Here, the norm of $\dot H^1(U)$ is given by $\| u\|_{\dot H^1(U)}=\| \na u\|_{L^2(U)}$. \begin{prop}\label{dotH1:complete}
The vector space $\dot H^1(U)$ equipped with the norm $\| u\|_{\dot H^1(U)}=\| \na u\|_{L^2(U)}$ is complete. 
\end{prop}
\begin{proof}
Suppose that $u_n$ is a Cauchy sequence in $\dot H^1(U)$. Then $\na u_n\to F$ in $L^2(U)$. We claim that $F=\na u$ for some $u\in L^2_{loc}(U)$. Indeed, for any  bounded domain $V\subset U$, the sequence $u_n-|V|^{-1}\int_V u_n$ is bounded in $L^2(V)$, according to the Poincar\'e inequality, hence weakly converges in $L^2(V)$.
By a diagonal process, we can find $u\in L^2_{loc}(U)$ and a subsequence $n_k\to \infty$ such that
\[
u_{n_k}-|V|^{-1}\int_V u_{n_k}\wc u\quad\text{in}~L^2(V)
\]
for any bounded $V\subset U$. Let $\varphi\in C^\infty_c(U)$ be a test vector field with $\supp \varphi \subset V\Subset U$. We have
\[
\begin{aligned}
\int_U \varphi\cdot \na u_n dx&=\int_V \varphi \na (u_n- |V|^{-1}\int_V u_ndy)dx\\
&=-\int_V  (u_n- |V|^{-1}\int_V u_ndy) \cnx \varphi dx\to -\int_V u\cnx\varphi dx.
\end{aligned}
\]
Thus,
\[
\int_U F\cdot \varphi dx=\lim_{n\to \infty}\int_U \na u_n\cdot \varphi dx=-\int_U u\cnx\varphi dx
\]
for any test vector field $\varphi$. This proves that $F=\na u$ and thus finishes the proof.
\end{proof}
 We refer to Appendix \ref{appendix:trace} of the present paper for a summary of trace theory, taken from \cite{LeoTice}, when $U$ is an infinite strip-like domain or a Lipschitz half space.
\begin{prop}\label{prop:vari1}
Consider the finite-depth case with $\underline b^-\in \dot W^{1, \infty}(\Rr^d)$. If $\eta\in \dot W^{1, \infty}(\Rr^d)$ then for every $f\in \wt H^\mez_{\frak{d}}(\Rr^d)$, there exists a unique variational solution $\phi \in \dot H^1(\Omega^-)$ to \eqref{eq:elliptic}. Moreover, $\phi$ satisfies 
\bq\label{variest:phi}
\|\na_{x, y}\phi\|_{L^2(\Omega^-)}\le \cF(\| \na\eta\|_{L^\infty})\| f\|_{\wt H^\mez_\frak{d}(\Rr^d)}
\eq
for some $\cF:\Rr^+\to \Rr^+$ depending only on $h$ and $\| \na\underline b^-\|_{L^\infty(\Rr^d)}$.
\end{prop}
\begin{proof}
By virtue of Theorem \ref{theo:lift}, there exists $\underline f \in \dot H^1(\Omega^-)$ such that $\text{Tr} (\underline f)(x, \eta(x))=f(x)$, $\text{Tr} (\underline f)(x, \underline b^-(x))=f(x)$, and 
\bq\label{bound:psi1}
\| \na_{x, y} \underline f\|_{L^2(\Omega^-)}\le \cF(\| \na\eta\|_{L^\infty})\| f\|_{\wt H^\mez_\mathfrak{d}(\Rr^d)}
\eq
where $\cF$ depends only on $h$ and $\| \na\underline b^-\|_{L^\infty(\Rr^d)}$. Set
\[
H^{1}_{0,\ast}(\Omega^-)=\{ u\in \dot H^1(\Omega^-):~\text{Tr}(u)\vert_\Sigma=0\}
\]
endowed with the norm of $\dot H^1(\Omega^-)$. We then define $\phi$ solution to \eqref{eq:elliptic} to be 
\bq\label{def:variphi}
\phi(x, y)=\underline f(x, y)+u(x, y)
\eq
where $u\in H^{1}_{0,\ast}(\Omega^-)$ is the unique  solution to the variational problem
\bq\label{variform:u}
\int_{\Omega^-}\na_{x, y} u\cdot \na_{x, y} \varphi dxdy=-\int_{\Omega^-}\na_{x, y} \underline f\cdot \na_{x, y} \varphi dxdy\quad\forall \varphi \in  H^{1}_{0,\ast}(\Omega^-).
\eq 
The existence and uniqueness of $u$ is guaranteed by the Lax-Milgram theorem upon using the bound \eqref{bound:psi1}. Setting $\varphi=u$ in \eqref{variform:u} and recalling the definition \eqref{def:variphi} of $\phi$ we obtain the estimate \eqref{variest:phi}. It follows from \eqref{variform:u} that 
\[
\int_{\Omega^-}\na_{x, y} \phi\cdot \na_{x, y} \varphi dxdy=0\quad\forall \varphi \in  H^{1}_{0,\ast}(\Omega^-).
\]
Thus, if $\phi$ is smooth then $\phi$ solves \eqref{eq:elliptic} in the classical sense upon integrating by parts. Finally, it is easy to see that the solution $\phi$ constructed by \eqref{def:variphi} and \eqref{variform:u} is independent of the choice of $\underline f\in \dot H^1(U)$ that has  trace $f$ on $\Sigma$.
\end{proof}
\begin{rema}
As the functions $\underline b^\pm$ are fixed, we shall omit the dependence on $\| \na\underline b^\pm\|_{L^\infty(\Rr^d)}$.
\end{rema}
\begin{prop}\label{prop:vari2}
Consider the infinite-depth case $\Gamma^-=\emptyset$. If $\eta\in \dot W^{1, \infty}(\Rr^d)$ then for every $f\in \wt H^\mez_\infty(\Rr^d)$ there exists a unique variational solution $\phi \in \dot H^1(\Omega^-)$ to \eqref{eq:elliptic}. Moreover, $\phi$ satisfies 
\bq\label{variest:phi2}
\|\na_{x, y}\phi\|_{L^2(\Omega^-)}\le \cF(\| \na\eta\|_{L^\infty})\| f\|_{\wt H^\mez_\infty(\Rr^d)}
\eq
for some $\cF:\Rr^+\to \Rr^+$ depending only on $h$.
\end{prop}
\begin{proof}
The proof follows along the same lines as in the proof of Proposition \ref{prop:vari1} upon using the trace Theorem \ref{theo:trace2} and the lifting Theorem \ref{theo:lift2} for the half space $U=\Omega^-$. The fact that $\na \phi\in L^2(\Omega^-)$ gives a sense to the boundary condition \eqref{vanish:dyphi}. 
\end{proof}
\begin{nota}
We denote 
\bq\label{wtH-}
\wt H^\mez_{-}=
\begin{cases}
\wt H^\mez_\infty(\Rr^d)~\quad \text{if}~\Gamma^-=\emptyset,\\
\wt H^\mez_{\mathfrak{d}}(\Rr^d)~\quad \text{if}~\underline b^-\in \dot W^{1, \infty}(\Rr^d),\\
\end{cases}
\eq
and 
\bq\label{wtH+}
\wt H^\mez_{+}=
\begin{cases}
\wt H^\mez_\infty(\Rr^d)~\quad \text{if}~\Gamma^+=\emptyset,\\
\wt H^\mez_{\mathfrak{d}'}(\Rr^d)~\quad \text{if}~ \underline b^+\in \dot W^{1, \infty}(\Rr^d),\quad \mathfrak{d}'(x):=a(\eta(x)-\underline b^+(x)),~a\in (0, 1).
\end{cases}
\eq
 For $s>\mez$, we denote
\begin{equation}\label{def:wtHs}
\widetilde{H}^s_\pm=\widetilde{H}^{\frac{1}{2}}_\pm\cap H^{1,s}(\Rr^d).
\end{equation}
\end{nota}
\begin{prop}\label{prop:lowG}
If $\eta\in \dot W^{1, \infty}(\Rr^d)$ then the Dirichlet-Neumann operator is continuous from $\wt H^\mez_-$ to $H^{-\mez}(\Rr^d)$. Moreover,  there exists a constant $C>0$ depending only on $h$ such that 
\bq\label{G:H-mez}
\| G^-(\eta)f\|_{H^{-\mez}}\le \cF(\| \na\eta\|_{L^\infty})\| f\|_{\wt H^\mez_{-}}.
\eq
\end{prop}
\begin{proof}
Let $\phi$ solve \eqref{eq:elliptic}. By virtue of  Propositions \ref{prop:vari1} and \ref{prop:vari2}, we have $\phi \in \dot H^1(\Omega^{-})$ and $\Delta \phi=0$. According to Theorem \ref{theo:normaltrace}, the trace 
\[
\sqrt{1+|\na \eta|^2}\frac{\p \phi}{\p n}\vert_\Sigma=G^-(\eta)f
\]
is well-defined in $H^{-\mez}(\Rr^d)$ and 
\[
\begin{aligned}
\| G^-(\eta)f\|_{H^{-\mez}}&\le Ch^{-\mez}(1+\| \na \eta\|_{L^\infty})\| \na \phi\|_{L^2(\Omega^-_h)}
\end{aligned}
\]
where $C$ is an absolute constant and $\Omega^-_h=\{(x, y)\in \Rr^{d+1}:~\eta(x)-h<y<\eta(x)\}$. Thus, \eqref{G:H-mez} follows from  \eqref{variest:phi} and \eqref{variest:phi2}.
\end{proof}

To propagate higher Sobolev regularity for $\phi$ and hence for $G^-(\eta)f$,  following \cite{LannesJAMS, ABZ3} we straighten the boundary as follows. Set
 \begin{equation}\label{lesomega}
\begin{cases}
\Omega^{-}_1 = \{(x,y): x \in \Rr^d, \eta(x)-h<y<\eta(x)\},\\
\Omega^{-}_2 = \{(x,y)\in \Omega^-: y\leq\eta(x)-h\}
\end{cases}
 \end{equation}
and 
\begin{equation}\label{omega1}
\left\{
\begin{aligned}
&\widetilde{\Omega^{-}_1}=\Rr^d\times (-1, 0),\\
&\widetilde{\Omega^{-}_2} = \{(x,z)\in \Rr^d \times (-\infty, -1]: (x, z+1+\eta(x)-h)\in \Omega^-_2\},\\
&\widetilde{\Omega^{-}}= \widetilde{\Omega^{-}_1}  \cup \widetilde{\Omega^{-}_2}.  
\end{aligned}
\right.
\end{equation}
Define 
 \begin{equation}\label{diffeo}
\left\{
\begin{aligned}
&\varrho(x,z)=  (1+z)e^{\tau z\langle D_x \rangle }\eta(x) -z \big\{e^{-(1+ z)\tau\langle D_x \rangle }\eta(x) -h\big\}\quad \text{if } (x,z)\in \widetilde{\Omega^{-}_1},\\
&\varrho(x,z) = z+1+\eta(x)-h\quad \text{if } (x,z)\in \widetilde{\Omega^{-}_2},
\end{aligned}
\right.
\end{equation}
where $\tau>0$ will be chosen in the next lemma. 
\begin{lemm}\label{lemm:diffeo} Assume $\eta\in B^1_{\infty, 1}(\Rr^d)$.

1) There exists a constant $C>0$ independent of $\tau$ such that 
\[
\| \na_{x,z}\varrho\|_{L^\infty(\wt{\Omega^-})}\le 1+C\| \eta\|_{B^1_{\infty, 1}}
\]
for all $(x, z)\in \wt{\Omega^-}$.

2) There exists $K>0$ such that if 
\begin{equation}\label{CondTau}
\tau \| \eta\|_{B^1_{\infty, 1}}\le \frac{h}{2K}
\end{equation}
then $ \min(1, \frac{h}{2})\le \p_z\varrho\le K\| \eta\|_{B^1_{\infty, 1}}$ and thus the mappings  $(x, z)\in \wt{\Omega^-_k}\mapsto (x, \varrho(x, z))\in\Omega^-_k$, $k=1, 2$ are  Lipschitz diffeomorphisms.
 \end{lemm}
 Lemma \ref {lemm:diffeo}  follows from straightforward calculations which we omit. Note that $H^s(\Rr^d)\subset B^1_{\infty, 1}(\Rr^d)$ for any $s>1+\frac d2$. A direct calculation shows that if $f:\Omega^-\to \Rr$ then $\wt f(x, z)=f(x, \varrho(x, z))$  satisfies 
\bq\label{div:eq}
\cnx_{x,z}(\mathcal{A}\na_{x,z}\wt f)(x, z)=\p_z\varrho(\Delta_{x, y}f)(x, \varrho(x, z))
\eq
with 
\bq\label{def:matrixA}
\mathcal{A}=
\begin{bmatrix}
\p_z\varrho\cdot Id & -\na_x\varrho\\
-(\na_x\varrho)^T & \frac{1+|\na_x\varrho|^2}{\p_z\varrho}
\end{bmatrix}.
\eq
In order to study functions inside the domain, we introduce adapted functional spaces. Given~$\mu\in\Rr$ we define the interpolation spaces
\begin{equation}\label{XY}\begin{aligned}
X^\mu(I)&=C^0_z(I;H^{\mu}({\mathbf{R}}^d))\cap L^2_z(I;H^{\mu+\mez}({\mathbf{R}}^d)),\\
Y^\mu(I)&=L^1_z(I;H^{\mu}({\mathbf{R}}^d))+L^2_z(I;H^{\mu-\mez}({\mathbf{R}}^d)).
\end{aligned}
\end{equation}
We prove the following useful inequalities. 
\begin{lemm}\label{lemm:XY}
Let $s_0$, $s_1$ and $s_2$ be real numbers, and let $J\subset \Rr$.\\
1) If 
\bq\label{cd:XY1}
\begin{cases}
s_0\le \min\{s_1+1, s_2+1\},\\
s_1+s_2>s_0+\frac{d}{2}-1,\\
s_1+s_2+1>0,
\end{cases}
\eq
then 
\bq\label{XY1}
\| u_1u_2\|_{Y^{s_0}(J)}\les \| u_1\|_{X^{s_1}(J)}\| u_2\|_{X^{s_2}(J)}.
\eq
2) If
\bq\label{cd:XY2}
\begin{cases}
s_0\le \min\{s_1, s_2\},\\
s_1+s_2>s_0+\frac{d}{2},\\
s_1+s_2>0,
\end{cases}
\eq
then 
\bq\label{XY2}
\| u_1u_2\|_{Y^{s_0}(J)}\les \| u_1\|_{Y^{s_1}(J)}\| u_2\|_{X^{s_2}(J)}.
\eq
In fact, we have
\begin{align}\label{est:XY22}
&\| T_{u_2}u_1\|_{Y^{s_0}(J)}\les \| u_1\|_{Y^{s_1}(J)}\| u_2\|_{L^\infty(J; H^{s_2})}\quad\text{if}~s_0\le s_1,~s_1+s_2>s_0+\frac{d}{2},\\
\label{est:XY3}
&\| T_{u_1}u_2\|_{Y^{s_0}(J)}\les \| u_1\|_{Y^{s_1}(J)}\| u_2\|_{L^\infty(J; H^{s_2})}\quad\text{if}~s_0\le s_2,~s_1+s_2>s_0+\frac{d}{2},\\
\label{est:XY4}
&\| R(u_1, u_2)\|_{L^1(J; H^{s_0})}\les \| u_1\|_{Y^{s_1}(J)}\| u_2\|_{X^{s_2}(J)}\quad\text{if}~s_1+s_2>0,~s_1+s_2>s_0+\frac{d}{2}.
\end{align}
\end{lemm} 
The proof of Lemma \ref{lemm:XY} is given in Appendix \ref{appendix:XY}. 
\begin{lemm}\label{lemm:tracewt}
Let $\eta\in H^s(\Rr^d)$ with $s>1+\frac d2$ and $\tau>0$ such that \eqref{CondTau} holds. If 
\begin{align*}
&\na_{x, z} \wt f \in L^2([-1, 0]; L^2(\Rr^d)),\\
&\cnx_{x, z}(\mathcal{A}\na_{x, z}\wt f)\in L^2([-1, 0]; H^{-1}(\Rr^d))
\end{align*}
then $\na_{x, z}\wt f\in X^{-\mez}([-1, 0])$ and 
\bq\label{trace:wt}
\| \na_{x, z}\wt f\|_{X^{-\mez}([-1, 0])}\le \cF(\| \eta\|_{H^s})(\|\na_{x, z} \wt f\|_{L^2([-1, 0]; L^2)}+\|\cnx_{x, z}(\mathcal{A}\na_{x, z}\wt f)\|_{L^2([-1, 0]; H^{-1})}).
\eq
\end{lemm}
\begin{proof}
By the definition of $X^{-\mez}([-1, 0])$, it suffices to prove that $\na_{x, z}\wt f\in C([-1, 0]; H^{-\mez}(\Rr^d))$ with norm bounded by the right-hand side of \eqref{trace:wt}. By virtue of  the  interpolation Theorem \ref{trace:Lions}, $\na_x{\wt f}\in C([-1, 0]; H^{-\mez}(\Rr^d))$ and
\bq\label{trace:dxwtf}
\begin{aligned}
\| \na_x {\wt f}\|_{C([-1, 0]; H^{-\mez})}&\les \| \na_x {\wt f}\|_{L^2([-1, 0]; L^2)}+\| \p_z\na_x {\wt f}\|_{L^2([-1, 0]; H^{-1})}\les \| \na_{x, z} {\wt f}\|_{L^2([-1, 0]; L^2)}.
\end{aligned}
\eq
Thus, it remains to prove that $\p_z\wt f\in C([-1, 0]; H^{-\mez}(\Rr^d))$. Setting $\Xi(x, z)=-\na_x\varrho\cdot \na_x \wt f+\frac{1+|\na_x\varrho|^2}{\p_z\varrho}\p_z\wt f$ we find that $\p_z\Xi$ is a divergence
\[
\p_z\Xi=-\cnx_x(\p_z\varrho\na_x\wt f-\na_x\varrho\p_z\wt f)+\cnx_{x, z}(\mathcal{A}\na_{x, z}\wt f).
\]
Consequently,
\[
\| \p_z\Xi\|_{L^2([-1, 0]; H^{-1})}\le \cF(\| \eta\|_{H^s})(\|\na_{x, z} \wt f\|_{L^2([-1, 0]; L^2)}+\|\cnx_{x, z}(\mathcal{A}\na_{x, z}\wt f)\|_{L^2([-1, 0]; H^{-1})}).
\]
On the other hand, using Lemma \ref{lemm:diffeo}, it is easy to see that 
\[
\| \Xi\|_{L^2([-1, 0]; L^2)}\le \cF(\| \eta\|_{H^s})\|\na_{x, z} \wt f\|_{L^2([-1, 0]; L^2)}.
\]
Then, applying Theorem \ref{trace:Lions} we obtain that $\Xi\in C([-1, 0]; H^{-\mez}(\Rr^d))$ and 
\[
\| \Xi\|_{C([-1, 0]; H^{-\mez})}\le \cF(\| \eta\|_{H^s})(\|\na_{x, z} \wt f\|_{L^2([-1, 0]; L^2)}+\|\cnx_{x, z}(\mathcal{A}\na_{x, z}\wt f)\|_{L^2([-1, 0]; H^{-1})}).
\]
Now from the definition of $\Xi$ we have
\[
\p_z\wt f=\frac{\p_z\varrho}{1+|\na_x\varrho|^2}\big(\Xi(x, z)+\na_x\varrho\cdot \na_x \wt f\big).
\]
For $s>1+\frac d2\ge \tdm$, using the product rule \eqref{pr} and the nonlinear estimate \eqref{est:F(u):S} gives
\[
\begin{aligned}
\| \frac{\p_z\varrho}{1+|\na_x\varrho|^2}\Xi\|_{C([-1, 0]; H^{-\mez})}&\les \big(\| \frac{\p_z\varrho}{1+|\na_x\varrho|^2}-h\|_{C([-1, 0]; H^{s-1})}+h\big)\| \Xi\|_{C([-1, 0]; H^{-\mez})}\\
&\le \cF(\| \eta\|_{H^s})(\|\na_{x, z} \wt f\|_{L^2([-1, 0]; L^2)}+\|\cnx_{x, z}(\mathcal{A}\na_{x, z}\wt f)\|_{L^2([-1, 0]; H^{-1})})
\end{aligned}
\]
and
\[
\begin{aligned}
\| \frac{\p_z\varrho\na_x\varrho}{1+|\na_x\varrho|^2}\cdot \na_x \wt f\|_{C([-1, 0]; H^{-\mez})}&\les \| \frac{\p_z\varrho\na_x\varrho}{1+|\na_x\varrho|^2}\|_{C([-1, 0]; H^{s-1})}\| \na_x\wt f\|_{C([-1, 0]; H^{-\mez})}\\
&\le \cF(\| \eta\|_{H^s})\|\na_{x, z} \wt f\|_{L^2([-1, 0]; L^2)},
\end{aligned}
\]
where \eqref{trace:dxwtf} was used in the last estimate. This finishes the proof.  
\end{proof}
Denote $v(x, z)=\phi(x, \varrho(x, z)):\Rr^d\times [-1, 0]\to \Rr$ where $\phi$ is the solution of \eqref{eq:elliptic}. Then $v$ satisfies  $v\vert_{z=0}=f$ and
\bq\label{div:eqv}
\cnx_{x,z}(\mathcal{A}\na_{x,z}v)=0,
\eq
while, by the chain rule,
\begin{equation}\label{ExpressionDNNew}
G^{-}(\eta)f=\left(\frac{1+|\na_x\varrho|^2}{\p_z\varrho}\p_zv-\na_x\varrho\cdot \na_x v\right)_{|z=0}.
\end{equation}

Expanding \eqref{div:eqv} yields
\bq\label{eq:v}
(\p_z^2+\alpha\Delta_x+\beta\cdot\nabla_x\p_z-{\gamma}\p_z)v=0\quad\text{in}~\Rr^d\times [-1, 0],
\eq
 where 
\begin{equation}\label{alpha}
\alpha= \frac{(\partial_z\rho)^2}{1+|\nabla_x   \rho |^2},\quad 
\beta=  -2 \frac{\partial_z \rho \nabla_x \rho}{1+|\nabla_x  \rho |^2} ,\quad 
{\gamma}= \frac{1}{\partial_z\rho}\bigl(  \partial_z^2 \rho 
+\alpha\Delta_x \rho + \beta \cdot \nabla_x \partial_z \rho\bigr).
\end{equation}
 Note that the restriction to $z\in [-1, 0]$ guarantees that $\varrho$ is smooth in $z$. We have the following Sobolev estimates for the inhomogeneous version of \eqref{eq:v}:
\begin{prop}[\protect{\cite[Proposition 3.16]{ABZ3}}]\label{prop:elliptic}
Let~$d\ge 1$, $s> 1 + \frac d 2$ and $\frac 1 2\le \sigma \leq s$.  Consider $f\in H^{1, \sigma}({\mathbf{R}}^d)$ and $\eta\in H^s(\Rr^d)$ satisfying $\dist(\eta, \Gamma^-)\ge h>0$. Assume that $F_0\in Y^{\sigma-1}([z_1, 0])$, $z_1\in (-1, 0)$ and~$v$ a solution of  
\bq\label{eq:vi}
(\p_z^2+\alpha\Delta_x+\beta\cdot\nabla_x\p_z-{\gamma}\p_z)v=F_0\quad\text{in}~\Rr^d\times [-1, 0]
\eq
with $v\vert_{z=0}=f$. If $z_0\in (z_1, 0)$ and
\bq\label{elliptic:lowcd}
\nabla_{x,z} v\in X^{-\mez}([z_0, 0])
\eq
 then  $\nabla_{x,z}v \in X^{\sigma-1}([z_0, 0])$ and
$$
\lA \nabla_{x,z} v\rA_{X^{\sigma-1}([z_0,0])}
\le \mathcal{F}(\| \eta \|_{H^s})\Big(\lA  \nabla_xf\rA_{H^{\sigma-1}}+\lA F_0\rA_{Y^{\sigma-1}([z_1,0])}
+ \lA \nabla_{x,z} v\rA_{X^{-\mez}([z_0, 0])} \Big)
$$
for some $\mathcal{F}:\Rr^+\rightarrow\Rr^+$ depending only on~$(\sigma, s, h, z_0, z_1)$. 
\end{prop}
\begin{rema}
In fact, Proposition 3.16 in \cite{ABZ3} assumes $f\in H^\sigma(\Rr^d)$. This comes from estimating $ v$ solving 
\[
\p_zv+T_{A}v=- w\quad z\in [z_0, 0],\quad v\vert_{z=0}=f
\]
where $w \in X^\sigma([z_0, 0])$ and $A\in \Gamma^1_\eps$, $\eps\in (0, \max\{\mez, s-1-\frac d2\})$, is given by \eqref{aA} below. To obtain estimates involving only $\| \na_xf\|_{H^{\sigma-1}}$, it suffices to differentiate this equation in $x$ and apply Proposition \ref{prop:negOp} to control $T_{\na_xA}v$.
\end{rema}
In the rest of this subsection, we fix $s>1+\frac d2$. For $v(x, z)=\phi(x, \varrho(x, z))$, $\phi$  solution of \eqref{eq:elliptic},   Lemma \ref{lemm:tracewt} combined with \eqref{div:eqv} yields
\[
\| \nabla_{x,z} v\|_{X^{-\mez}([-1,0])}\le \cF(\| \eta\|_{H^s})\| \na_{x,z}v\|_{L^2([-1, 0]; L^2(\Rr^d))} \le \cF(\| \eta\|_{H^s})\| \na_{x, y}\phi\|_{L^2(\Omega^-)}
\]
In conjunction with \eqref{variest:phi} and \eqref{variest:phi2}, this implies 
\bq\label{nav:X-mez}
\| \nabla_{x,z} v\|_{X^{-\mez}([-1,0])}\le
\cF(\| \eta\|_{H^s})\| f\|_{\wt H^\mez_{-}}.
\eq
This verifies condition \eqref{elliptic:lowcd} of Proposition \ref{prop:elliptic} from which the estimate for $\| \na_{x,z}v\|_{X^{\sigma-1}}$, $\sigma\in [\mez, s]$, follows. Using this and the product rule \eqref{pr} one can easily  deduce the continuity of $G^{-}(\eta)$ in higher Sobolev norms:
\begin{theo}[\protect{\cite[Theorem 3.12]{ABZ3}}]\label{theo:estDN}
Let~$d\ge 1$, $s>1+\frac{d}{2}$ and $\frac 1 2 \le \sigma \leq s$. Consider  $f\in \wt H^\sigma_{-}$  and  $\eta\in H^s(\Rr^d)$ with $\dist(\eta, {\Gamma^{-}})\ge h>0$. Then we have 
$G^{-}(\eta)f\in H^{\sigma-1}(\Rr^d)$, together with the estimate
\begin{equation}\label{est:DN}
\| G^{-}(\eta)f \|_{H^{\sigma-1}}\le 
\mathcal{F}\big(\| \eta \|_{H^s}\big)\Vert f\Vert_{\widetilde{H}^\sigma_-}
\end{equation}
for some $\cF:\Rr^+\to \Rr^+$  depending only on $(s, \sigma, h)$.
\end{theo}
\begin{rema}
Theorem \ref{theo:estDN} was proved in \cite{ABZ3} for $f\in H^\sigma(\Rr^d)$. The two-phase Muskat problem involves $G^-(\eta)f^-$ where $f^-$ is obtained from system \eqref{system:fpm}. In particular, $f^-$ is only determined up to an additive constant.  
\end{rema}
\subsection{Paralinearization with tame error estimate}\label{section:paralin}
The principal symbol of the Dirichlet-Neumann operator is given by 
\bq\label{ld}
\lambda(x, \xi)=\sqrt{(1+|\nabla\eta(x)|^2)|\xi|^2-(\nabla\eta(x)\cdot \xi)^2}.
\eq
Note that when $d=1$, \eqref{ld} reduces to $\lambda(x, \xi)=|\xi|$. 

We first recall a paralinearization result from \cite{ABZ3}.
\begin{theo}[\protect{\cite[Proposition 3.13]{ABZ3}}]\label{paralin:ABZ}
Let $r>1+\frac d2$ with $d\ge 1$, and let  $\delta\in (0, \mez]$ satisfy $\delta<r-1-\frac d2$. Let $\sigma\in [\mez, r-\delta]$.
If $\eta\in H^r(\Rr^d)$ and $f\in \wt H^\sigma_-$ with $\dist(\eta, \Gamma^-)>h>0$, then  we have
\begin{align}\label{def:R0}
&G^{-}(\eta)f=T_\lambda f+R_0^-(\eta)f,\\ \label{RDN:ABZ}
&\| R_0^-(\eta)f\|_{H^{\sigma-1+\delta}}\le \cF(\| \eta\|_{H^r})\Vert f\Vert_{\wt{H}^\sigma_-}
\end{align}
for some $\cF:\Rr^+\to \Rr^+$  depending only on $(r, \sigma, \delta, h)$.
\end{theo} 
\begin{rema}
In the statement of Proposition 3.1.3 in \cite{ABZ3}, $\sigma\in [\mez, r-\mez]$. However, its proof (see page 116) allows for $\sigma\in [\mez, r-\delta$].
\end{rema}
Our goal in this subsection is to prove the next theorem, which isolates the main term in the Dirichlet-Neumann operator as an operator and which will be the key ingredient for obtaining a priori estimates for the Muskat problem in any subcritical Sobolev regularity.
\begin{theo}\label{tame:RDN}
Let $s>1+\frac d2$ with $d\ge 1$, and let  $ \delta \in (0, \mez]$ satisfy $\delta<s-1-\frac d2$. For any $\sigma\in[\mez, s]$, if $\eta\in H^{s+\mez-\delta}(\Rr^d)$ and $f\in \widetilde{H}^\sigma_{-}$ satisfies $\dist(\eta, \Gamma^-)>h>0$ then 
\bq\label{ParaDN}
G^{-}(\eta)f=T_\lambda(f-T_B\eta)-T_V\cdot\nabla \eta+R^-(\eta)f
\eq
where
\bq\label{BV}
B=\frac{\nabla \eta\cdot \nabla f+G^{-}(\eta)f}{1+|\nabla \eta|^2},\quad V=\nabla f-B\nabla \eta
\eq
and the remainder $R^-(\eta)f$ satisfies 
\bq\label{est:RDN}
\| R^-(\eta)f\|_{H^{\sigma-\mez}}\le \cF(\| \eta\|_{H^s})\big(1+\| \eta\|_{H^{s+\mez-\delta}}\big)\Vert f\Vert_{\widetilde{H}^\sigma_-}.
\eq
for some $\cF:\Rr^+\to \Rr^+$  depending only on $(s, \sigma, \delta, h)$.  In fact, $B=(\p_y\phi)\vert_{y=\eta(x)}$ and $V=(\nabla_x\phi)\vert_{y=\eta(x)}$ where $\phi$ is the solution of \eqref{eq:elliptic}.
\end{theo} 
\begin{rema}
For $s>1+\frac d2$, we can apply Theorem \ref{paralin:ABZ} with $r=s+\mez$, $\sigma=s$ and $\delta=\mez$ (the maximal value allowed) to have
\bq\label{compare:ABZ}
\| R_0^-(\eta)f\|_{H^{s-\mez}}\le \cF(\| \eta\|_{H^{s+\mez}})\Vert f\Vert_{\widetilde{H}^s_-}.
\eq
Both \eqref{compare:ABZ} and \eqref{est:RDN} provide a gain of $\mez$ derivative for $f$. The improvement of \eqref{est:RDN} is in that 1) there is  a gain of $\delta$ derivative for $\eta$; 2) the highest  norm  $\| \eta\|_{H^{s+\mez-\delta}}$ of $\eta$ appears linearly. For the sake of a priori estimates, 1) gives room to choose the time of existence $T$ as a small parameter; 2) is required to gain $\mez$ derivative using the parabolicity when measured in $L^2$ in time.  
\end{rema}
 We fix $s>1+\frac d2$ in the rest of this subsection. Setting
\bq\label{def:Q}
\begin{aligned}
Q=\p_z^2+\alpha\Delta_x+\beta\cdot\nabla_x\p_z,
\end{aligned}
\eq
we can rewrite \eqref{eq:v} as
\begin{equation}\label{Neweq:v}
Qv=\frac{\partial_zv}{\partial_z\varrho}Q\varrho.
\end{equation}
The coefficients of $Q$ can easily be controlled using \eqref{alpha}, \eqref{diffeo} and Lemma \ref{lemm:diffeo}:
\bq\label{est:alpha}
\| \alpha-h^2\|_{X^{r-1}([-1, 0])}+\| \beta\|_{X^{r-1}([-1, 0])}+\| \gamma\|_{X^{r-2}([-1, 0])}\le \cF(\| \eta\|_{H^s})\|  \eta\|_{H^r}\quad\forall r\ge s
\eq
since it follows from \eqref{diffeo} that
\begin{equation}\label{EstimrhoNew}
\Vert (\na_x\varrho, \p_z\varrho-h)\Vert_{X^{\mu-1}([-1,0])}\le C \Vert \eta\Vert_{H^\mu}\quad\forall \mu \in \Rr.
\end{equation}
We start with a factorization of $Q$ by paradifferential operators and a remainder:
\begin{lemm}\label{prop:decompose}
 With the symbols
\begin{equation}\label{aA}
	\begin{aligned}
		a&=\mez\big(-i\beta\cdot\xi-\sqrt{4\alpha|\xi|^2-(\beta\cdot\xi)^2}\big),\quad A=\mez\big(-i\beta\cdot\xi+\sqrt{4\alpha|\xi|^2-(\beta\cdot\xi)^2}\big),
	\end{aligned}
\end{equation}
we define $R_Q$ by
\bq\label{decompose}
R_Qg=Qg-(\p_z-T_a)(\p_z-T_A)g.
\eq
Let $0<\delta\le 1$ satisfy $\delta<s-1-\frac d2$. If $\tt$ satisfies 
\bq\label{cd:tt-s}
\tt\le s-\delta,\quad\tt+s>1+\delta,
\eq
 then, for any $z_0\in (-1, 0)$ we have
\bq\label{est:RAa}
\|R_Qg\|_{L^2([z_0, 0]; H^{\theta})}\le \cF(\| \eta\|_{H^s})(1+\| \eta\|_{H^{s+\mez-\delta}})\|\na_{x, \textcolor{red}{z}}g\|_{X^{\theta}([z_0, 0])}.\\
\eq
On the other hand, if $\tt$ satisfies
\bq\label{cd:tt:2}
\tt\le s,\quad \tt+s>1+\delta,
\eq
\bq\label{est:RAa2}
\|R_Qg\|_{Y^\tt([z_0, 0])}\le \cF(\| \eta\|_{H^s})\|\na_{x,z}g\|_{X^{\theta-\delta}([z_0, 0])}. 
\eq
\end{lemm}
\begin{proof}
From \eqref{aA} we have that $a+A=-i\beta \cdot \xi$, $aA=-\alpha|\xi|^2$, and hence 
\[
\p_z^2+T_\alpha\Delta_x+T_\beta\cdot\nabla_x\p_z=(\p_z-T_a)(\p_z-T_A)g-(T_aT_A-T_{\alpha}\Delta_x)g+T_{\p_zA}g.
\]
It follows that
\bq\label{form:RQ}
R_Qg=(\alpha-T_\alpha)\Delta_xg+(\beta-T_\beta)\cdot\nabla_x\p_z g-(T_aT_A-T_{\alpha}\Delta_x)g+T_{\p_zA}g,
\eq
{\it Proof of \eqref{est:RAa}.}  Assuming \eqref{cd:tt-s}, we claim that
\bq\label{EquivPQ}
\| (\alpha-T_\alpha)\Delta_xg+(\beta-T_\beta)\cdot\nabla_x\p_z g\|_{L^2([z_0, 0]; H^{\theta})}\le \cF(\| \eta\|_{H^s})\big(1+\|\eta \|_{H^{s+\mez-\delta}}\big)\| \na_{x,z}g\|_{L^\infty([z_0, 0]; H^{\theta})}.
\eq
Using \eqref{paralin:product}, \eqref{est:alpha} and \eqref{cd:tt-s}, we have
\begin{equation*}\label{paralin:beta}
\begin{split}
\| (T_\beta-\beta)\nabla_x\p_z g\|_{L^2H^{\theta}}&\les \| \beta\|_{L^2 H^{s-\delta}}\| \nabla_x\p_zg\|_{L^\infty H^{\theta-1}}\les\| \beta\|_{L^2 H^{s-\delta}}\| \p_zg\|_{L^\infty H^{\theta}}\\
&\le \cF(\| \eta\|_{H^s}) \big(1+\|\eta \|_{H^{s+\mez-\delta}}\big)\| \p_zg\|_{L^\infty H^{\theta}}.
\end{split}
\end{equation*}
 As for $(\alpha-T_\alpha)\Delta_xg$ we write
\[
(\alpha-T_\alpha)\Delta_xg=(\alpha-h^2-T_{\alpha-h^2})\Delta_xg+h^2(\text{Id}-T_1)\Delta_xg.
\]
The first term can be estimated as above and in view of  \eqref{eq.para}, $\text{Id}-T_1=\text{Id}-\Psi(D)$ is a smoothing operator so that
\[
\| (\text{Id}-T_1)\Delta_xg\|_{L^2H^{\theta}}\le C\| \nabla_xg\|_{L^2 H^{\theta}}.
\]
We thus obtain \eqref{EquivPQ}.

Next it is readily seen that $A$ and $a$ satisfy (see \eqref{defi:norms})
\begin{align}
\label{norm:aA:0}
&{M}^1_\delta(a; [-1, 0])+{M}^1_\delta(A; [-1, 0])\le \cF(\| \eta\|_{H^s})\\ \label{norm:aA:1}
&{M}^1_{\mez}(a; [-1, 0])+{M}^1_{\mez}(A; [-1, 0])\le \cF(\| \eta\|_{H^s})(1+\| \eta\|_{H^{s+\mez-\delta}})
\end{align}
Consequently, by Theorem \ref{theo:sc} (ii), $T_aT_A-T_{\alpha}\Delta_x$ is of order $\tdm$ and 
\bq\label{Estim1}
\begin{aligned}
\| (T_aT_A-T_{\alpha}\Delta_x)g\|_{L^2 H^{\theta}}&\le \cF(\| \eta\|_{H^s})(1+\| \eta\|_{H^{s+\mez-\delta}})\|  \nabla_x g\|_{L^2 H^{\theta+\frac{1}{2}}},
\end{aligned}
\eq
where Remark \ref{rema:low} has been used.  Now in view of the seminorm bounds
\[
\| M_0^1(\p_zA)\|_{L^2([-1, 0])}\le  \cF(\| \eta\|_{H^s})(1+\| \eta\|_{H^{s+\mez-\delta}}),
\]
Theorem \ref{theo:sc} (i) combined with Remark \ref{rema:low}  gives
\bq\label{Estim2}
\| T_{\p_zA}g\|_{L^2 H^{\theta}}\le \cF(\| \eta\|_{H^s})(1+\| \eta\|_{H^{s+\mez-\delta}})\| \nabla_x g\|_{L^\infty H^{\theta}}.
\eq
From \eqref{EquivPQ}, \eqref{Estim1} and \eqref{Estim2}, the proof of \eqref{est:RAa} is complete.

{\it Proof of \eqref{est:RAa2}}.  Assume \eqref{cd:tt:2}. Using \eqref{boundpara}, \eqref{Bony1} and \eqref{est:alpha}, we have
\begin{equation*}
\begin{aligned}
\| T_{\nabla_x\p_z g}\cdot \beta\|_{L^2H^{\theta-\mez}}&\les \| \nabla_x\p_zg\|_{L^\infty H^{\theta-1-\delta}}\| \beta\|_{L^2 H^{s-\mez}}\\
&\le \cF(\| \eta\|_{H^s}) \| \p_zg\|_{L^\infty H^{\theta-\delta}}
\end{aligned}
\end{equation*}
and 
\begin{align*}
\| R(\nabla_x\p_z g, \beta)\|_{L^1H^\tt}&\les \| \beta\|_{L^2 H^{s-\mez}}\| \nabla_x\p_zg\|_{L^2 H^{\theta-\mez-\delta}}\\
&\le \cF(\| \eta\|_{H^s}) \| \p_zg\|_{L^2 H^{\theta+\mez-\delta}}.
\end{align*}
 The term $(T_\alpha-\alpha)\Delta_xg$ can be treated similarly. Next using \eqref{norm:aA:0} and Theorem \ref{theo:sc} (ii) we find that  $T_aT_A-T_{\alpha}\Delta_x$ is of order $2-\delta$ and 
\bq\label{Estim3}
\begin{aligned}
\| (T_aT_A-T_{\alpha}\Delta_x)g\|_{L^2 H^{\theta-\mez}}&\le \cF(\| \eta\|_{H^s})\|  \nabla_x g\|_{L^2 H^{\theta+\mez-\delta}}.
\end{aligned}
\eq
As for $T_{\p_zA}g$, we note that since 
\[
\| (\alpha-h^2, \beta)\|_{L^\infty([-1, 0]; H^{s-1})}+\| (\p_z\alpha, \p_z\beta)\|_{L^\infty([-1, 0]; H^{s-2})}\le \cF(\| \eta\|_{H^s}),
\]
and $H^{s-2}(\Rr^d)\subset C_*^{-1+\delta}(\Rr^d)$, we have 
\[
M^1_{-1+\delta}(\p_zA)\le  \cF(\| \eta\|_{H^s}).
\]
By virtue of Proposition \ref{prop:negOp}, $T_{\p_zA}$ is of order $2-\delta$ and 
\[
\| T_{\p_zA}g\|_{L^2 H^{\tt-\mez}}\le  \cF(\| \eta\|_{H^s})\| \na_xg\|_{L^2 H^{\tt+\mez-\delta}}.
\]
This completes the proof of \eqref{est:RAa2}. 
\end{proof}
We can now start analyzing \eqref{Neweq:v}. We fix $\sigma\in [\mez, s]$ and  apply Proposition \ref{prop:elliptic} to have
\bq\label{estv:base}
\| \nabla_{x, z}v\|_{X^{\sigma-1}([z_0, 0])}\le \cF(\| \eta\|_{H^s})\Vert f\Vert_{\widetilde{H}^\sigma_-}
\eq
for all  $z_0\in (-1, 0)$. 

 Next we introduce  
\bq\label{def:goodunknown}
\mb=\frac{\p_zv}{\p_z\varrho},\qquad u=v-T_\mb\varrho.
\eq
We note that $\mb\vert_{z=0}=(\p_y\phi)(x, \eta(x))=B$ given by \eqref{BV}, and $u\vert_{z=0}=f-T_B\eta$. The new variable $u$ is known as the ``good unknown'' \`a la Alinhac. Fixing $\delta\in (0, \mez]$ satisfying $\delta<s-1-\frac{d}{2}$, for all $\sigma\in [\mez, s]$ we have 
\bq\label{sigma:mainpara}
\sigma+s>2+\delta.
\eq
\begin{lemm}\label{LemEstb}
Let $z_0\in (-1, 0)$. For  $\sigma \in [\mez, s]$ we have
\begin{align}\label{est:b1}
&\| \mb \|_{L^\infty ([z_0, 0]; H^{\sigma-1})}\le \cF(\| \eta\|_{H^s})\Vert f\Vert_{\widetilde{H}^\sigma_-},\\ \label{est:b2}
&\| \nabla_{x,z}\mb\|_{L^2([z_0, 0]; H^{\sigma-\tdm})}\le \cF(\| \eta\|_{H^s})\Vert f\Vert_{\widetilde{H}^\sigma_-}~\text{if}~\sigma+s>\frac 52,\\\label{est:b3}
&\| \nabla_{x,z}\mb\|_{L^\infty([z_0, 0]; H^{\sigma-2})}\le \cF(\| \eta\|_{H^s})\Vert f\Vert_{\widetilde{H}^\sigma_-}~\text{if}~\sigma+s>3,\\
\label{estb:Y}
&\| \p_zb\|_{Y^{\sigma-1}([z_0, 0])}\le \cF(\| \eta\|_{H^s})\Vert f\Vert_{\widetilde{H}^\sigma_-}.
\end{align}
\end{lemm}
\begin{proof}
\eqref{est:b1} follows from  \eqref{EstimrhoNew}, \eqref{estv:base} and the product rule \eqref{pr} with $s_0=s_1=\sigma-1$ and $s_2=s-1$ since $\sigma-1+s-1>0$ in view of \eqref{sigma:mainpara}. 

The estimate \eqref{est:b2} for $ \nabla_x\mb$ can be proved similarly upon using \eqref{pr} with $s_0=s_1=\sigma-\tdm$, $s_2=s-1$ and noting that  guarantees $s_1+s_2=\sigma+s-\frac52>0$.  As for $\p_z\mb$ we use \eqref{eq:v} to have the formula for $\p_z^2v$, then apply \eqref{pr} as in the estimate for $\na_x\mb$.

The proof of \eqref{est:b3} is similar to \eqref{est:b2}: we apply \eqref{pr} with $s_0=s_1=\sigma-2$, $s_2=s-1$ and note that $s_1+s_2>0$ if $\sigma+s>3$.

Let us prove \eqref{estb:Y}. We first compute using \eqref{eq:v} that
\[
\begin{aligned}
\p_z\mb&=-\p_zv\frac{\p_z^2\varrho}{(\p_z\varrho)^2}+\frac{\p_z^2v}{\p_z\varrho}\\
&=-\p_zv\frac{\p_z^2\varrho}{(\p_z\varrho)^2}-\frac{\alpha}{\p_z\varrho}\Delta_xv-\frac{\beta}{\p_z\varrho}\cdot\nabla_x\p_zv+\p_zv\frac{\gamma}{\p_z\varrho}.
\end{aligned}
\]
For the first and last terms,  we apply  \eqref{XY1} with $s_0=\sigma-1$, $s_1=\sigma-1$ and $s_2=s-2$, giving
\[
\|\p_zv\frac{\p_z^2\varrho}{(\p_z\varrho)^2} \|_{X^{\sigma-1}}\les \| \p_zv\|_{X^{\sigma-1}}\| \frac{\p_z^2\varrho}{(\p_z\varrho)^2}\|_{X^{s-2}}\le\cF(\| \eta\|_{H^s})\Vert f\Vert_{\widetilde{H}^\sigma_-}.
\]
Next for the second and third  terms, applying \eqref{XY1} with $s_0=\sigma-1$, $s_1=s-1$ and $s_2=\sigma-2$ yields
\[
\| \frac{\beta}{\p_z\varrho}\cdot\nabla_x\p_zv\|_{Y^{\sigma-1}}\les \| \frac{\beta}{\p_z\varrho}\|_{X^{s-1}}\| \nabla_x\p_zv\|_{X^{\sigma-2}}\le\cF(\| \eta\|_{H^s})\Vert f\Vert_{\widetilde{H}^\sigma_-}.
\]
 This finishes the proof of \eqref{estb:Y}.
\end{proof}
We can now state our main technical estimate.
\begin{lemm}\label{lemm:F2}
For any $z_0\in (-1, 0)$ and $\sigma\in [\mez, s]$, we have
\begin{align}\label{eqF2}
&(\p_z-T_a)\big[(\p_z-T_A)v-T_\mb(\p_z-T_A)\varrho\big]=F_2,\\ \label{est:F2}
&\| F_2\|_{Y^{\sigma-\mez}([z_0, 0])}\le \cF(\| \eta\|_{H^s})\big(1+\|\eta \|_{H^{s+\mez-\delta}}\big)\Vert f\Vert_{\widetilde{H}^\sigma_-}.
\end{align}
\end{lemm}
\begin{rema}
The direct consideration of the good unknown $u=v-T_\mb \varrho$ in \cite{ABZ1, PoyNg1} consists in obtaining good estimates for  
\[
(\p_z-T_a)(\p_z-T_A)u.
\]
 In our setting, even when $\sigma=s$, estimating this in $Y^{s-\mez}$ demands an estimate for $\| \p_z^2b\|_{X^{s-3}}$. However, in one space dimension, the low regularity $s>\tdm$ makes it challenging to prove  that $\p_z^3v\in X^{s-3}$, where $\p_z^3v$ appears when differentiating $\mb$ twice in $z$.  Lemma \ref{lemm:F2} avoids this issue. 
 \end{rema}
\begin{proof}[Proof of Lemma \ref{lemm:F2}]
Using \eqref{Neweq:v} and Lemma \ref{prop:decompose} with $\theta=\sigma-1$, we see that
\begin{equation*}
\begin{split}
(\partial_z-T_a)(\partial_z-T_A)v+R_Qv=T_\mb(\partial_z-T_a)(\partial_z-T_A)\varrho+(\mb Q\varrho-T_{\mb}Q\varrho)+T_\mb R_Q\varrho
\end{split}
\end{equation*}
which gives
\begin{equation*}
\begin{split}
F_2&=[T_\mb,(\partial_z-T_a)](\partial_z-T_A)\varrho-R_Qv+T_\mb R_Q\varrho+(\mb Q\varrho-T_{\mb}Q\varrho).
\end{split}
\end{equation*}
It is readily checked that 
\[
\| Q\varrho\|_{L^2 H^{s-1-\delta}}\le \cF(\| \eta\|_{H^s})\| \eta\|_{H^{s+\mez-\delta}}
\]
which in conjunction with \eqref{paralin:product} and \eqref{est:b1} yields
\[
\| \mb Q\varrho-T_{\mb}Q\varrho\|_{L^2 H^{\sigma-1}}\les \| \mb\|_{L^\infty H^{\sigma-1}}\| Q\varrho\|_{L^2 H^{s-1-\delta}}\le \cF(\| \eta\|_{H^s})\|\eta \|_{H^{s+\mez-\delta}}\Vert f\Vert_{\widetilde{H}^\sigma_-}.
\]
where  we have used \eqref{sigma:mainpara} in the first inequality. 

In view of \eqref{sigma:mainpara},  \eqref{est:RAa} can be applied with $\tt=\sigma-1$, implying the control of $R_Qv$. As for $T_\mb R_Q\varrho$ we apply \eqref{boundpara}, \eqref{est:RAa} with $\tt=s-1$, and \eqref{est:b1} to have
\begin{align*}
\|T_\mb R_Q\varrho\|_{L^2 H^{\sigma-1}}&\les \| \mb\|_{L^\infty H^{\sigma-1}}\| R_Q\varrho\|_{L^2 H^{s-1}}\\
&\le\cF(\| \eta\|_{H^s})\big(1+\|\eta \|_{H^{s+\mez-\delta}}\big)\Vert f\Vert_{\widetilde{H}^\sigma_-}.
\end{align*}
 Regarding the commutator in $F_2$, we write 
\[
[T_\mb,(\partial_z-T_a)](\partial_z-T_A)\rho=-T_{\p_z\mb}(\partial_z-T_A)\rho-[T_\mb, T_a](\partial_z-T_A)\rho.
\]
For $ T_{\p_z\mb}(\p_z-T_A)\varrho$ we distinguish two cases. 

{\it Cases 1:} $\sigma\in [s-\delta, s]$. Then, $\sigma+s\ge 2s-\delta>\frac52$ and \eqref{est:b2} can be applied. Noting in addition that $\sigma-1\le s-1\le s-\mez-\delta$,  \eqref{boundpara} yields
\[
\begin{aligned}
\| T_{\p_z\mb}(\p_z-T_A)\varrho\|_{L^2 H^{\sigma-1}}&\les\| \p_z\mb\|_{L^2 H^{\sigma-\tdm}}\| (\p_z-T_A)\varrho-h\|_{L^\infty H^{s-\mez-\delta}}\\
&\le \cF(\| \eta\|_{H^s})\big(1+\|\eta \|_{H^{s+\mez-\delta}}\big)\Vert f\Vert_{\widetilde{H}^\sigma_-}.
\end{aligned}
\]
{\it Cases 2:} $\sigma\in [\mez, s-\delta]$. In view of \eqref{estb:Y}, applying \eqref{est:XY3} we obtain 
\begin{align*}
\| T_{\p_z\mb}(\p_z-T_A)\varrho\|_{Y^{\sigma-\mez}}&\les \|\p_z\mb\|_{Y^{\sigma-1}}\| (\p_z-T_A)\varrho-h\|_{L^\infty H^{s-\mez-\delta}}\\
&\le\cF(\| \eta\|_{H^s})\big(1+\|\eta \|_{H^{s+\mez-\delta}}\big)\Vert f\Vert_{\widetilde{H}^\sigma_-}.
\end{align*}
To  treat the commutator $ [T_a, T_\mb]$ we again distinguish two cases. 

{\it Case 1:}  $\sigma\in (s-\delta, s]$. Then we have $\nu:=\sigma+\delta-s\in (0, \mez]$ since  $\delta\le \mez$ and $\sigma\le s$. In addition, $\sigma-1-\frac{d}{2}>\nu$ and $s-1-\frac{d}{2}>\nu$, implying  that $\mb\in  H^{\sigma-1}\subset W^{\nu, \infty}$ and $a\in \Gamma^1_\nu$ uniformly in $z$. Consequently, by virtue of Theorem \ref{theo:sc}, $[T_a, T_\mb]$ is of order $1-\nu$ and 
\begin{align*}
\| [T_\mb, T_a](\partial_z-T_A)\rho\|_{L^2 H^{\sigma-1}}&\le \cF(\| \eta\|_{H^s})\Vert f\Vert_{\widetilde{H}^\sigma_-}\| (\partial_z-T_A)\rho-h\|_{L^2 H^{\sigma-\nu}}\\
&= \cF(\| \eta\|_{H^s})\Vert f\Vert_{\widetilde{H}^\sigma_-}\| (\partial_z-T_A)\rho-h\|_{L^2 H^{s-\delta}}\\
&\le \cF(\| \eta\|_{H^s})\big(1+\|\eta \|_{H^{s+\mez-\delta}}\big)\Vert f\Vert_{\widetilde{H}^\sigma_-},
\end{align*}
where we have used \eqref{est:b1} in the first inequality.

{\it Case 2:} $\sigma\in [\mez, s-\delta]$. In this case we do not use the structure of the commutator but directly estimate using Theorem \ref{theo:sc} (i) and \eqref{boundpara}:
\begin{align*}
\Vert T_aT_\mb(\partial_z-T_A)\varrho\Vert_{L^2H^{\sigma-1}}&\le \cF(\| \eta\|_{H^s})\|T_\mb(\partial_z-T_A)\varrho\Vert_{L^2H^\sigma}\\
&\le \cF(\| \eta\|_{H^s})\Vert \mb\Vert_{L^\infty H^{\sigma-1}}\Vert(\partial_z-T_A)\varrho-h\Vert_{L^2H^{s-\delta}}\\
&\le \cF(\| \eta\|_{H^s})\big(1+\|\eta \|_{H^{s+\mez-\delta}}\big)\Vert f\Vert_{\widetilde{H}^\sigma_-},
\end{align*}
where in the  second inequality we have used  the fact that $\sigma\le s-\delta$.  The term   $T_\mb T_a(\partial_z-T_A)\varrho$ can be controlled similarly. This completes the proof of Lemma \ref{lemm:F2}.  
\end{proof}
\begin{proof}[Proof of Theorem \ref{tame:RDN}] The proof proceeds in two steps. 

{\bf Step 1.} Let us fix $-1<z_0<z_1<0$ and introduce a cut-off $\chi$ satisfying $\chi(z)=1$ for $z>z_1$ and $=0$ for $z<z_0$. Set 
\[
w=\chi(z)\big[(\p_z-T_A)v-T_\mb(\p_z-T_A)\varrho\big].
\]
 It follows from \eqref{eqF2} that
\bq\label{eq:w}
(\p_z-T_a)w=F_3:=\chi(z)F_2+\chi'(z)\big[(\p_z-T_A)v-T_\mb(\p_z-T_A)\varrho\big].
\eq
By virtue of \eqref{est:F2}, \eqref{estv:base}, \eqref{est:b1} and \eqref{norm:aA:0} we have
\bq\label{est:F3}
\| F_3\|_{Y^{\sigma-\mez}([z_0, 0])}\le  \cF(\| \eta\|_{H^s})\Vert f\Vert_{\widetilde{H}^\sigma_-}\big(1+\|\eta \|_{H^{s+\mez-\delta}}\big).
\eq
Next we note that 
\[
\RE(-a) \ge\frac{1}{\cF(\| \eta\|_{H^s})}|\xi|.
\]
Since $w(z_0)=0$,  applying Proposition \ref{prop:parabolic} to equation \eqref{eq:w} with the aid of \eqref{est:F3} we obtain
\bq\label{w:main0}
\| w\|_{X^{\sigma-\mez}([z_0, 0])}\le \cF(\| \eta\|_{H^s})\| F_3\|_{Y^{\sigma-\mez}([z_0, 0])}\le   \cF(\| \eta\|_{H^s})\Vert f\Vert_{\widetilde{H}^\sigma_-}\big(1+\|\eta \|_{H^{s+\mez-\delta}}\big).
\eq

\medskip
{\bf Step 2.} Starting from \eqref{ExpressionDNNew} and using Bony's decomposition, we find that
\begin{equation*}
\begin{split}
G^-(\eta)f&=T_{\frac{1+\vert\nabla_x\varrho\vert^2}{\partial_z\varrho}}(T_Av-T_\mb T_A\varrho)-(T_{\nabla_x\varrho}\nabla_xv-T_{\partial_zv}T_{\frac{\nabla_x\varrho}{\partial_z\varrho}}\nabla_x\varrho)-(T_{\nabla_xv}-T_{\partial_zv}T_{\frac{\nabla_x\varrho}{\partial_z\varrho}})\nabla_x\varrho\\
&\quad+\left\{T_{\partial_zv}\left(\frac{1+\vert\nabla_x\varrho\vert^2}{\partial_z\varrho}\right)+T_{\frac{1+\vert\nabla_x\varrho\vert^2}{\partial_z\varrho}}T_\mb\partial_z\varrho-2T_{\partial_zv}T_{\frac{\nabla_x\varrho}{\partial_z\varrho}}\nabla_x\varrho\right\}\\
&\quad +T_{\frac{1+\vert\nabla_x\varrho\vert^2}{\partial_z\varrho}}w-R(\nabla_x\varrho,\nabla_xv)+R(\frac{1+\vert\nabla_x\varrho\vert^2}{\partial_z\varrho},\partial_zv)
\end{split}
\end{equation*}
where the right-hand side is evaluated at $z=0$. We will see that this gives \eqref{ParaDN} by estimating each term one by one.

Using Theorem \ref{theo:sc}, \eqref{norm:aA:0}, \eqref{norm:aA:1} and \eqref{estv:base}, we first observe that
\begin{equation*}
\begin{split}
R^1=\left\{T_{\frac{1+\vert\nabla_x\varrho\vert^2}{\partial_z\varrho}}(T_Av-T_\mb T_A\varrho)-(T_{\nabla_x\varrho}\nabla_xv-T_{\partial_zv}T_{\frac{\nabla_x\varrho}{\partial_z\varrho}}\nabla_x\varrho)\right\}-T_{\frac{1+\vert\nabla_x\varrho\vert^2}{\partial_z\varrho}A-i\xi\cdot\nabla_x\varrho}\underbrace{(v-T_\mb\rho)}_{=u}
\end{split}
\end{equation*}
satisfies estimates as in \eqref{est:RDN}. Using the formula \eqref{ld} and \eqref{aA}, we see that
\bq\label{form:ld}
A(x, \xi)\frac{1+\la\na_x\varrho\ra^2}{\partial_z\varrho}-i\xi \cdot\na_x\varrho\vert_{z=0}=\lambda
\eq
and this gives the first main term in \eqref{ParaDN}. Similarly, we obtain that
\begin{equation*}
\begin{split}
R^2=-(T_{\nabla_xv}-T_{\partial_zv}T_{\frac{\nabla_x\varrho}{\partial_z\varrho}})\nabla_x\varrho+T_{\nabla_xv-\mb\nabla_x\varrho}\nabla_x\varrho&=(T_{\partial_zv}T_{\frac{\nabla_x\varrho}{\partial_z\varrho}}-T_{\frac{\partial_zv\nabla_x\varrho}{\partial_z\varrho}})\nabla_x\varrho
\end{split}
\end{equation*}
is acceptable, and since
\[
\na_x v-\mb\na_x\varrho\vert_{z=0}=\nabla f-B\nabla \eta=V
\]
we obtain the second main estimate in \eqref{ParaDN}.

We claim that all the other terms are remainders. Next we paralinearize the function $F(m, n)=\frac{1+|m|^2}{n+h}-h^{-1}$ where $m\in \Rr^d$ and $n\in \Rr$. Clearly $F(0, 0)=0$, $\nabla_mF=\frac{2m}{n+h}$, and $\p_nF=-\frac{1+m^2}{(n+h)^2}$. Applying Theorem \ref{paralin:nonl} with $\mu=s-\mez-\delta$ and $\tau=\delta$ yields
\[
\frac{1+|\na_x\varrho|^2}{\partial_z\varrho}-h^{-1}=F(\na_x\varrho, \p_z\varrho-h)=T_{2\frac{\na_x\varrho }{\p_z\varrho}}\cdot\na_x\varrho-T_{\frac{1+|\na_x\varrho|^2}{(\p_z\varrho)^2}}(\p_z\varrho-h)+R_F(\na_x\varrho, \p_z\varrho-h^{-1})
\]
with
\[
\| R_F(\na_x\varrho, \p_z\varrho-h^{-1})\|_{ H^{s-\mez}}\le \cF(\| \eta\|_{H^s})\|\eta \|_{H^{s+\mez-\delta}}.
\]
 Then by virtue of Theorem \ref{theo:sc} (ii) with $\rho=\delta$ we obtain that
\[
R^3=T_{\p_zv}\big(\frac{1+\la\na_x\varrho\ra^2}{\partial_z\varrho}-h^{-1}\big)-2T_{\partial_zv}T_{\frac{\na_x\varrho}{\p_z\varrho}}\cdot\na_x\varrho+T_{\frac{1+|\na_x\varrho|^2}{\p_z\varrho}}T_{\frac{\partial_zv}{\partial_z\varrho}}(\p_z\varrho-h).
\]
is acceptable as in \eqref{est:RDN}. The next term follows from \eqref{w:main0}. Finally, by \eqref{estv:base}, \eqref{Bony1} we get
\[
\| R(\na_x\varrho, \na_x v)\|_{L^\infty H^{\sigma-\mez}}\les \| \na_x\varrho\|_{L^\infty H^{s-\mez-\delta}}\| \na_x v\|_{L^\infty H^{\sigma-1}}\les \cF(\| \eta\|_{H^s})\Vert f\Vert_{\widetilde{H}^\sigma_-}\|\eta \|_{H^{s+\mez-\delta}}
\]
and similarly, since $\frac{1+|\na_x\varrho|^2}{\p_z\varrho}-\frac{1}{h}\in L^\infty H^{s-\mez-\delta}\subset L^\infty C_*^\mez$ it follows from \eqref{Bony2} that
\begin{align*}
\| R(\frac{1+|\na_x\varrho|^2}{\p_z\varrho}, \na_x v)\|_{L^\infty H^{\sigma-\mez}}&\les \big(\| \frac{1+|\na_x\varrho|^2}{\p_z\varrho}-\frac{1}{h}\|_{L^\infty H^{s-\mez-\delta}}+\frac{1}{h}\big)\| \na_x v\|_{L^\infty H^{\sigma-1}}\\
&\les \cF(\| \eta\|_{H^s})\Vert f\Vert_{\widetilde{H}^\sigma_-}\big(1+\|\eta \|_{H^{s+\mez-\delta}}\big).
\end{align*}
The proof of Theorem \ref{tame:RDN} is complete. 
\end{proof}
\subsection{Contraction estimates}\label{section:contraction}
In order to obtain uniqueness and stability estimates for the Muskat problem, we need  contraction estimates for the Dirichlet-Neumann operator associated to two different surfaces $\eta_1$ and $\eta_2$. Since we always assume in this subsection that $\eta_j\in L^\infty(\Rr^d)$ and $\dist(\eta_j, \Gamma^-)>h>0$, Proposition \ref{prop:wtH} guarantees that the spaces $\wt H^s_\pm$, defined by \eqref{wtH-}-\eqref{wtH+}-\eqref{def:wtHs}, are independent of $\eta_j$. We have the following results.
\begin{theo}\label{theo:contractionDN2}
Let  $s>1+\frac{d}{2}$ with $d\ge 1$. Let $\delta\in (0, \mez]$ satisfy $\delta<s-1-\frac{d}{2}$. 
 Consider $f\in \widetilde{H}^s_-$ and  $\eta_1$, $\eta_2\in H^s(\Rr^d)$ with $\dist(\eta_j, \Gamma^-)>4h>0$ for $j=1, 2$. Then  for any $\sigma\in[\mez+\delta, s]$, we have
\bq
G^{-}(\eta_1)f-G^{-}(\eta_2)f =-T_{\ld_2B_2}(\eta_1-\eta_2)-T_{V_2}\cdot\na (\eta_1-\eta_2)+R^-_2(\eta_1, \eta_2)f
\eq 
where
\begin{equation}\label{contraction:DN2}
\| R^-_2(\eta_1, \eta_2)f  \|_{H^{\sigma-1}}\le 
\mathcal{F}\big(\| (\eta_1, \eta_2) \|_{H^s}\big)\| \eta_1-\eta_2\|_{H^{\sigma-\delta}}\Vert f\Vert_{\widetilde{H}^s_-}
\end{equation}
for some $\cF:\Rr^+\to \Rr^+$ depending only on $(s, \sigma, h, \delta)$. 
\end{theo} 
\begin{coro}\label{theo:contractionDN}
Let  $s>1+\frac{d}{2}$ with $d\ge 1$. Consider $f\in \wt H^s_-$ and  $\eta_1$, $\eta_2\in H^s(\Rr^d)$ with $\dist(\eta_j, \Gamma^-)> 4h>0$ for $j=1, 2$. Then  for all $\sigma\in [\mez, s]$, we have 
\begin{equation}\label{contraction:DN}
\| G^{-}(\eta_1)f-G^{-}(\eta_2)f  \|_{H^{\sigma-1}}\le 
\mathcal{F}\big(\| (\eta_1, \eta_2) \|_{H^s}\big)\| \eta_1-\eta_2\|_{H^\sigma}\| f\|_{\wt H^s_{-}}
\end{equation}
for some $\cF:\Rr^+\to \Rr^+$ depending only on $(s, \sigma, h)$.
\end{coro}
 \begin{rema}
 The following contraction estimate was obtained in Theorem 5.2 in \cite{ABZ3}
 \bq\label{contraction:ABZ}
 \| G^-(\eta_1)f-G(\eta_2)f\|_{H^{r-2}}\le \mathcal{F}\big(\| (\eta_1, \eta_2) \|_{H^r}\big)\| \eta_1-\eta_2\|_{H^{r-1}}\| f\|_{H^{r-\mez}},
 \eq
 where $r>\tdm+\frac{d}{2}$ ($\mez$ derivative above scaling). It was also noted  in \cite{ABZ3} (see Remark 5.3 therein) that the authors were unable to obtain a similar estimate in higher norms. Applying  \ref{contraction:DN} with $s=r-\mez$ gives such estimates.
  \end{rema}
From now on, to simplify notation, we let 
\[
N_s=\Vert\eta_1\Vert_{H^s}+\Vert \eta_2\Vert_{H^s}.
\]
 The rest of this section is devoted to the proof of Theorem \ref{theo:contractionDN2}.  We follow similar steps as in the previous section, the main novelty coming from the two different domains. To define $G(\eta_j)f$ we call $\phi_j$ solution to \eqref{eq:elliptic} with surface $\eta_j$ and Dirichlet data $f$. For the sake of contraction estimates, we shall use a diffeomorphism different from the one defined by \eqref{lesomega}-\eqref{omega1}-\eqref{diffeo}. Assume $\dist(\eta_j, \Gamma^-)>4h>0$. There exists $\eta_*:\Rr^d\to \Rr$ such that 
 \bq
\|  \eta_*+\frac{3h}{2}\|_{H^{s+100}}\le CN_s
 \eq
 and
\bq\label{def:eta*}
\underline b^-(x)+2h<\eta_*(x)<\eta_j(x)-h
\eq
when the depth is finite and $\eta_*(x)<\eta_j(x)-h$ when $\Gamma^-=\emptyset$. One can take $\eta_*$ to be a mollification of $\min(\eta_1(x), \eta_2(x))-\frac{3h}{2}$. Then we divide $\Omega_j^-$ into 
 \begin{equation}\label{divideOmegaj}
\begin{cases}
\Omega^{-}_{j, 1} = \{(x,y): x \in \Rr^d, \eta_*(x)<y<\eta_j(x)\},\\
\Omega^{-}_{j ,2} = \{(x,y)\in \Omega^-_j: y\leq\eta_*(x)\}
\end{cases}
 \end{equation}
and set $\wt{\Omega^-}=\wt{\Omega^-_1}\cup \wt{\Omega^-_2}$ where
\bq
\begin{cases}
\wt{\Omega^-_1}= \Rr^d\times (-1, 0),\\
\wt{\Omega^{-}_2} =
\begin{cases}
 \Rr^d\times (-\infty, -1]\quad\text{if}~\Gamma^-=\emptyset,\\
 \{(x, z)\in \Rr^d\times (-\infty, -1]:  z+1+\eta_*(x)> \underline b^-(x)\}\quad\text{if}~\underline b^-\in W^{1, \infty}.
 \end{cases}
\end{cases}
 \eq
 Note that $\Omega^-_{1,2}=\Omega^-_{2,2}$ and the sets  $\wt{\Omega^-_1}$ and $\wt{\Omega^-_2}$ are independent of $\eta_j$. Define 
 \begin{equation}\label{diffeo:diff}
\left\{
\begin{aligned}
&\varrho_j(x,z)=  (1+z)e^{\tau z\langle D_x \rangle }\eta_j(x) -z \eta_*(x)\quad \text{for } (x,z)\in \wt{\Omega^-_1},\\
&\varrho_j(x,z) =z+1+\eta_*(x)\quad\text{for } (x,z)\in \wt{\Omega^-_2}.
\end{aligned}
\right.
\end{equation}
In particular, $\varrho_1=\varrho_2$ in $\wt{\Omega^-_2}$. For $\tau>0$ sufficiently small, it is easy to check that the mappings $(x, z)\in \wt{\Omega^-}\mapsto (x, \varrho_j(x, z))\in \Omega^-_j$ and $(x, z)\in\wt{\Omega^-_1}\mapsto (x, \varrho_j(x, z))\in \Omega^-_{j, 1}$ are Lipschitz diffeomorphisms, where the latter is smooth in $z$. Letting also $\varrho_\delta=\varrho_1-\varrho_2$, we observe as in \eqref{EstimrhoNew} that
\begin{equation}\label{EstimRho12}
\begin{split}
& \min(\partial_z\varrho_1,\partial_z\varrho_2)\ge  \min(1,\frac h2),\\
&\Vert (\na_x\varrho_j, \p_z\varrho_j-\frac{3h}{2})\Vert_{X^{s-1}([-1,0])} \lesssim \Vert \eta_1\Vert_{H^s}+\Vert \eta_2\Vert_{H^s}, \\
&\Vert \varrho_\delta\Vert_{X^\sigma((-\infty,0])}\lesssim \Vert \eta_1-\eta_2\Vert_{H^\sigma}
\end{split}
\end{equation}
if $\tau>0$ is chosen small enough (depending on $\Vert \eta_1\Vert_{B^1_{\infty,1}}+\Vert \eta_2\Vert_{B^1_{\infty,1}}$).

As in \eqref{eq:v}, 
\[
v_j(x, z):=\phi_j(x, \varrho_j(x, z))
\]
 solves  
\bq\label{eq:vj}
\mathcal{L}_jv_j:=(\p_z^2+\alpha_j\Delta_x+\beta_j\cdot\nabla_x\p_z-\gamma_j\p_z)v_j=0\quad\text{in}~\Rr^d\times [-1, 0],
\eq
with $(\alpha_j, \beta_j, \gamma_j)$ defined in terms of $\varrho_j$ as in \eqref{alpha} and satisfies\footnote{A priori, Proposition \ref{prop:elliptic} would only give a bound in $X^{s-1}([z_0,0])$ for some $z_0>-1$. However, one can first apply this with $\varrho_j$ replaced by $\varrho_{j,\ast}$ which is equal to $\varrho_j$ for $-1\le z\le 0$ and smooth for $-2\le z\le 0$ to obtain a bound on $[-1,0]$.}
\begin{equation}\label{EstimV_jNew}
\Vert \nabla_{x,z}v_j\Vert_{X^{s-1}([-1,0])}\lesssim \mathcal{F}(N_s)\Vert f\Vert_{\widetilde{H}^s_-}.
\end{equation}
The difference
\[
v=v_1-v_2
\]
then solves
\bq\label{eq:vdiff}
\mathcal{L}_1v=F:=-(\alpha_1-\alpha_2)\Delta_xv_2-(\beta_1-\beta_2)\cdot\nabla_x\p_zv_2+(\gamma_1-\gamma_2)\p_zv_2\quad\text{in}~\Rr^d\times [-1, 0].
\eq

As before, we start with an estimate for $v$ in the low norm $X^{-\mez}([-1, 0])$. 
\begin{lemm}\label{lemm:contractlow} 
\bq\label{contractest:low}
\| \na_{x,z}v\|_{X^{-\mez}([-1, 0])}\le \cF(N_s)\| \eta_1-\eta_2\|_{H^\mez}\Vert f\Vert_{\widetilde{H}^s_-}.
\eq
\end{lemm}
\begin{proof}
We first recall the variational characterization \eqref{variform:u} 
\bq\label{variform:phij}
\int_{\Omega^-_j}\na_{x, y} \phi_j\cdot \na_{x, y} \varphi dxdy=0\quad\forall \varphi \in  H^{1}_{0,\ast}(\Omega^-_j)=\{\varphi \in \dot H^1(\Omega^-_j): \varphi\vert_{\Sigma_j}=0\}.
\eq
In the fixed domain $\wt{\Omega^-}$, this becomes
\[
\int_{\wt{\Omega^-}}\mathcal{A}_j\na_{x, z} v_j\cdot \na_{x, z} \tt dxdy=0\quad\forall \tt \in  H^{1}_{0,\ast}(\wt{\Omega^-})=\{\varphi \in \dot H^1(\wt{\Omega^-}): \varphi\vert_{z=0}=0\},
\]
where 
\[
\mathcal{A}_j=
\begin{bmatrix}
\p_z\varrho_j\cdot Id & -\na_x\varrho_j\\
-(\na_x\varrho_j)^T& \frac{1+|\na_x\varrho_j|^2}{\p_z\varrho_j}
\end{bmatrix}.
\]
Consequently,
\bq\label{variform:fixed}
\int_{\wt{\Omega^-}}\mathcal{A}_1\na_{x, z} v\cdot \na_{x, z} \tt dxdy=\int_{\wt{\Omega^-}}(\mathcal{A}_2-\mathcal{A}_1)\na_{x, z} v_2\cdot \na_{x, z} \tt dxdy\quad\forall \tt \in  H^1_{0, *}(\wt{\Omega^-}).
\eq
Since $v\vert_{z=0}=0$, we have $v\in H^{1}_{ 0,\ast}(\wt{\Omega^-})$. Inserting $\tt=v$ into \eqref{variform:fixed} yields
\bq\label{variform:fixed:2}
\begin{aligned}
\int_{\wt{\Omega^-}}\mathcal{A}_1\na_{x, z} v\cdot \na_{x, z} v dxdy&=\int_{\wt{\Omega^-}}(\mathcal{A}_2-\mathcal{A}_1)\na_{x, z} v_2\cdot \na_{x, z} v dxdy\\
&=\int_{\Rr^d\times (-1, 0)}(\mathcal{A}_2-\mathcal{A}_1)\na_{x, z} v_2\cdot \na_{x, z} v dxdy,
\end{aligned}
\eq
where we used the fact that $\mathcal{A}_2=\mathcal{A}_1$ in $\wt{\Omega^-_2}$, which in turn comes from the fact that $\varrho_1=\varrho_2$  in $\wt{\Omega^-_2}$. In view of  \eqref{EstimRho12} and \eqref{EstimV_jNew}, 
\bq\label{est:diffA}
\begin{split}
&\| \mathcal{A}_1-\mathcal{A}_2\|_{L^2(\Rr^d\times (-1, 0))}\le \cF(\|(\eta_1, \eta_2)\|_{B^1_{\infty, 1}})\|\eta_1-\eta_2\|_{H^\mez},\\
&\Vert \na_{x, z}v_2\Vert_{L^\infty(\mathbb{R}^d\times (-1,0))}\le \cF(N_s)\Vert f\Vert_{\widetilde{H}^s_-}
\end{split}
\eq
and since (see \eqref{EstimRho12})
\begin{equation*}\label{lowerbound:A}
\langle\mathcal{A}_1\na_{x, z} v\cdot \na_{x, z} v\rangle \ge \min(\p_z\varrho_1, \frac{1}{\p_z\varrho_1})| \na_{x, z}v|^2\ge \frac{1}{\cF(\| \eta_1\|_{B^1_{\infty, 1}})}| \na_{x, z}v|^2
\end{equation*}
pointwise in $\wt{\Omega^-}$, we obtain that
\bq\label{contraction:H1}
\| \na_{x, z} v\|_{L^2(\wt{\Omega^-})}\le \cF(N_s)\|\eta_1-\eta_2\|_{H^\mez}\| f\|_{\widetilde{H}^s_{-}}.
\eq
Since $\Rr^d\times (-1, 0)\subset \wt{\Omega^-}$, this yields
\bq\label{nax:vdiff}
\| \na_{x, z} v\|_{L^2((-1, 0); L^2(\Rr^d))}\le  \cF(N_s)\|\eta_1-\eta_2\|_{H^\mez}\| f\|_{\widetilde{H}^s_{-}}.
\eq
According to Theorem \ref{trace:Lions}, 
\bq\label{nax:vdiffC}
\| \na_x v\|_{X^{-\mez}([-1, 0])}\les \| \na_{x, z} v\|_{L^2((-1, 0); L^2)}.
\eq
As for $\| \p_zv\|_{X^{-\mez}([-1, 0])}$ it remains to estimate $\| \p_zv\|_{C([-1, 0]; H^{-\mez})}$. Setting 
\[
\Xi_j(x, z)=-\na_x\varrho_j\cdot \na_x v_j+\frac{1+|\na_x\varrho_j|^2}{\p_z\varrho_j}\p_zv_j,
\]
 it follows from the equation $\cnx_{x, z}(\mathcal{A}_j\na_{x,z}v_j)=0$ that
\[
\p_z\Xi_j=-\cnx_x(\p_z\varrho_j\na_xv_j-\na_x\varrho_j\p_zv_j).
\]
Hence $\Xi=\Xi_1-\Xi_2$ is a divergence 
\[
\p_z\Xi=-\cnx_x(\p_z\varrho_\delta\na_xv_1+\p_z\varrho_2\na_xv-\na_x\varrho_\delta\p_zv_1-\na_x\varrho_2\p_zv),
\]
and using \eqref{EstimRho12} and \eqref{EstimV_jNew}, we obtain the bounds
\begin{equation}
\begin{split}
\Vert \Xi\Vert_{L^2_zL^2}&\le  \cF(N_s)\| \eta_1-\eta_2\|_{H^\mez}\| f\|_{\wt H^s_{-}},\\
\| \p_z\Xi\|_{L^2_zH^{-1}}&\le \| \p_z\varrho_\delta\na_xv_1+\p_z\varrho_2\na_xv-\na_x\varrho_\delta\p_zv_1-\na_x\varrho_2\p_zv\|_{L^2_zL^2},\\
&\le  \cF(N_s)\| \eta_1-\eta_2\|_{H^\mez}\| f\|_{\wt H^s_{-}}.
\end{split}
\end{equation}
Theorem \ref{trace:Lions} then yields
\[
\| \Xi\|_{C([-1, 0]; H^{-\mez})}\le  \cF(N_s)\| \eta_1-\eta_2\|_{H^\mez}\| f\|_{\widetilde{H}^s_{-}}.
\]
Finally, by writting
\[
\p_zv=\frac{\p_z\varrho_2}{1+|\na_x\varrho_2|^2}\Big\{\Xi+\na_x\varrho_2\cdot \na_x v+\frac{(1+\vert\nabla_x\varrho_1\vert^2)}{\partial_z\varrho_1\p_z\varrho_2}\partial_z\varrho_\delta+(\nabla_xv_1-\frac{\p_zv_1\nabla_x(\varrho_1+\varrho_2)}{\partial_z\varrho_1})\nabla_x\varrho_\delta\Big\}
\]
we deduce that 
\[
\| \p_zv\|_{C([-1, 0]; H^{-\mez})}\le  \cF(N_s)\| \eta_1-\eta_2\|_{H^\mez}\| f\|_{\widetilde{H}^s_{-}}.
\]
This completes  the proof of Lemma \ref{lemm:contractlow}.
\end{proof}

This low-regularity bound can easily be upgraded to a bound with no loss of regularity in $\eta_1-\eta_2$  with the aid of the next lemma. We shall use frequently the fact that for $s>1+\frac {d}{2}$, $\sigma\in [\mez, s]$ and $\delta<s-1-\frac{d}{2}$, we have
\bq\label{cons:cdsigma}
\sigma+s>2+\delta.
\eq
\begin{lemm}
For any $\sigma\in [\mez, s]$, we have
\begin{align}\label{contra:alpha}
&\| \alpha_1-\alpha_2\|_{X^{\sigma-1}([-1, 0])}+\| \beta_1-\beta_2\|_{X^{\sigma-1}([-1, 0])}\le \cF(\| \eta\|_{H^s})\| \eta_1-\eta_2\|_{H^\sigma},\\
\label{contra:gammaY}
&\| \gamma_1-\gamma_2\|_{Y^{\sigma-1}([-1, 0])}\le \cF(\| \eta\|_{H^s})\| \eta_1-\eta_2\|_{H^\sigma}.
\end{align}
\end{lemm}
\begin{proof}
From the definition of $\alpha$ and $\beta$ we see that they are nonlinear functions of $\na_{x, z}\varrho$ which is bounded in $L^\infty_z H^{s-1}_x$. By the  product rule \eqref{pr}, we have that  multiplication with $H^{s-1}$ is a continuous linear operator from $H^\nu$ to $H^\nu$ for any $\nu\in [-\mez, s]$. Thus, \eqref{contra:alpha}  follows easily. 

For $\gamma_1-\gamma_2$, let us consider the typical term  $\frac{\alpha_1}{\p_z\varrho_1}\Delta_x\varrho_1-\frac{\alpha_2}{\p_z\varrho_2}\Delta_x\varrho_2$ which in turns contains the following typical terms
\[
E_1=\frac{\alpha_1-\alpha_2}{\p_z\varrho_1}\Delta_x\varrho_1,\quad E_2=\frac{\alpha_2}{\p_z\varrho_2}\Delta_x\varrho_\delta.
\]
In view of \eqref{cons:cdsigma}, using \eqref{XY1} with $s_0=\sigma-1$, $s_1=\sigma-1$ and $s_2=s-2$ we obtain
\[
\| E_1\|_{Y^{\sigma-1}}\les \| \frac{\alpha_1-\alpha_2}{\p_z\varrho_1}\|_{X^{\sigma-1}}\| \Delta_x\varrho_1\|_{X^{s-2}}\le\cF(\| \eta\|_{H^s})\| \eta_1-\eta_2\|_{H^\sigma}.
\]
On the other hand, applying \eqref{XY1} with $s_0=\sigma-1$, $s_1=s-1$ and $s_2=\sigma-2$, we bound $E_2$ as
\[
\| E_2\|_{Y^{\sigma-1}}\les \| \frac{\alpha_2}{\p_z\varrho_2}\|_{X^{s-1}}\| \Delta_x\varrho_\delta\|_{X^{\sigma-2}}\le\cF(\| \eta\|_{H^s})\| \eta_1-\eta_2\|_{H^\sigma}.
\]
The other terms can be treated similarly; this finishes the proof of \eqref{contra:gammaY}. 
\end{proof}

\begin{lemm}\label{lemm:contractv}
For any $\sigma\in [\mez, s]$, we have
\bq\label{est:contractv}
\lA \nabla_{x,z} v\rA_{X^{\sigma-1}([z_0,0])}\le \mathcal{F}\big(N_s\big)\| \eta_1-\eta_2\|_{H^\sigma}\| f\|_{\wt H^s_{-}}.
\eq
\end{lemm}
\begin{proof}
We claim that for $z_1\in (-1, 0)$, and $F$ as in \eqref{eq:vdiff}, there exists $\cF$ such that
\begin{equation}\label{estF:contra}
\| F\|_{Y^{\sigma-1}([z_1, 0])}
\le \mathcal{F}\big(N_s\big)\| \eta_1-\eta_2 \|_{H^\sigma}\| f\|_{\wt H^s_{-}}.
\end{equation}
We first apply \eqref{XY2} with $s_0=\sigma-1$, $s_1=\sigma-1$ and $s_2=s-1$:
\[
\| (\gamma_1-\gamma_2)\p_zv_2\|_{Y^{\sigma-1}}\les \| \gamma_1-\gamma_2\|_{Y^{\sigma-1}}\| \p_zv_2\|_{X^{s-1}}\le \mathcal{F}\big(N_s\big)\| \eta_1-\eta_2 \|_{H^\sigma}\| f\|_{\wt H^s_{-}},
\]
where we have used \eqref{contra:gammaY} in the second inequality. As for $(\alpha_1-\alpha_2)\Delta_xv_2$ we apply \eqref{XY1} with $s_0=\sigma-1$, $s_1=\sigma-1$ and $s_2=s-2$:
\[
\| (\alpha_1-\alpha_2)\Delta_xv_2\|_{Y^{\sigma-1}}\les \| \alpha_1-\alpha_2\|_{X^{\sigma-1}}\| \Delta_xv_2\|_{X^{s-2}}\les  \mathcal{F}\big(N_s\big)\| \eta_1-\eta_2 \|_{H^\sigma}\| f\|_{\wt H^s_{-}}.
\]
Proceeding similarly for $(\beta_1-\beta_2)\cdot\nabla_{x}\partial_zv_2$, we obtain \eqref{estF:contra}. Since $v\vert_{z=0}=0$, Proposition  \ref{prop:elliptic} gives that
\bq\label{contractest:v}
\lA \nabla_{x,z} v\rA_{X^{\sigma-1}([z_0,0])}
\le \mathcal{F}(N_s)\big(\lA F\rA_{Y^{\sigma-1}([z_1, 0])}
+ \lA \nabla_{x,z} v\rA_{X^{-\mez}([-1,0])} \big)
\eq
 for  $-1<z_1<z_0<0$. Combining this with \eqref{estF:contra}, Lemma \ref{lemm:contractlow} and the condition $\sigma\ge \mez$, we finish the proof.
\end{proof}
\begin{proof}[Proof of Corollary \ref{theo:contractionDN}]
 Corollary \ref{theo:contractionDN} can be deduced from Theorem \ref{contraction:DN}, Theorem \ref{theo:sc} (i) and the fact that $B_2$ and  $V_2$ are in $L^\infty_x$. Here we give a short proof using Lemma \ref{lemm:contractv}. In view of \eqref{ExpressionDNNew}, we find that typical terms in $G^-(\eta_1)f-G^-(\eta_2)f$ are $\na_x(\varrho_1-\varrho_2)\cdot\na_x v_1\vert_{z=0}$ and $\na_x\varrho_2\cdot\na_x(v_1-v_2)\vert_{z=0}$. Using \eqref{pr} and \eqref{cons:cdsigma}, we have at $z=0$ that
\begin{align*}
&\| \na_x(\varrho_1-\varrho_2)\cdot\na_x v_1\|_{H^{\sigma-1}}\les \| \na_x(\varrho_1-\varrho_2)\|_{H^{\sigma-1}}\| \na_xv_1\|_{H^{s-1}}\le\mathcal{F}\big(N_s\big)\| \eta_1-\eta_2\|_{H^\sigma}\| f\|_{\wt H^s_{-}},\\
&\|\na_x\varrho_2\cdot\na_x(v_1-v_2)\|_{H^{\sigma-1}}\les \| \na_x\varrho_2\|_{H^{s-1}}\| \na_xv\|_{H^{\sigma-1}}\le\mathcal{F}\big(N_s\big)\| \eta_1-\eta_2\|_{H^\sigma}\| f\|_{\wt H^s_{-}},
\end{align*}
where we have applied Lemma \ref{lemm:contractv} in the last inequality. This finishes the proof of Corollary \ref{theo:contractionDN}. 
\end{proof}

Let us turn to the proof of Theorem \ref{theo:contractionDN2}. Fixing $\sigma\in [\mez+\delta, s]$, we have $\sigma-\delta\in [\mez, s-\delta]$, and  hence  Lemma \ref{lemm:contractv} yields the contraction estimate
\bq\label{est:contractv:2}
\lA \nabla_{x,z} v\rA_{X^{\sigma-1-\delta}([z_0,0])}\le \mathcal{F}\big(N_s\big)\| \eta_1-\eta_2\|_{H^{\sigma-\delta}}\| f\|_{\wt H^s_{-}}.
\eq
  We first prove a technical analog of Lemma \ref{lemm:F2}.
\begin{lemm}\label{lemm:F2analog}
With notation similar to Lemma \ref{prop:decompose}, letting $\mb_2=\frac{\partial_zv_2}{\partial_z\varrho_2}$ we have
\begin{equation}\label{DefR2Delta}
\begin{split}
&(\partial_z-T_{a_1})\left[(\partial_z-T_{A_1})v-T_{\mb_2}(\partial_z-T_{A_1})\varrho_\delta\right]=R_\delta^1,\\
&\Vert R^1_\delta\Vert_{Y^{\sigma-1}([z_0,0])}\le\mathcal{F}(N_s)\Vert \eta_1-\eta_2\Vert_{H^{\sigma-\delta}}\Vert f\Vert_{\widetilde{H}^s_-}
\end{split}
\end{equation}
for any $z_0\in (-1, 0)$.
\end{lemm}
\begin{proof}
 Applying \eqref{est:b1}, \eqref{est:b2} and \eqref{est:b3} with $\sigma=s$ (note that $s+s>3$) we obtain that $\mb_2$ satisfies  
\begin{equation}\label{BdonB2}
 \Vert \mb_2\Vert_{L^\infty([z_0,0]; H^{s-1})}+\Vert \nabla_{x,z}\mb_2\Vert_{X^{s-2}([z_0,0])}\le\mathcal{F}(N_s)\Vert f\Vert_{\widetilde{H}^s_-}.
\end{equation} 
We also recall from \eqref{EstimV_jNew} that
\[
\Vert \nabla_{x,z}v_j\Vert_{X^{s-1}([-1,0])}\lesssim \mathcal{F}(N_s)\Vert f\Vert_{\widetilde{H}^s_-}.
\]
Set $Q_j=\p_z^2+\alpha_j\Delta_x+\beta_j\cdot\nabla_x\p_z$. Using \eqref{pr}, \eqref{EstimRho12} and \eqref{cons:cdsigma} we obtain the bounds 
\begin{equation*}
\begin{split}
&\Vert (\nabla_{x}\varrho_1,\nabla_{x}\varrho_2)\Vert_{X^{s-1}}+\Vert (\p_z\varrho_1-\tdm h,\p_z\varrho_2-\tdm h)\Vert_{X^{s-1}}+\Vert (Q_1\varrho_1,\gamma_1)\Vert_{X^{s-2}}\le \mathcal{F}(N_s),\\
&\Vert \na_{x, z}\varrho_\delta\Vert_{X^{\sigma-1-\delta}}+\Vert (\alpha_1-\alpha_2,\beta_1-\beta_2)\Vert_{X^{\sigma-1-\delta}}\le\mathcal{F}(N_s)\Vert \eta_1-\eta_2\Vert_{H^{\sigma-\delta}}.
\end{split}
\end{equation*}
On the other hand, we claim that
\bq\label{est:Q1rhodelta}
\| Q_1\varrho_\delta\|_{Y^{\sigma-1-\delta}}\le\mathcal{F}(N_s)\Vert \eta_1-\eta_2\Vert_{H^{\sigma-\delta}}.
\eq
Indeed, from the definition of $\varrho_\delta$ we have
\[
\| \p_z^2\varrho_\delta\|_{L^2 H^{\sigma-\tdm-\delta}}\le \cF(N_s)\Vert \eta_1-\eta_2\Vert_{H^{\sigma-\delta}};
\]
on the other hand, by \eqref{XY1},
\[
\| \alpha_1\Delta_x\varrho_\delta\|_{Y^{\sigma-1-\delta}}\les(\|\alpha_1-h^2\|_{X^{s-1}}+1)\| \Delta_x\varrho_\delta\|_{X^{\sigma-2-\delta}}\le \cF(N_s)\Vert \eta_1-\eta_2\Vert_{H^{\sigma-\delta}}
\]
and similarly for $\beta_j\cdot\nabla_x\p_z\varrho_\delta$.

{\bf Step 1.} From \eqref{eq:vj} and the definition \eqref{alpha} of $\gamma$ we have 
\begin{equation*}
\begin{split}
Q_1v-(\gamma_1-\gamma_2)\partial_zv_2&=\gamma_1\partial_zv-(Q_1-Q_2)v_2,\\
\gamma_1-\gamma_2-\frac{1}{\partial_z\varrho_2}Q_1\varrho_\delta &=\frac{1}{\partial_z\varrho_2}(Q_1-Q_2)\varrho_2+\frac{Q_1\varrho_1}{\partial_z\varrho_1\partial_z\varrho_2}\partial_z\varrho_\delta.
\end{split}
\end{equation*}
It follows that 
\begin{align}\label{Q_1v}
&Q_1v-\mb_2Q_1\varrho_\delta=R_\delta^\prime,\\
&R_\delta'=\gamma_1\partial_zv-(Q_1-Q_2)v_2+\mb_2(Q_1-Q_2)\varrho_2+\frac{\mb_2}{\p_z\varrho_1}Q_1\varrho_1\p_z\varrho_\delta.
\end{align}
We claim that
\begin{align}\label{est:Rdelta'}
\Vert R^\prime_\delta\Vert_{Y^{\sigma-1}([z_0,0])}\le\mathcal{F}(N_s)\Vert \eta_1-\eta_2\Vert_{H^{\sigma-\delta}}\Vert f\Vert_{\widetilde{H}^s_-}.
\end{align}
 We first apply \eqref{XY1} with $s_0=\sigma-1$, $s_1=s-2$, $s_2=\sigma-1-\delta$, giving
 \bq\label{est:sigma:cd1}
\| \gamma_1\partial_zv\|_{Y^{\sigma-1}}\les \|\gamma_1 \|_{X^{s-2}}\| \p_zv\|_{X^{\sigma-1-\delta}}\le \mathcal{F}(N_s)\Vert \eta_1-\eta_2\Vert_{H^{\sigma-\delta}}\Vert f\Vert_{\widetilde{H}^s_-}.
\eq
 By the same argument we can control $(Q_1-Q_2)\varrho_2$ and  $(Q_1-Q_2)v_2$ in $Y^{\sigma-1}$. For example, when distributing derivatives in the term $(\alpha_1-\alpha_2)\Delta_xv_2$ in $(Q_1-Q_2)v_2$, we see that  $\alpha_1-\alpha_2$ and $\Delta_xv_2$ respectively play the role of $\p_zv$ and $\gamma_1$ in the product $\gamma_1\p_zv$. 

For products we note that \eqref{XY2} gives 
\[
\| ab\|_{Y^{\sigma-1}}\les \|a\|_{X^{s-1}}\|b\|_{Y^{\sigma-1}}.
\]
Applying this for $(a, b)=(\mb_2, (Q_1-Q_2)\varrho_2)$ we obtain the control of $\mb_2(Q_1-Q_2)\varrho_2$. As for the last term in $R_\delta'$, we take $a=\frac{\mb_2}{\p_z\varrho_1}$ and $b=Q_1\varrho_1\p_z\varrho_\delta$ where applying \eqref{XY1} again yields
\[
\|Q_1\varrho_1\p_z\varrho_\delta\|_{Y^{\sigma-1}}\les \| Q_1\varrho_1\|_{X^{s-2}}\|\p_z\varrho_\delta\|_{X^{\sigma-1-\delta}}\le \mathcal{F}(N_s)\Vert \eta_1-\eta_2\Vert_{H^{\sigma-\delta}}\Vert f\Vert_{\widetilde{H}^s_-}.
\]
We thus conclude the proof of \eqref{est:Rdelta'}.

{\bf Step 2.}  Using  \eqref{decompose}, we factorize $Q_1v$ and $Q_1\varrho_\delta$ in \eqref{Q_1v}. Then we obtain the first equation in \eqref{DefR2Delta} with
\begin{equation*}
\begin{split}
R^1_\delta&=(\partial_z-T_{a_1})\left[(\partial_z-T_{A_1})v-T_{\mb_2}(\partial_z-T_{A_1})\varrho_\delta\right]\\
&= R_\delta^\prime+T_{Q_1\varrho_\delta}\mb_2+R(\mb_2, Q_1\varrho_\delta)-R_{Q_1}v+T_{\mb_2}R_{Q_1}\varrho_\delta+[T_{\mb_2},\partial_z-T_{a_1}](\partial_z-T_{A_1})\varrho_\delta.
\end{split}
\end{equation*}
In view of \eqref{cons:cdsigma}, applying  \eqref{est:XY4} and \eqref{est:Q1rhodelta} gives
\[
\| R(\mb_2, Q_1\varrho_\delta)\|_{L^1H^{\sigma-1}}\les\| Q_1\varrho_\delta\|_{Y^{\sigma-1-\delta}}\| \mb_2\|_{X^{s-1}} \le\mathcal{F}(N_s)\Vert \eta_1-\eta_2\Vert_{H^{\sigma-\delta}}\Vert f\Vert_{\widetilde{H}^s_-}.
\]
On the other hand,  by \eqref{est:XY3} and \eqref{est:Q1rhodelta},
\[
\| T_{Q_1\varrho_\delta}\mb_2\|_{Y^{\sigma-1}}\les\| Q_1\varrho_\delta\|_{Y^{\sigma-1-\delta}} \| \mb_2\|_{L^\infty H^{s-1}}\le\mathcal{F}(N_s)\Vert \eta_1-\eta_2\Vert_{H^{\sigma-\delta}}\Vert f\Vert_{\widetilde{H}^s_-}.
\]
Next we compute 
\[
[T_{\mb_2},\partial_z-T_{a_1}](\partial_z-T_{A_1})\varrho_\delta=-T_{\p_z\mb_2}(\partial_z-T_{A_1})\varrho_\delta-[T_{\mb_2},T_{a_1}](\partial_z-T_{A_1})\varrho_\delta.
\]
In view of \eqref{BdonB2} for $\|\p_z\mb_2\|_{X^{s-2}}$, we have 
\[
\begin{aligned}
\|T_{\p_z\mb_2}(\partial_z-T_{A_1})\varrho_\delta\|_{L^2 H^{\sigma-\tdm}}&\les \| \p_z\mb_2\|_{L^2 H^{s-\tdm}}\| (\partial_z-T_{A_1})\varrho_\delta\|_{L^\infty H^{\sigma-1-\delta}}\\
& \le\mathcal{F}(N_s)\Vert \eta_1-\eta_2\Vert_{H^{\sigma-\delta}}\Vert f\Vert_{\widetilde{H}^s_-}.
\end{aligned}
\]
On the other hand, the commutator $[T_{\mb_2},T_{a_1}](\partial_z-T_{A_1})\varrho_\delta$ can be controlled in $L^2 H^{\sigma-\tdm}$ upon using Theorem \ref{theo:sc} (ii).

It remains to control terms involving $R_{Q_1}$. Using  \eqref{cons:cdsigma} we see that $\tt=\sigma-1$ and $\delta$ satisfy \eqref{cd:tt-s}. Consequently, the estimate \eqref{est:RAa2}  in Lemma \ref{prop:decompose} can be applied, giving
\[
\|R_Qg\|_{Y^{\sigma-1}}\le \cF(\| \eta\|_{H^s})\|\na_{x,z}g\|_{X^{\sigma-1-\delta}}.
\]
Applying this with $g=v$ and taking into account \eqref{est:contractv:2} we deduce that $R_{Q_1}v$ is controllable.  Finally, with $g=\varrho_\delta$ and with the aid of  \eqref{est:XY22}, we have
 \[
 \| T_{\mb_2}R_{Q_1}\varrho_\delta\|_{Y^{\sigma-1}}\les \| \mb_2\|_{L^\infty H^{s-1}} \| R_{Q_1}\varrho_\delta\|_{Y^{\sigma-1}}  \le\mathcal{F}(N_s)\Vert \eta_1-\eta_2\Vert_{H^{\sigma-\delta}}\Vert f\Vert_{\widetilde{H}^s_-}.
\]
The proof of \eqref{DefR2Delta} is complete.
\end{proof}

\begin{proof}[Proof of Theorem \ref{theo:contractionDN2}]

As in the proof of Theorem \ref{tame:RDN}, let us fix $-1<z_0<z_1<0$ and introduce a cut-off $\chi$ satisfying $\chi(z)=1$ for $z>z_1$ and $\chi(z)=0$ for $z<z_0$. It follows from Lemma \ref{lemm:F2analog} that
\[
w_\delta :=\chi(z)\big[(\p_z-T_{A_1})v-T_{\mb_2}(\p_z-T_{A_1})\varrho_\delta \big]
\]
satisfies
\[
(\p_z-T_{a_1})w_\delta=R^2_\delta:=\chi(z)R^1_\delta+\chi'(z)\big[(\p_z-T_{A_1})v-T_{\mb_2}(\p_z-T_{A_1})\varrho\big].
\]
Applying Theorem \ref{theo:sc} (i) and \eqref{est:contractv:2} we have
\[
\begin{aligned}
\| (\p_z-T_{A_1})v\|_{L^2([z_0, 0]; H^{\sigma-\frac{3}{2}})}&\le    \mathcal{F}(N_s)\| \na_{x,z}v\|_{L^2([z_0, 0]; H^{\sigma-\tdm})}\\
&\le   \mathcal{F}(N_s)\| \na_{x,z}v\|_{L^2([z_0, 0]; H^{\sigma-\mez-\delta})}\\
&\le  \mathcal{F}\big(N_s\big)\| \eta_1-\eta_2\|_{H^{\sigma-\delta}}\Vert f\Vert_{\widetilde{H}^s_-}.
\end{aligned}
\]
In addition, \eqref{boundpara} together with \eqref{BdonB2} implies 
\[
\begin{aligned}
\| T_{\mb_2}(\p_z-T_{A_1})\varrho_\delta\|_{L^2([z_0, 0]; H^{\sigma-\frac{3}{2}})}&\les \| \mb_2\|_{L^\infty([z_0, 0]; H^{s-1})}\| (\p_z-T_{A_1})\varrho_\delta\|_{L^2([z_0, 0]; H^{\sigma-\frac{3}{2}})}\\
&\le\mathcal{F}\big(N_s\big)\| f\|_{\wt H^s_{-}}\| \na_{x,z}\varrho_\delta\|_{L^2([z_0, 0]; H^{\sigma-\frac{3}{2}})}\\
&\le  \mathcal{F}\big(N_s\big)\| \eta_1-\eta_2\|_{H^{\sigma-\delta}}\| f\|_{\wt H^s_{-}}.
\end{aligned}
\]
Thus,  $R^2_\delta$ satisfies similar estimates as $R^1_\delta$ in \eqref{DefR2Delta}. Since $w_\delta(z_0)=0$, applying Proposition \ref{prop:parabolic} yields
\begin{equation}\label{BoundonWDelta}
\| w_\delta\|_{X^{\sigma-1}([z_0, 0])}\le   \mathcal{F}\big(N_s\big)\| \eta_1-\eta_2\|_{H^{\sigma-\delta}}\| f\|_{\wt H^s_{-}}.
\end{equation}
In the rest of this proof, functions of $(x, z)$ are evaluated at $z=0$. Besides, we write $g_1\simeq g_2$ to signify that $g_1$ and $g_2$ agree up to acceptable errors,
\begin{equation*}
\Vert g_1-g_2\Vert_{H^{\sigma-1}}\lesssim\mathcal{F}(N_s)\Vert \eta_1-\eta_2\Vert_{H^{\sigma-\delta}}\Vert f\Vert_{\widetilde{H}^s_-}.
\end{equation*}
Set
\[
p_j=\frac{1+|\na_x\varrho_j|^2}{\p_z\varrho_j}=\frac{\partial_z\varrho_j}{\alpha_j},\quad j=1, 2
\]
so that, by \eqref{EstimRho12},
\begin{equation}\label{Bdsp1q1}
\begin{split}
&\Vert p_1\Vert_{H^{s-1}}+\Vert p_2\Vert_{H^{s-1}}+\Vert \nabla_{x,z}\varrho_1\Vert_{H^{s-1}}+\Vert \nabla_{x,z}\varrho_2\Vert_{H^{s-1}}\le\mathcal{F}(N_s)\\
&\Vert p_1-p_2\Vert_{H^{\sigma-1-\delta}}+\Vert \nabla_{x,z}\varrho_\delta\Vert_{H^{\sigma-1-\delta}}\le \mathcal{F}(N_s) \Vert \eta_1-\eta_2\Vert_{H^{\sigma-\delta}}.
\end{split}
\end{equation}
Using \eqref{ExpressionDNNew} and the fact that $v\vert_{z=0}\equiv0$, we write 
\begin{equation}\label{DiffDNSenegal1}
\begin{split}
G^-(\eta_1)f-G^-(\eta_2)f&=(p_1-p_2)\partial_zv_2+p_1\partial_zv-\nabla_x\varrho_\delta\cdot\nabla_xv_2\\
&=(p_1-p_2)\partial_zv_2-\nabla_x\varrho_\delta\cdot\nabla_xv_2+p_1\left[T_{A_1}v+T_{\mb_2}(\partial_z-T_{A_1})\varrho_\delta\right]+p_1w_\delta,
\end{split}
\end{equation}
where $T_{A_1}v=0$ (at $z=0$). Using this,   \eqref{EstimV_jNew},   \eqref{BoundonWDelta}, \eqref{Bdsp1q1} with Theorem \ref{theo:sc} (i), \eqref{boundpara}, \eqref{paralin:product}, we find that
\begin{equation*}
\begin{split}
G^-(\eta_1)f-G^-(\eta_2)f&\simeq T_{\partial_zv_2}(p_1-p_2)-T_{\nabla_xv_2}\nabla_x\varrho_\delta +T_{p_1}T_{\mb_2}(\partial_z-T_{A_1})\varrho_\delta,
\end{split}
\end{equation*}
where by virtue of Theorem \ref{theo:sc} (ii), 
\[
T_{p_1}T_{\mb_2}(\partial_z-T_{A_1})\varrho_\delta \simeq T_{\mb_2}T_{p_1}(\partial_z-T_{A_1})\varrho_\delta.
\]
Next applying Theorem \ref{paralin:nonl} we find that
\begin{equation*}
\begin{split}
T_{\p_zv_2}(p_1-p_2)&\simeq T_{2\mb_2\na_x\varrho_2}\cdot\na_x\varrho_\delta-T_{p_2\mb_2}\p_z\varrho_\delta.
\end{split}
\end{equation*}
We thus arrive at
\begin{equation*}
\begin{split}
G^-(\eta_1)f-G^-(\eta_2)f&\simeq -T_{\mb_2}(T_{p_1A_1}-T_{\nabla_x\varrho_2}\nabla_x)\varrho_\delta+T_{\mb_2\nabla_x\varrho_2-\nabla_xv_2}\nabla_x\varrho_\delta+T_{\mb_2}T_{p_1-p_2}\partial_z\varrho_\delta.
\end{split}
\end{equation*}
By \eqref{form:ld}, we have
\begin{equation*}
\begin{split}
&p_1A_1-p_2A_2=\lambda_1-\lambda_2-i\xi\cdot\nabla_x\varrho_\delta=\widetilde{A}\cdot\nabla_x\varrho_\delta,\\
&M^0_\delta(\nabla_x\varrho_\delta)+M^1_\delta(\widetilde{A})\lesssim \mathcal{F}(N_s).
\end{split}
\end{equation*}
Theorem \ref{theo:sc} (i) and \eqref{boundpara} yield that
\begin{equation*}
\begin{split}
&\Vert T_{p_1-p_2}\partial_z\varrho_\delta\Vert_{H^{\sigma-1}}\lesssim \Vert\p_z\varrho_\delta\Vert_{H^{s-1}}\Vert p_1-p_2\Vert_{H^{\sigma-1-\delta}}\lesssim\mathcal{F}(N_s)\Vert \eta_1-\eta_2\Vert_{H^{\sigma-\delta}},\\
&\Vert T_{p_1A_1-p_2A_2}\varrho_\delta\Vert_{H^{\sigma-1}}\le\mathcal{F}(N_s)\Vert \varrho_\delta\Vert_{H^{\sigma-\delta}}+\Vert T_{\nabla_x\varrho_\delta}T_{\widetilde{A}}\varrho_\delta\Vert_{H^{\sigma-1}}\les\mathcal{F}(N_s)\Vert \varrho_\delta\Vert_{H^{\sigma-\delta}}.
\end{split}
\end{equation*}
We conclude that
\begin{equation*}
\begin{split}
G^-(\eta_1)f-G^-(\eta_2)f\simeq -T_{\lambda_2B_2}(\eta_1-\eta_2)+T_{V_2}\cdot\nabla(\eta_1-\eta_2)
\end{split}
\end{equation*}
which finishes the proof of Theorem \ref{theo:contractionDN2}. 
\end{proof}
For future reference, let us end this subsection by providing a variant of Corollary \ref{theo:contractionDN}.
\begin{prop}\label{prop:contraDN100}
Let  $s>1+\frac{d}{2}$ with $d\ge 1$. Consider  $\eta_1$, $\eta_2\in H^{s}(\Rr^d)$ with $\dist(\eta_j, \Gamma^-)> 4h>0$ for $j=1, 2$. For all $\sigma\in [\mez, s]$, there exists $\cF:\Rr^+\to \Rr^+$ depending only on $(s, \sigma, h)$ such that
\begin{equation}\label{contraDN:100}
\| G^{-}(\eta_1)f-G^{-}(\eta_2)f  \|_{H^{\sigma-1}}\le 
\mathcal{F}\big(\| (\eta_1, \eta_2) \|_{H^s}\big)\| \eta_1-\eta_2\|_{H^s}\| f\|_{\wt H^\sigma_{-}}.
\end{equation}
\end{prop}
\begin{proof}
   We  shall use the notation in the proof of Theorem \ref{theo:contractionDN2}. In our setting, we can strengthen \eqref{EstimRho12} to
\begin{equation}\label{EstimRho12NewSenegal}
\Vert \varrho_\delta\Vert_{X^s}\lesssim \Vert \eta_1-\eta_2\Vert_{H^s},
\end{equation}
but as $f\in \wt H^\sigma_-$, in place of \eqref{EstimV_jNew} we have
\bq\label{navj:100}
\Vert \nabla_{x,z}v_j\Vert_{X^{\sigma-1}([-1,0])}\lesssim \mathcal{F}(N_s)\Vert f\Vert_{\widetilde{H}^\sigma_-},\quad N_s=\| (\eta_1, \eta_2)\|_{H^s}.
\eq
Recall that $v=v_1-v_2$ solves \eqref{eq:vdiff}. Upon using the product rule \eqref{pr} and proceeding as in the proof of Corollary \ref{theo:contractionDN}, \eqref{contraDN:100} follows easily from the following estimate for $v$
\bq\label{contrav:100}
\| \na_{x, z}v\|_{X^{\sigma-1}([z_0, 0])}\le \cF(N_s)\| \eta_1-\eta_2\|_{H^s}\| f\|_{\wt H^\sigma_{-}},\quad z_0\in (-1, 0).
\eq
To prove \eqref{contrav:100}, we apply Proposition \ref{prop:elliptic} to have
\[
\lA \nabla_{x,z} v\rA_{X^{\sigma-1}([z_0,0])}
\le \mathcal{F}(N_s)\big(\lA F\rA_{Y^{\sigma-1}([z_1, 0])}
+ \lA \nabla_{x,z} v\rA_{X^{-\mez}([-1,0])} \big),
\]
  where $-1<z_1<z_0<0$. Let us estimate each term on the right-hand side. We claim that 
  \[
  \| \na_{x,z}v\|_{X^{-\mez}([-1, 0])}\le \cF(N_s)\| \eta_1-\eta_2\|_{H^s}\Vert f\Vert_{\widetilde{H}^\sigma_-}.
  \]
  This follows along the same lines as in the proof of Lemma \ref{lemm:contractlow} except that for the right-hand side of \eqref{variform:fixed:2}, in place of \eqref{est:diffA} we estimate 
\[
\begin{split}
&\| \mathcal{A}_1-\mathcal{A}_2\|_{L^\infty(\Rr^d\times (-1, 0))}\les\| \mathcal{A}_1-\mathcal{A}_2\|_{L^\infty((-1, 0); H^{s-1})}\les \cF(N_s)\|\eta_1-\eta_2\|_{H^s},\\
&\Vert \na_{x, z}v_2\Vert_{L^2(\mathbb{R}^d\times (-1,0))}\les \Vert \na_{x, z}v_2\Vert_{L^2((-1,0); H^{\sigma-\mez})}\les \cF(N_s)\Vert f\Vert_{\widetilde{H}^\sigma_-}.
\end{split}
\]
 It remains to estimate $\lA F\rA_{Y^{s-\tdm}([z_1, 0])}$. The first term in $F$ can be bounded using \eqref{XY1} as
\[
\begin{aligned}
\| (\alpha_1-\alpha_2)\Delta_xv_2\|_{Y^{\sigma-1}}&
\les \| \alpha_1-\alpha_2\|_{X^{s-1}}\|\Delta_xv_2\|_{X^{\sigma-2}}\\
&\les \cF(\| \eta\|_{H^s})\| \eta_1-\eta_2\|_{H^s}\| f\|_{\wt H^\sigma_-}.
\end{aligned}
\]
Applying \eqref{XY1} again gives
 \[
\| (\gamma_1-\gamma_2)\p_zv_2\|_{Y^{\sigma-1}}\les \| \gamma_1-\gamma_2\|_{X^{s-2}}\| \p_zv_2\|_{X^{\sigma-1}}\les \cF(\| \eta\|_{H^s})\| \eta_1-\eta_2\|_{H^s}\| f\|_{\wt H^\sigma_-}.
\]
The proof is complete.
\end{proof}

\section{Proof of the main theorems}\label{SecMainProof}
\subsection{Proof of Theorem \ref{theo:wp1}}
Let us fix  $s>1+\frac d2$ and consider  either $\Gamma^-=\emptyset$ or $\underline{b}^-\in \dot W^{1, \infty}(\Rr^d)$.  This section is organized as follows. First, we assume that $\eta$ is a solution of \eqref{eq:eta} on $[0, T]$ such that
\begin{align}\label{reg:eta}
&\eta\in C([0, T]; H^s)\cap L^2([0, T]; H^{s+\mez}),\\
\label{apriori:h}
&\dist(\eta(t), \Gamma^-)\ge h\quad\forall t\in [0, T],\\
\label{lower:1-B}
&\inf_{x\in \Rr^d}(1-B(x, t))\ge \ma>0\quad\forall t\in [0, T],
\end{align} 
where $B$ is given by \eqref{BV} with $f=\eta$. Under these assumptions, a priori estimates are derived in Subsections \ref{subsection:reduction}, \ref{subsection:parabolicity}, \ref{section:energyest} and \ref{section:propaRT}. These estimates will be used for solutions of \eqref{eq:eta:app} which is similar to \eqref{eq:eta} with $\partial_t\to\partial_t-\varepsilon\Delta$. This modification only improves the equation and for simplicity, we perform the analysis on solutions of the simpler equation \eqref{eq:eta}. Finally, the proof of Theorem \ref{theo:wp1} is given in Subsection \ref{subsection:proofwp1}.
 
\subsubsection{Paradifferential reduction}\label{subsection:reduction}
We first apply Theorem \ref{tame:RDN} with $f=\eta$ and $\sigma=s$ to have
\bq\label{eta:lin}
\p_t\eta=-\ka T_\lambda(\eta-T_B\eta)+\ka T_V\cdot\nabla \eta-\ka R^-(\eta)\eta
\eq
where 
 $R^-(\eta)\eta$ obeys the estimate \eqref{est:RDN} with $\sigma=s$
\bq\label{ap:R1}
\| R^-(\eta)\eta\|_{H^{s-\mez}}\le \cF(\| \eta\|_{H^s})\| \eta\|_{H^s}\big(1+\| \eta\|_{H^{s+\mez-\delta}}\big)
\eq
where  $\delta \in  (0, \mez]$ satisfies $\delta<s-1-\frac d2$. Recall that $V$ and  $B$ can be expressed in terms of $\eta$ by virtue of the formulas \eqref{BV} with $f=\eta$. Note that
\bq\label{M1:ld}
 M^1_\delta (\lambda)\le  \cF(\| \eta\|_{H^s}),\quad \| B\|_{W^{\delta, \infty}}\le C\| B\|_{H^{s-1}}\le  \cF(\| \eta\|_{H^s}).
\eq
Owing to Theorem \ref{theo:sc} (ii),  $T_\ld T_B-T_{\ld B}$ is of order $1-\delta$ and
\bq\label{ap:R2}
\| (T_\ld T_B-T_{\ld B})\eta\|_{H^{s-\mez}}\le  \cF(\| \eta\|_{H^s})\| \eta\|_{H^{s+\mez-\delta}}.
\eq
Combining \eqref{eta:lin}, \eqref{ap:R1} and \eqref{ap:R2} we arrive at the following paradifferential reduction for the one-phase Muskat problem \eqref{eq:eta}.
\begin{prop}
  For $\delta \in (0, \mez]$ satisfying $\delta<s-1-\frac d2$, there exists $\cF:\Rr^+\to \Rr^+$  depending only on $(s, \delta, h, \kappa)$ such that
\bq\label{eta:reduce}
\p_t\eta-\ka T_V\cdot\nabla \eta+\ka T_{\lambda (1-B)}\eta=f
\eq
with $f$ satisfying
\bq\label{est:RHS}
\| f\|_{H^{s-\mez}}\le \cF(\| \eta\|_{H^s})\| \eta\|_{H^s}\big(1+\| \eta\|_{H^{s+\mez-\delta}}\big).
\eq
\end{prop}

\subsubsection{Parabolicity}\label{subsection:parabolicity}
In equation \eqref{eta:reduce}, $T_{V}\cdot \nabla \eta$ is an advection term. We now prove that $T_{\ka \lambda(1-B)}$ is an elliptic operator, showing that \eqref{eta:reduce} is a first-order drift-diffusion equation.
\begin{lemm}\label{prop:1-B}
For any $t\in [0, T]$ we have
\bq\label{parabolicity}
\sup_{x\in \Rr^d}B(x, t)<1.
\eq
\end{lemm}
In view of the formula \eqref{BV} for $B$, Lemma \ref{prop:1-B} is a direct consequence of the following surprising upper bound.
\begin{prop}\label{lemm:upperbounG}
 Assume that either $\Gamma^-=\emptyset$ or $\underline b^-\in \dot W^{1, \infty}(\Rr^d)$. If $f\in H^s(\Rr^d)$ with $s>1+\frac d2$, $d\ge 1$, then there exists $c_0 >0$ such that
 \bq\label{ubG}
G^-(f)f(x)<1-c_0\quad\forall x\in \Rr^d.
\eq
\end{prop}
\begin{proof}
 Let $\Omega^-$ denote the fluid domain with the top boundary $\Sigma=\{y=f(x)\}$ and the bottom $\Gamma^-$. 
 
1. {\it Finite depth}. According to Proposition \ref{prop:vari1}, there exists a unique solution $\phi\in \dot H^1(\Omega^-)$ to the problem
  \bq\label{eq:ellipticRT}
\begin{cases}
\Delta_{x, y} \phi=0\quad\text{in}~\Omega^-,\\
\phi=f\quad\text{on}~\Sigma,\\
\frac{\p\phi}{\p \nu^-}=0\quad\text{on}~\Gamma^-
\end{cases}
\eq
 in the sense of \eqref{variform:u}
 \bq\label{variform:RT}
\int_{\Omega^-}\na_{x, y} \phi\cdot \na_{x, y} \varphi dxdy=0\quad\forall \varphi \in  H^1_{0, *}(\Omega^-).
\eq 
Inserting $\varphi=\min\{\phi- \inf_{\Rr^d}f, 0\}\in H^1_{0, *}(\Omega^-)$ into \eqref{variform:RT} we obtain the minimum principle 
\bq\label{min:principle}
\inf_{\Omega^-}\phi\ge \inf_{\Rr^d}f.
\eq
Consequently,
\bq\label{psi:below}
\psi(x, y):=\phi(x, y)-y\ge \inf_{\Rr^d}f-y>0
\eq
for $(x, y)\in \Omega^-$ with $y\le k:=(\inf_{\Rr^d}f)-1$. We claim that $\psi$ is also nonnegative elsewhere: 
\bq\label{psi:above}
\psi(x, y)\ge 0\quad \text{in}\quad \Omega_{\wt b}:=\{(x, y)\in \Rr^d\times \Rr: \wt b(x)<y<f(x) \},
\eq
where $\wt b(x)=\max\{\underline b^-(x), k\}$ is a Lipschitz and bounded function, $\wt b\in W^{1, \infty}(\Rr^d)$. 
 Let $\chi:\Rr^d\to \Rr^+$ be a compactly  function that equals $1$ in $B(0, 1)$ and vanishes outside $B(0, 2)$. Then consider the test functions $\varphi_n =\psi^-(x, y)\chi(\frac{x}{n})\le 0$ where $\psi^-= \min\{\psi, 0\}$ and $n\ge 1$.  By \eqref{psi:below}, $\supp\psi^-\subset \overline{\Omega_{\wt b}}$ and thus
 \[
 \supp\varphi_n\subset \overline{\Omega_{\wt b, n}},\quad \Omega_{\wt b, n}:=\Omega_{\wt b}\cap \{ |x|<2n\}
 \]
which gives $\varphi_n\in H^1_{0, *}(\Omega^-)$. Replacing $\phi$ with $\psi+y$ and $\varphi$ with $\varphi_n$ in \eqref{variform:RT} gives 
  \[
 \begin{aligned}
\int_{\Omega^-}\na_{x, y} \psi\cdot \na_{x, y} \varphi_n dxdy&=-\int_{\Omega_{\wt b}\cap \{|x|<2n\}}\p_y\varphi_n  dxdy\\
&=-\int_{\{|x|<2n\}}\int_{\wt b(x)}^{f(x)}\p_y\varphi_n(x, y)dydx\\
&=\int_{\{|x|<2n\}}\varphi_n(x, \wt b(x))dx\le 0.
\end{aligned}
\]
 On the other hand,
\[
\int_{\Omega^-}\na_{x, y} \psi\cdot \na_{x, y} \varphi_n dxdy= \int_{U_n}|\na_{x, y} \psi|^2\chi(\frac{x}{n})dxdy+\frac{1}{n} \int_{U_n}\psi\na_x \psi \cdot (\na_x\chi)(\frac{x}{n}) dxdy,
\]
where $U_n=\{(x, y)\in \Omega_{\wt b, n}: \psi(x, y)< 0\}$. Thus,
\bq\label{approx:max}
 \int_{U_n}|\na_{x, y} \psi|^2\chi(\frac{x}{n})dxdy\le \frac{-1}{n} \int_{U_n}\psi\na_x \psi \cdot (\na_x\chi)(\frac{x}{n}) dxdy\quad\forall n\ge 1.
\eq
Since
\[
\psi(x, y)=-\int_y^{f(x)}\p_y\phi(x, y')dy'+f(x)-y
\]
and $f(x)-y\ge 0$, we deduce that if $\psi(x, y)<0$ then 
\[
|\psi(x, y)|=\la \int_y^{f(x)}\p_y\phi(x, y')dy'\ra-|f(x)-y| \le L^\mez\Big(\int_y^{f(x)}|\p_y\phi(x, y')|^2dy'\Big)^\mez
\]
where $L=\| f-\wt b\|_{L^\infty(\Rr^d)}<\infty$. In particular, for $(x, y)\in U_n\subset \Omega_{\wt b, n}$ we have
\[
\int_{U_n}|\psi(x, y)|^2dxdy\le L\int_{\Omega_{\wt b, n}}\int_y^{f(x)}|\p_y\phi(x, y')|^2dy'dxdy\le L^2 \int_{\Omega_{\wt b, n}}|\p_y\phi(x, y')|^2dxdy'.
\]
This combined with the fact that $\na_x\psi=\na_x\phi$ yields
\[
\begin{aligned}
\la \frac{1}{n} \int_{U_n}\psi\na_x \psi \cdot (\na_x\chi)(\frac{x}{n}) dxdy\ra& \le \frac{L\| \na\chi\|_{L^\infty(\Rr^d)}}{n}\| \na_x\phi\|_{L^2(\Omega_{\wt b, n})}\| \p_y\phi\|_{L^2(\Omega_{\wt b, n})} \\
&\le \frac{L\| \na\chi\|_{L^\infty(\Rr^d)}}{n}\| \na_{x, y}\phi\|_{L^2(\Omega^-)}^2
\end{aligned}
\]
which tends to $0$ as $n\to \infty$. Then applying  the monotone convergence theorem to \eqref{approx:max}, we arrive that 
\[
\int_{\{(x, y)\in \Omega_{\wt b}: \psi(x, y)<0\}}|\na_{x, y} \psi|^2dxdy\le 0
\]
which proves that $\psi\ge 0$ in $\Omega_{\wt b}$ as claimed in \eqref{psi:above}. Combining \eqref{psi:below} and \eqref{psi:above} we conclude that $\psi\ge 0$ in $\Omega^-$. 

Now by virtue of Proposition \ref{prop:elliptic}, $\phi\in C^1_b(\Omega^-_h)$ where
\[
\Omega^-_h=\{(x, y)\in \Rr^d\times \Rr: \eta(x)-h<y<\eta(x)\}
\]
for some $h>0$. Consequently, $\psi\in C^1_b(\Omega^-_h)$, $\psi\ge 0$ everywhere in $\Omega^-_h$ and $\psi\vert_\Sigma=0$. The infimum of $\psi$ over $\overline{\Omega^-_h}$ is thus $0$ and is attained at any points of $\Sigma$; moreover, $\psi<0$ in $\Omega^-_h$ thanks to the strong maximum principle.  On one hand, it follows from Theorem \ref{theo:estDN} that $G^-(f)f\in H^{s-1}(\Rr^d)$ with $s-1>\frac{d}{2}$, implying that  $G(f)f(x)<\mez$ for all $|x|\ge M$ for $M$ sufficiently large. On the other hand, letting $V=\Omega_h^-\cap \{|x|<2M\}$, we can apply Hopf's lemma (see \cite{Safo}) to the $C^{1, \alpha}$ boundary $\Sigma\cap\{|x|< 2M\}$ of $V$ to have $\frac{\p \psi}{\p n}<0$ for $|x|\le M$, where $n$ is the upward-pointing normal to $\Sigma$. Hence,
\[
\frac{\p \phi}{\p n}<\frac{\p y}{\p n}=\frac{1}{\sqrt{1+|\na f|^2}}
\]
which yields $G^-(f)f(x)=\sqrt{1+|\na f|^2}\frac{\p \phi}{\p n}<1$ for all $|x|\le M$. By the continuity of $G^-(f)f$ on $\{|x|\le M\}$, we conclude that $G^-(f)f(x)\le 1-c_0$ for some $c_0>0$ and for all $x\in \Rr^d$. 

2. {\it Infinite depth}.  
The proof for this case is in fact simpler. We first let $\phi\in \dot H^1(\Omega^-)$ be the solution in the sense of Proposition \ref{prop:vari2} to the problem 
  \bq\label{eq:ellipticRT2}
\begin{cases}
\Delta_{x, y} \phi=0\quad\text{in}~\Omega^-,\\
\phi=f\quad\text{on}~\Sigma,\\
\na_{x, y}\phi\to 0\quad\text{as}~y\to -\infty;
\end{cases}
\eq
that is, \eqref{variform:RT} holds. The minimum principle \eqref{min:principle} remains valid, implying \eqref{psi:below}. Then we can proceed as in the previous case upon replacing $\Omega_{\wt b}$ with $\{(x, y)\in \Rr^d\times \Rr: \eta(x)-k<y<\eta(x)\}$.
\end{proof}
\begin{rema}
The proof of Proposition \ref{lemm:upperbounG} is simpler for the periodic case $x\in \T^d$. Indeed, when $x\in \T^d$ we have $\psi^-\in \dot H^1(\Omega^-)$ and thus localization in $x$ by $\chi(\frac{x}{n})$ is not needed.
\end{rema}
\begin{rema}
The one-phase problem \eqref{eq:eta} dissipates the energy $E(t)=\mez \| \eta(t)\|^2_{L^2}$ since
\[
\mez\|\eta(t)\|^2_{L^2}-\mez\|\eta(0)\|^2_{L^2}=-\ka \int_0^t (G^-(\eta)\eta, \eta)_{L^2(\Rr^d)}dr\le 0.
\]
By virtue of the upper bound \eqref{ubG}, if $\eta(t)$ remains nonnegative on $[0, T]$ then the energy dissipation over $[0, T]$ is bounded by the $L^1$ norm of $\eta$:
\[
\ka \int_0^T (G^-(\eta)\eta, \eta)_{L^2(\Rr^d)}dr\le \ka\int_0^T\| \eta(r)\|_{L^1(\Rr^d)}dr.
\]
Note that this bound is linear, while the energy is quadratic. In the case of constant viscosity, the same bound was proved in \cite{ConCorGanStr} without the sign condition on $\eta$. 
\end{rema}
\subsubsection{A priori estimates for $\eta$}\label{section:energyest}
Denote $\langle D_x\rangle=(1+|D_x|^2)^\mez$ and set $\eta_s=\langle D_x\rangle^s$. Conjugating the paradifferential equation \eqref{eta:reduce} with $\langle D_x\rangle$ gives
\bq\label{eq:etas}
\p_t\eta_s=\ka T_V\cdot\nabla\eta_s-\ka T_{\ld(1-B)}\eta_s+f_1
\eq
where 
\[
f_1=\ka[\langle D_x\rangle^s, T_V\cdot \nabla]\eta-\ka[\langle D_x\rangle^s, T_{\ld(1-B)}]\eta+\langle D_x\rangle^s f.
\]
As in \eqref{M1:ld} we have $\| V\|_{W^{\delta, \infty}}\le\cF(\| \eta\|_{H^s})$. This combined with \eqref{est:RHS}, \eqref{M1:ld} and Theorem \ref{theo:sc} (iii) implies that $[\langle D_x\rangle^s, T_V\cdot \nabla]$ and $[\langle D_x\rangle^s, T_{\ld(1-B)}]\eta$ are of order $s+1-\delta$, and that
\bq\label{energy:f1}
\| f_1\|_{H^{-\mez}}\le  \cF(\| \eta\|_{H^s})\| \eta\|_{H^s}\big(1+\| \eta\|_{H^{s+\mez-\delta}}\big).
\eq
Taking the $L^2$ inner product of \eqref{eq:etas} with $\eta_s$ gives
\bq\label{dt:energy}
\mez\frac{d}{dt}\|\eta_s\|_{L^2}^2=\ka (T_V\cdot\nabla\eta_s, \eta_s)_{L^2}-\ka (T_{\ld(1-B)}\eta_s, \eta_s)_{L^2}+(f_1, \eta_s)_{L^2}.
\eq
We have
\[
\ka (T_V\cdot\nabla\eta_s, \eta_s)_{L^2}=\frac{\ka}{2} \Big(\big(T_V\cdot\nabla+(T_V\cdot\nabla)^*\big)\eta_s, \eta_s\Big)_{L^2}
\]
where by virtue of Theorem \ref{theo:sc} (iii), $T_V\cdot\nabla+(T_V\cdot\nabla)^*$ is of order $1-\delta$ and 
\[
\| \big((T_V\cdot\nabla)^*+T_V\cdot\nabla\big)\eta_s\|_{H^{-\mez+\delta}}\le \cF(\| \eta\|_{H^s})\| \eta_s\|_{H^\mez}.
\]
Consequently,
\bq\label{energy:advec}
\ka |(T_V\cdot\nabla\eta_s, \eta_s)_{L^2}|\le \cF(\| \eta\|_{H^s})\| \eta_s\|_{H^\mez} \| \eta_s\|_{H^{\mez-\delta}}.
\eq
Next we write
\bq\label{energy:pos}
\begin{aligned}
 (T_{\ld(1-B)}\eta_s, \eta_s)_{L^2}&=(T_{\sqrt{\omega}}\eta_s, T_{\sqrt{\omega}}\eta_s)_{L^2}+\Big(T_{\sqrt{\omega}}\eta_s, \big((T_{\sqrt{\omega}})^*-T_{\sqrt{\omega}}\big)\eta_s\Big)_{L^2}\\
 &\qquad+\Big(\big(T_\omega-T_{\sqrt{\omega}}T_{\sqrt{\omega}}\big)\eta_s, \eta_s\Big)_{L^2}
 \end{aligned}
\eq
where $\omega=\ld(1-B)$. In view of \eqref{ld} and \eqref{lower:1-B} we have $\omega \ge \ma|\xi|$,  hence
\[
M^\mez_\delta(\sqrt\omega)\le \cF(\| \eta\|_{H^s}, \frac{1}{\ma}).
\]
According to Theorem \ref{theo:sc} (ii) and (iii), $(T_{\sqrt{\omega}})^*-T_{\sqrt{\omega}}$ and $T_\omega-T_{\sqrt{\omega}}T_{\sqrt{\omega}}$ are of order $\mez-\delta$ and $1-\delta$ respectively. Thus
\bq\label{energy:e1}
\begin{aligned}
\left|\Big(T_{\sqrt{\omega}}\eta_s, \big((T_{\sqrt{\omega}})^*-T_{\sqrt{\omega}}\big)\eta_s\Big)_{L^2}\right|&\le \| T_{\sqrt{\omega}}\eta_s\|_{H^{-\delta}}\| \big((T_{\sqrt{\omega}})^*-T_{\sqrt{\omega}}\big)\eta_s\|_{H^\delta}\\
&\le  \cF(\| \eta\|_{H^s}, \frac{1}{\ma})\| \eta_s\|_{H^{\mez-\delta}}\| \eta_s\|_{H^\mez}
\end{aligned}
\eq
and 
\bq\label{energy:e2}
\begin{aligned}
\left|\Big(\big(T_\omega-T_{\sqrt{\omega}}T_{\sqrt{\omega}}\big)\eta_s, \eta_s\Big)_{L^2}\right|&\le \| \big(T_\omega-T_{\sqrt{\omega}}T_{\sqrt{\omega}}\big)\eta_s\|_{H^{-\mez+\delta}}\| \eta_s\|_{H^{\mez-\delta}}\\
&\le  \cF(\| \eta\|_{H^s}, \frac{1}{\ma})\| \eta_s\|_{H^\mez}\| \eta_s\|_{H^{\mez-\delta}}.
\end{aligned}
\eq
In addition, Theorem \ref{theo:sc} (iii) gives that $T_{\sqrt{\omega}^{-1}}T_{\sqrt\omega}-\text{Id}$ is of order $-\delta$ and that
\begin{align*}
\| \eta_s\|_{H^\mez}&\le \| T_{\sqrt{\omega}^{-1}}T_{\sqrt\omega} \eta_s\|_{H^\mez}+\cF(\| \eta\|_{H^s}, \frac{1}{\ma})\| \eta_s\|_{H^{\mez-\delta}}\\
&\le \cF(\| \eta\|_{H^s}, \frac{1}{\ma})\Big(\|T_{\sqrt\omega} \eta_s\|_{L^2}+\| \eta_s\|_{H^{\mez-\delta}}\Big),
\end{align*}
whence 
\bq\label{energy:e3}
\|T_{\sqrt\omega} \eta_s\|^2_{L^2}\ge \frac{1}{\cF(\| \eta\|_{H^s}, \frac{1}{\ma})}\| \eta_s\|^2_{H^\mez}-\| \eta_s\|_{H^{\mez-\delta}}^2.
\eq
Combining \eqref{energy:pos}, \eqref{energy:f1}, \eqref{energy:e1}, \eqref{energy:e2}, and \eqref{energy:e3} leads to 
\bq\label{energy:dif}
-\ka  (T_{\ld(1-B)}\eta_s, \eta_s)_{L^2}\le -\frac{1}{\cF(\| \eta\|_{H^s}, \frac{1}{\ma})}\| \eta_s\|_{H^\mez}^2+ \cF(\| \eta\|_{H^s}, \frac{1}{\ma})\| \eta_s\|_{H^\mez}\| \eta_s\|_{H^{\mez-\delta}}.
\eq
Putting together \eqref{dt:energy}, \eqref{energy:advec}, and \eqref{energy:dif} we obtain
\bq
\mez\frac{d}{dt}\|\eta_s\|_{L^2}^2\le -\frac{1}{\cF(\| \eta\|_{H^s}, \frac{1}{\ma})}\| \eta_s\|_{H^\mez}^2+ \cF(\| \eta\|_{H^s}, \frac{1}{\ma})\Big(\| \eta_s\|_{H^\mez}\| \eta_s\|_{H^{\mez-\delta}}+\| \eta_s\|_{H^\mez}\| \eta_s\|_{L^2}\Big).
\eq
We assume without loss of generality that $\delta\le \mez$. The gain of $\delta$ derivative gives room to  interpolate 
\[
\| \eta_s\|_{H^\mez}\| \eta_s\|_{H^{\mez-\delta}}\le C\| \eta_s\|_{H^\mez}^{1+\mu}\| \eta_s\|_{L^2}^{1-\mu}
\]
for some $\mu\in (0, 1)$. Applying Young's inequality yields
\[
\mez\frac{d}{dt}\|\eta_s\|_{L^2}^2\le -\frac{1}{\cF(\| \eta\|_{H^s}, \frac{1}{\ma})}\| \eta_s\|_{H^\mez}^2+\cF(\| \eta\|_{H^s}, \frac{1}{\ma})\| \eta_s\|_{L^2}^2
\]
where $\cF$ depends only on $(s, h, \ka)$. Finally, using Gr\"onwall's lemma we obtain the following a priori estimate for $\eta$. 
\begin{prop}\label{prop:apriori}
Let $s>1+\frac{d}{2}$. Assume that $\eta$ is a solution of \eqref{eq:eta} on $[0, T]$ with the properties \eqref{reg:eta}, \eqref{apriori:h} and \eqref{lower:1-B}.  Then there exists an increasing function $\cF:\Rr^+\times \Rr^+\to \Rr^+$ depending only on $(s, h, \ka)$ such that
\bq\label{apriori:est}
\| \eta\|_{L^\infty([0, T]; H^s)}+\| \eta\|_{L^2([0, T]; H^{s+\mez})}\le \cF\Big(\| \eta(0)\|_{H^s}+T \cF\big(\| \eta\|_{L^\infty([0, T]; H^s)}, \ma^{-1}\big), \ma^{-1}\Big).
\eq
\end{prop}
In order to close \eqref{apriori:est}, we prove a priori estimates for $\ma$ and $h$ in the next subsection.
\subsubsection{A priori estimates for the parabolicity and the depth}\label{section:propaRT}

Using \eqref{eq:eta} (or  the approximate equation \eqref{eq:eta:app} below) and Theorem \ref{theo:estDN} (with $\sigma=\mez$) we first observe that
\begin{equation*}
\begin{split}
\Vert \eta(t)-\eta(0)\Vert_{H^{-\mez}}&\le \int_{0}^t\Vert \partial_t\eta( \tau )\Vert_{H^{-\mez}}d\tau \le t\mathcal{F}(\Vert \eta\Vert_{L^\infty([0, t]; H^s)})\Vert \eta\Vert_{L^\infty([0,t];H^\tdm)}\les  t\mathcal{F}(\Vert \eta\Vert_{L^\infty([0, t]; H^s)}).
\end{split}
\end{equation*}
Hence, by interpolation, for $s^-=s-\frac{\delta}{2}>1+\frac{d}{2}+\frac{\delta}{2}$,
\begin{equation*}
\begin{split}
\Vert \eta(t)-\eta(0)\Vert_{H^{s^-}}&\lesssim (\Vert\eta(t)\Vert_{H^s}+\Vert \eta(0)\Vert_{H^s})^{1-\tt}\Vert \eta(t)-\eta(0)\Vert_{H^{-\mez}}^\tt,\quad \tt\in (0, 1)\\
\end{split}
\end{equation*}
and similarly, by virtue of Theorem \ref{theo:estDN} and Corollary \ref{theo:contractionDN} (note that $s^->s-\mez-\delta$),
\begin{equation*}
\begin{split}
&\Vert G(\eta(t))(\eta(t)-\eta(0))\Vert_{H^{s^--1}}\le\mathcal{F}(\Vert \eta\Vert_{L^\infty H^s})\Vert\eta(t)-\eta(0)\Vert_{H^{s^-}},\\
&\Vert (G(\eta(t)-G(\eta(0)))\eta(0)\Vert_{H^{s^--1}}\le\mathcal{F}(\Vert \eta\Vert_{L^\infty H^s})\Vert\eta(0)\Vert_{H^s}\Vert \eta(t)-\eta(0)\Vert_{H^{s^-}}.
\end{split}
\end{equation*}
Thus,  there exists $\theta>0$ such that
\begin{equation*}
\begin{split}
&\Vert G(\eta(t))\eta(t)-G(\eta(0)\eta(0)\Vert_{L^\infty}+\Vert \nabla_x\eta(t)-\nabla_x\eta(0)\Vert_{L^\infty}+\Vert \eta(t)-\eta(0)\Vert_{L^\infty}\le t^\theta\mathcal{F}(\Vert \eta\Vert_{L^\infty([0, t]; H^s)}.
\end{split}
\end{equation*}
Recalling the definition in \eqref{BV}, we deduce that
\bq\label{bootstrap:B}
\inf_{(x, t)\in \Rr^d\times [0, T]}(1-B(x, t))\ge \inf_{x\in \Rr^d}(1-B(x, 0))- T^\theta\cF_1\big( \| \eta \|_{L^\infty([0, T]; H^s)}\big)
\eq
 and that \bq\label{bootstrap:h}
\begin{aligned}
\inf_{t\in [0, T]}\dist(\eta(t),\Gamma^-)
&\ge \dist(\eta_0,\Gamma^-)- T^\theta \cF_1\big(\| \eta \|_{L^\infty([0, T]; H^s)}\big),
\end{aligned}
\eq
where $\tt\in (0, 1)$ and $\cF_1:\Rr^+\to \Rr^+$ depends only on $(s, h, \kappa)$.

\subsubsection{Proof of Theorem \ref{theo:wp1}}\label{subsection:proofwp1}

Having the a priori estimates \eqref{apriori:est}, \eqref{bootstrap:B} and \eqref{bootstrap:h} in hand, we turn to prove the existence of  $H^s$ solutions of \eqref{eq:eta}.  By a contraction mapping argument, we can prove that for each $\eps \in (0, 1)$ the parabolic approximation 
\bq\label{eq:eta:app}
\p_t\eta_\eps=-\ka G(\eta_\eps)\eta_\eps+\eps\Delta \eta_\eps,\quad \eta_\eps\vert_{t=0}=\eta_0,
\eq
has a unique solution $\eta_\eps$ in the complete metric space 
\bq\label{def:E}
E_{h}(T_\eps)=\{ v\in C([0, T_\eps]; H^s)\cap L^2([0, T_\eps]; H^{s+1}): \dist(v(t), \Gamma^-)\ge h~ \forall~ t\in [0, T_\eps]\}
\eq
provided that $T_\eps$ is sufficiently small and that $\dist(\eta_0, \Gamma^-)\ge 2h>0$ . Let us note that the dissipation term $\eps\Delta \eta_\eps$ in  \eqref{eq:eta:app} has higher order than the term $-\ka G(\eta_\eps)\eta_\eps$ so that the parabolicity coming from $-\ka G(\eta_\eps)\eta_\eps$ is not needed in the definition of $E_h$. 

On the other hand, since $\eta_0\in H^s(\Rr^d)$ with $s>1+\frac d2$, applying the upper bound in Proposition \ref{lemm:upperbounG} with $f=\eta_0$ we obtain that
\bq
\inf_{x\in \Rr^d}(1-B_\eps(x, 0))\ge 2\ma
\eq
for some constant $\ma>0$ independent of $\eps$. It  then follows from the a priori estimates \eqref{apriori:est}, \eqref{bootstrap:B}, \eqref{bootstrap:h} and a continuity argument that there exists a positive time $T$ such that $T<T_\eps$ for all $\eps \in (0, 1)$. Moreover, on $[0, T]$, we have the uniform bounds
\begin{align}\label{ubound:1-B}
&\inf_{(x, t)\in \Rr^d\times [0, T]}(1-B_\eps(x, t))\ge \ma,\\\label{ubound:depth}
&\dist(\eta_\eps(T),\Gamma^-)\ge h,\\\label{uniform:eta}
&\| \eta_\eps\|_{L^\infty([0, T]; H^s)}+\| \eta_\eps\|_{L^2([0, T]; H^{s+\mez})}\le \cF\big(\| \eta(0)\|_{H^s}, \ma^{-1}\big),
\end{align}
where $\cF$ depends only on $(s, h, \kappa)$. In addition, the $L^2$ norm of $\eta_\eps$ is nonincreasing in time since 
\[
(G(\eta_\eps)\eta_\eps, \eta_\eps)_{L^2(\Rr^d)}\ge 0,\quad 
(\Delta\eta_\eps, \eta_\eps)_{L^2(\Rr^d)}=-\| \na \eta_\eps\|^2_{L^2(\Rr^d)}\le 0.
\]
Next we show that for any sequence $\eps_n\to 0$, the solution sequence $\eta_n\equiv \eta_{\eps_n}$ is Cauchy in the space
\[
Z^{s-1}(T)=L^\infty([0, T]; H^{s-1})\cap L^2([0, T]; H^{s-\mez}).
\]
Fix $\delta \in(0, \mez]$ satisfying $\delta<s-1-\frac d2$.  We introduce the difference $\eta_\delta=\eta_m-\eta_n$ and claim that it satisfies a nice equation:
\begin{equation}\label{equForDifferenceNew}
\partial_t\eta_\delta=-\kappa \big(T_{\frak P}\eta_\delta- T_{V}\cdot\nabla_x\eta_\delta\big)+\mathcal{R}_1+\mathcal{R}_2,
\end{equation}
where 
\[
\frak P=\mez\big(\lambda_n(1-B_n)+\lambda_m(1-B_m)\big),\quad V=\mez(V_n+V_m),
\]
and the remainder terms satisfy
\begin{equation*}
\begin{split}
\Vert \mathcal{R}_1(t)\Vert_{H^{s-\frac{3}{2}}}&\le\mathcal{F}(\Vert \eta_n\Vert_{L^\infty H^s}+\Vert \eta_m\Vert_{L^\infty H^s})\Vert \eta_\delta(t)\Vert_{H^{s-\frac{1}{2}-\delta}},\\
\Vert \mathcal{R}_2(t)\Vert_{ H^{s-\frac{3}{2}}}&\le (\varepsilon_m+\varepsilon_n)(\Vert \eta_m(t)\Vert_{H^{s+\frac{1}{2}}}+\Vert \eta_n(t)\Vert_{H^{s+\frac{1}{2}}})
\end{split}
\end{equation*}

Indeed, taking the difference in \eqref{eq:eta:app}, we obtain
\[
\begin{aligned}
\p_t\eta_\delta&=-\ka \Big(G^-(\eta_m)\eta_\delta+ [G^-(\eta_m)-G^-(\eta_n)]\eta_n\Big)+\eps_m\Delta \eta_m-\eps_n\Delta\eta_n.
\end{aligned}
\]
and we can directly set $\mathcal{R}_2:=\eps_m\Delta \eta_m-\eps_n\Delta\eta_n$. For the remaining terms, we apply Theorem  \ref{paralin:ABZ} (with $\sigma=s-\mez-\delta$) and Theorem \ref{theo:contractionDN2} (with $\sigma=s-\mez$) to get
\begin{equation*}
\begin{split}
G^-(\eta_m)\eta_\delta +[G^-(\eta_m)-G^-(\eta_n)]\eta_n&=T_{\lambda_m}\eta_\delta-T_{\lambda_nB_n}\eta_\delta-T_{V_n}\cdot\nabla_x\eta_\delta\\
&\qquad+R_0^-(\eta_m)\eta_\delta+R_2^-(\eta_m,\eta_n)\eta_n.
\end{split}
\end{equation*}
Note that the remainder $R_0^-$ and $R^-_2$ lead to acceptable terms as in $\mathcal{R}_1$. But we also have 
\begin{equation*}
\begin{split}
G^-(\eta_n)\eta_\delta +[G^-(\eta_n)-G^-(\eta_m)]\eta_m&=T_{\lambda_n}\eta_\delta-T_{\lambda_mB_m}\eta_\delta-T_{V_m}\cdot\nabla_x\eta_\delta\\
&\qquad+R_0^-(\eta_n)\eta_\delta+R_2^-(\eta_n,\eta_m)\eta_n.
\end{split}
\end{equation*}
Thus, by taking the average of the above two identities, we  arrive at \eqref{equForDifferenceNew}.

Now, $H^{s-1}$ energy estimates using \eqref{equForDifferenceNew} give that
\begin{equation*}
\begin{split}
\frac{1}{2}\frac{d}{dt}\Vert \eta_\delta(t)\Vert_{H^{s-1}}^2&=-\kappa(\langle D_x\rangle^{s-\tdm}T_{\frak B}\eta_\delta,\langle D_x\rangle^{s-\mez}\eta_\delta)_{L^2}+\kappa(\langle D_x\rangle^{s-\tdm}T_{V}\cdot\nabla\eta_\delta,\langle D_x\rangle^{s-\mez}\eta_\delta)_{L^2}\\
&\quad+(\langle D_x\rangle^{s-\tdm}\mathcal{R}_1,\langle D_x\rangle^{s-\mez}\eta_\delta)_{L^2}+(\langle D_x\rangle^{s-\tdm}\mathcal{R}_2,\langle D_x\rangle^{s-\mez}\eta_\delta)_{L^2}\\
&:=I_1+I_2+I_3+I_4.
\end{split}
\end{equation*}
We can now estimate each term one by one. First, it follows from the above estimates for $\mathcal{R}_1$ and  $\mathcal{R}_2$ that
\begin{equation*}
\begin{split}
\vert I_4\vert&\le\Vert \mathcal{R}_2\Vert_{H^{s-\frac{3}{2}}}\Vert\eta_\delta\Vert_{H^{s-\frac{1}{2}}}\\
&\le (\varepsilon_m+\varepsilon_n)(\Vert\eta_m\Vert_{H^{s+\frac{1}{2}}}^2+\Vert \eta_n\Vert_{H^{s+\frac{1}{2}}}^2+\Vert \eta_\delta\Vert_{H^{s-\frac{1}{2}}}^2),\\
\vert I_3\vert&\le\Vert \mathcal{R}_1\Vert_{H^{s-\frac{3}{2}}}\Vert\eta_\delta\Vert_{H^{s-\frac{1}{2}}}\\
&\lesssim\mathcal{F}(\Vert \eta_n\Vert_{L^\infty H^s}+\Vert \eta_m\Vert_{L^\infty H^s})\Vert \eta_\delta\Vert_{H^{s-\frac{1}{2}-\delta}}\Vert \eta_\delta\Vert_{H^{s-\frac{1}{2}}}\\
&\lesssim \mathcal{F}(\Vert \eta_n\Vert_{L^\infty H^s}+\Vert \eta_m\Vert_{L^\infty H^s})\Vert \eta_\delta\Vert_{H^{s-\frac{1}{2}}}^{1+\theta}\Vert \eta_\delta\Vert_{H^{s-1}}^{1-\theta}\\
&\le \varepsilon_\ast \Vert \eta_\delta\Vert_{H^{s-\frac{1}{2}}}^2+C_{\varepsilon_\ast} \mathcal{F}(\Vert \eta_n\Vert_{L^\infty H^s}+\Vert \eta_m\Vert_{L^\infty H^s})\Vert \eta_\delta\Vert_{H^{s-1}}^2,
\end{split}
\end{equation*}
where $\eps_*>0$ is arbitrary.   Proceeding as in \eqref{energy:advec}, we see that
\begin{equation*}
\begin{split}
\vert I_2\vert&\le \mathcal{F}(\Vert \eta_n\Vert_{L^\infty H^s}+\Vert \eta_m\Vert_{L^\infty H^s})\Vert \eta_\delta\Vert_{H^{s+\frac{1}{2}}}\Vert \eta_\delta\Vert_{H^{s-\frac{1}{2}+\delta}}\\
&\le \varepsilon_\ast\Vert \eta_\delta\Vert_{H^{s-\frac{1}{2}}}^2+ C_{\varepsilon_\ast}\mathcal{F}(\Vert \eta_n\Vert_{L^\infty H^s}+\Vert \eta_m\Vert_{L^\infty H^s})\Vert \eta_\delta\Vert_{H^{s-1}}^2\\
\end{split}
\end{equation*}
and finally, as in \eqref{energy:dif},
\begin{equation*}
\begin{split}
I_1&\le- \frac{1}{\mathcal{F}(\Vert \eta_n\Vert_{L^\infty H^s},\frac{1}{\mathfrak{a}_n})} \Vert \eta_\delta\Vert_{H^{s-\frac{1}{2}}}^2+ \mathcal{F}(\Vert \eta_n\Vert_{L^\infty H^s},\frac{1}{\mathfrak{a}_n})\Vert\eta_\delta\Vert_{H^{s-\frac{1}{2}}}\Vert \eta_\delta\Vert_{H^{s-\frac{1}{2}+\delta}}\\
&\le -\Big(\frac{1}{\mathcal{F}(\Vert \eta_n\Vert_{L^\infty H^s},\frac{1}{\mathfrak{a}_n})}-\varepsilon_\ast\Big) \Vert \eta_\delta\Vert_{H^{s-\frac{1}{2}}}^2+C_{\varepsilon_\ast}\mathcal{F}(\Vert \eta_n\Vert_{L^\infty H^s},\frac{1}{\mathfrak{a}_n})\Vert \eta_\delta\Vert_{H^{s-1}}^2.
\end{split}
\end{equation*} 
Adding all the above estimates yields
\begin{equation*}
\begin{split}
&\frac{1}{2}\frac{d}{dt}\Vert \eta_\delta(t)\Vert_{H^{s-1}}^2\le -\Big(\frac{1}{\mathcal{F}(\Vert \eta_n\Vert_{L^\infty H^s},\frac{1}{\mathfrak{a}_n})}-3\eps_*-\varepsilon_m-\varepsilon_n\Big) \Vert \eta_\delta(t)\Vert_{H^{s-\frac{1}{2}}}^2\\
&\quad+ C_{\varepsilon_\ast}\mathcal{F}(\Vert \eta_n\Vert_{L^\infty H^s}+\Vert \eta_m\Vert_{L^\infty H^s})\Vert \eta_\delta(t)\Vert_{H^{s-1}}^2+(\varepsilon_m+\varepsilon_n)(\Vert\eta_m(t)\Vert_{H^{s+\frac{1}{2}}}^2+\Vert \eta_n(t)\Vert_{H^{s+\frac{1}{2}}}^2).
\end{split}
\end{equation*}
Bt virtue of the uniform bounds for $\eta_n$, there exists $c_*>0$ such that
\[
\frac{1}{\mathcal{F}(\Vert \eta_n\Vert_{L^\infty([0, T];  H^s)},\frac{1}{\mathfrak{a}_n})}\ge c_*\quad\forall n\in \Nn.
\]
Choosing $\eps_*=\frac{c_*}{10}$ and taking $m$ and $n$ sufficiently large so that $\eps_m,~\eps_n\le\frac{c_*}{10}$, we obtain
\begin{equation}\label{ineq:diff}
\begin{split}
\frac{1}{2}\frac{d}{dt}\Vert \eta_\delta(t)\Vert_{H^{s-1}}^2&\le -\frac{c_*}{2}\Vert \eta_\delta\Vert_{H^{s-\frac{1}{2}}}^2+ C\Vert \eta_\delta\Vert_{H^{s-1}}^2+(\varepsilon_m+\varepsilon_n)(\Vert\eta_m\Vert_{H^{s+\frac{1}{2}}}^2+\Vert \eta_n\Vert_{H^{s+\frac{1}{2}}}^2).
\end{split}
\end{equation}
Ignoring the first term on the right-hand side, then integrating in time we obtain
\[
\| \eta_\delta(t)\|_{H^{s-1}}^2\le C\int_0^t\| \eta_\delta(\tau)\|_{H^{s-1}}^2d\tau+C(\varepsilon_m+\varepsilon_n)(\Vert\eta_m\Vert_{H^{s+\frac{1}{2}}}^2+\Vert \eta_n\Vert_{H^{s+\frac{1}{2}}}^2).
\]
In view of \eqref{uniform:eta}, the sequence $\| \eta_n\|_{L^2([0, T]; H^{s+\mez})}^2$ is bounded, whence Gr\"onwall's lemma implies 
\bq\label{diff:est1}
\| \eta_\delta\|_{L^\infty([0, T]; H^{s-1})}^2\le C'\big(\eps_m+\eps_n\big)\exp(C'T)
\eq
We then integrate \eqref{ineq:diff} in time and use \eqref{diff:est1} to get the dissipation estimate
\bq\label{diff:est2}
\| \eta_\delta\|_{L^2([0, T]; H^{s-\mez})}^2\le C''\big(\eps_m+\eps_n\big)\exp(C''T).
\eq
It follows from \eqref{diff:est1} and \eqref{diff:est2} that  $\eta_n$ is a Cauchy sequence in $Z^{s-1}(T)$. Therefore, there exists $\eta \in Z^s(T)$ such that  $\eta_n\to \eta$ in $Z^{s-1}(T)$. By virtue of Theorem \ref{theo:estDN} and Corollary \ref{theo:contractionDN}, $G^-(\eta_n)\eta_n\to G^-(\eta)\eta$ in $H^{s-1}$ and thus $\eta$ is a solution of \eqref{eq:eta} in $Z^s(T)$. 

Repeating the above proof of the fact that $\eta_n$ is a Cauchy sequence in $Z^{s-1}(T)$, we obtain the following stability estimate.
\begin{prop}\label{prop:stab}
Let $\eta_1$ and $\eta_2$ be two solutions of \eqref{eq:eta} in $Z^s(T)$ defined by \eqref{def:spaceZs} with $s>1+\frac d2$ and
\[
\begin{aligned}
&\dist(\eta_j(t), \Gamma^-)\ge h\quad\forall t\in [0, T],\\
&\inf_{x\in \Rr^d}(1-B_j(x, t))\ge \ma>0\quad\forall t\in [0, T].
\end{aligned} 
\]
Then, 
\bq
\| \eta_1-\eta_2\|_{Z^{s-1}(T)}\le  \mathcal{F}\big(\| (\eta_1, \eta_2) \|_{L^\infty([0, T]; H^s)}\big)\| (\eta_1-\eta_2)\vert_{t=0}\|_{H^{s-1}}
\eq
for some $\cF:\Rr^+\to \Rr^+$ depending only on $(s, h, \ma, \ka)$. 
\end{prop}
Finally, uniqueness and continuous dependence on initial data for the solution $\eta\in Z^{s-1}(T)$ constructed above follow at once from Proposition \ref{prop:stab}.
\subsection{Proof of Theorem \ref{theo:wp2}}

We now consider the two-phase problem  \eqref{eq:eta2p}-\eqref{system:fpm}. We assume throughout that either $\Gamma^\pm=\emptyset$ or $\underline b^\pm\in \dot W^{1, \infty}(\Rr^d)$. 
\subsubsection{Well-posedness of the elliptic problem \eqref{system:fpm}}
\begin{prop}\label{prop:fpm}
Let $\eta \in W^{1, \infty}(\Rr^d)\cap H^\mez(\Rr^d)$. Then there exists a unique solution $f^\pm \in \wt H^\mez_{\Theta_\pm}(\Rr^d)$,
\[
 \Theta_\pm(x):=
 \begin{cases}
 \infty\quad\text{if}~\Gamma^\pm=\emptyset,\\
\frac{\mp(\eta^\pm-\underline{b}^\pm(x))}{2(\| \na_x\eta\|_{L^\infty(\Rr^d)}+\|\underline{b}^\pm\|_{L^\infty(\Rr^d)})} \quad\text{if}~\underline b^\pm\in \dot W^{1, \infty}(\Rr^d),
\end{cases}
\]
 to the system \eqref{system:fpm}. Moreover, $f^\pm$ satisfy 
\bq\label{variest:fpm}
\| f^\pm\|_{\wt H^\mez_{\Theta_\pm}(\Rr^d)}\le C(1+\| \eta\|_{W^{1, \infty}(\Rr^d)}){(1+\| (\na_x\eta, \na_x\underline{b}^\pm)\|_{L^\infty(\Rr^d)})}\lb\rho\rb\| \eta\|_{H^\mez(\Rr^d)},
\eq
where the constant $C$ depends only on $(\mu^\pm, h)$.
\end{prop}
\begin{proof}
Observe that  $f^\pm=q^\pm\vert_\Sigma$ where $q^\pm$ solve the two-phase elliptic problem 
\bq\label{elliptic:qpm}
\begin{cases}
\Delta q^\pm=0\quad\text{in}~\Omega^\pm,\\
q^+-q^-=-\lb \rho\rb\eta,\quad \frac{\p_nq^+}{\mu^+}-\frac{\p_nq^-}{\mu^-}=0\quad\text{on}~\Sigma,\\
\p_{\nu^\pm}q^\pm=0\quad\text{on}~\Gamma^\pm.
\end{cases}
\eq
Here the Neumann boundary conditions need to be modified as in \eqref{vanish:dyphi} when $\Gamma^+$ or $\Gamma^-$ is empty. Thus, it remains to prove the unique solvability of \eqref{elliptic:qpm}.  To remove the jump of $q$ at the interface, let us fix a cut-off $\chi \in C^\infty(\Rr)$ satisfying $\chi(z)=1$ for $|z|<\mez$, $\chi(z)=0$ for $|z|>1$, and set 
\[
\underline\theta(x, z)=-\frac{\lb \rho\rb}{2}\chi(z)e^{-|z|}\langle D_x\rangle\eta(x),\quad(x, z)\in \Rr^d\times \Rr
\]
and
\[
\theta(x, y)=\underline\theta(x, \frac{y-\eta(x)}{h}),\quad(x, y)\in \Rr^d\times \Rr.
\]
Then $\tt(x, \eta(x))=-\frac{\lb \rho\rb}{2}\eta(x)$ and $\tt$ vanishes near $\Gamma^\pm$. Moreover, we have 
\bq\label{tt:eta}
\| \tt\|_{H^1(\Omega)}\le C\lb\rho\rb(1+\| \eta\|_{W^{1, \infty}})\| \eta\|_{H^\mez},\quad \Omega=\Omega^+\cup \Omega^-.
\eq
We then need to prove that there exists a unique solution $r \in \dot H^1(\Omega)$ to the problem 
\bq\label{elliptic:rpm}
\begin{cases}
-\Delta r=\pm \Delta \tt\quad\text{in}~\Omega^\pm,\\
  \frac{\p_n r}{\mu^+}-\frac{\p_n r}{\mu^-}=-\p_n\tt(\frac{1}{\mu^+}+\frac{1}{\mu^-})\quad \text{on}~\Sigma,\\
\p_{\nu^\pm}r=0\quad\text{on}~\Gamma^\pm.
\end{cases}
\eq
The pair $q^\pm:=r\vert_{\Omega^\pm}\pm \tt$ is the unique solution of \eqref{elliptic:qpm}. For a smooth solution $r$ and for any smooth test function  $\phi:\Omega\to \Rr$ we have after integrating by parts that
\[
\int_{\Omega^+} \na r\cdot \na \phi dx+\int_\Sigma \p_nr\phi dS=-\int_{\Omega^+} \na \tt\cdot \na \phi dx-\int_\Sigma \p_n\tt\phi dS
\]
and 
\[
\int_{\Omega^-} \na r\cdot \na \phi dx-\int_\Sigma \p_nr\phi dS=\int_{\Omega^-} \na \tt\cdot \na \phi dx-\int_\Sigma \p_n\tt\phi dS.
\]
 Multiplying the first equation by $\frac{1}{\mu^+}$ and the second one by $\frac{1}{\mu^-}$ then adding and using the jump conditions in \eqref{elliptic:rpm} we obtain
\bq\label{vari:rpm}
\begin{aligned}
&\frac{1}{\mu^+}\int_{\Omega^+} \na r\cdot \na \phi dx+\frac{1}{\mu^-}\int_{\Omega^-} \na r\cdot \na \phi dx=-\frac{1}{\mu^+}\int_{\Omega^+} \na \tt\cdot \na \phi dx+\frac{1}{\mu^-}\int_{\Omega^-} \na \tt\cdot \na \phi dx.
\end{aligned}
\eq
Conversely, if $r$ is a sufficiently smooth function that verifies \eqref{vari:rpm}, then upon integrating by parts we can show that $r$ solves \eqref{elliptic:rpm}. Therefore,  $r\in \dot H^1(\Omega)$ is a variational solution of \eqref{elliptic:rpm} if  the weak formulation \eqref{vari:rpm} is satisfied for all test functions $\phi\in \dot H^1(\Omega)$. By virtue of the estimate \eqref{tt:eta} and Proposition \ref{dotH1:complete}, the Lax-Milgram theorem guarantees the existence of a unique variational solution $r$; moreover,  the variational bound 
\[
\| \na r\|_{L^2(\Omega)}\le C\lb\rho\rb(1+\| \eta\|_{W^{1, \infty}})\| \eta\|_{H^\mez}
\]
holds for some constant $C$ depending only on $(\mu^\pm, h)$. This combined with \eqref{tt:eta} implies that $q^\pm=r\vert_{\Omega^\pm}\pm \tt$ satisfy the same bound
\[
\| q^\pm\|_{\dot H^1(\Omega^\pm)}\le C\lb\rho\rb(1+\| \eta\|_{W^{1, \infty}})\| \eta\|_{H^\mez}.
\]
Finally, \eqref{variest:fpm} follows from this and the trace inequalities \eqref{trace:ineq} and \eqref{trace:ineq2}.
\end{proof}
Since it is always possible to find $a\in (0, 1)$ such that $\Theta_\pm(x)\ge  \mp a(\eta^\pm(x)-\underline{b}^\pm(x))$, we have $f^\pm\in \wt H^\mez_\pm$.
\begin{rema}
1) In fact, the proof of Proposition \ref{prop:fpm} shows that for $\eta\in W^{1, \infty}(\Rr^d)$ and $g\in H^\mez(\Rr^d)$, there exists a unique variational solution  $f^\pm \in \wt H^\mez_{\Theta_\pm}(\Rr^d)$ to the system
\bq
\begin{cases}
 f^+-f^-= g,\\
\frac{1}{\mu^+}G^+(\eta)f^+-\frac{1}{\mu^-}G^-(\eta)f^-=0.
\end{cases}
\eq
In addition, there exists  a constant $C$ depending only on $(\mu^\pm, h)$ such that
\bq
\| f^\pm\|_{\wt H^\mez_{\Theta_\pm}(\Rr^d)}\le C(1+\| \eta\|_{W^{1, \infty}(\Rr^d)}){(1+\| (\na_x\eta, \na_x\underline{b}^\pm)\|_{L^\infty(\Rr^d)})}\| g\|_{H^\mez(\Rr^d)}.
\eq
2) If $\Gamma^+=\Gamma^-=\emptyset$, then it suffices to assume $\eta\in \dot W^{1, \infty}(\Rr^d)$ and $g\in \dot H^\mez(\Rr^d)$ since localization away from $\Gamma^\pm$ is  not needed and one can choose 
\[
\underline\theta(x, z)=\frac{1}{2}e^{-|z||D_x|}g(x),\quad(x, z)\in \Rr^d\times \Rr.
\]
Then, $\tt(x, y):=\underline\tt(x, y-\eta(x))$ satisfies 
\[
\| \tt\|_{\dot H^1(\Omega)}=\| \tt\|_{\dot H^1(\Rr^d)}\le C(1+\| \na_x\eta\|_{L^\infty(\Rr^d)})\| g\|_{\dot H^\mez(\Rr^d)}.
\]
Consequently, 
\bq
\| f^\pm\|_{\wt H^\mez_\infty(\Rr^d)}\le C(1+\| \na_x\eta\|_{L^\infty(\Rr^d)})^2\| g\|_{\dot H^\mez(\Rr^d)}.
\eq
\end{rema}
\subsubsection{A priori higher regularity estimate for $f^\pm$}
The weak regularity bound can then be bootstrapped to regularity of the surface $\eta$.
\begin{prop}\label{prop:fpmdotHs}
Let $f^\pm$ be the solution of \eqref{system:fpm} as given by Proposition \ref{prop:fpm}. If $\eta \in H^s(\Rr^d)$ with $s>1+\frac d2$ then for any $r\in [\mez, s]$, we have
\bq\label{fpm:dotHs}
\| f^\pm\|_{\widetilde{H}^r_{\pm}}\le \cF(\| \eta\|_{H^s})\| \eta\|_{H^r}.
\eq
\end{prop}
\begin{proof}
Fix $\delta\in \big(0, \min(s-1-\frac{d}{2}, 1)\big)$. We first claim that whenever $f^\pm \in \wt H^{\sigma}$ with $\sigma\in [\mez, s-\delta]$ we have $f^\pm \in \wt H^{\sigma+\delta}$ and
\bq\label{BoostrapHRNew}
\| f^\pm\|_{\wt H^{\sigma+\delta}}\le  \cF(\|\eta\|_{H^s})\big(\| f^\pm\|_{\wt H^\sigma}+\| \eta\|_{H^{\sigma+\delta}}\big).
\eq
Indeed, applying Theorem \ref{paralin:ABZ} we have 
 \begin{align*}
& G^\pm(\eta)f^\pm=\mp T_{\lambda}f^\pm+R^\pm_0(\eta)f^\pm,\\
 &\Vert R^\pm_0(\eta)f^\pm\Vert_{H^{\sigma-1+\delta}}\le \mathcal{F}(\Vert \eta\Vert_{H^s})\Vert f^\pm\Vert_{\widetilde{H}^\sigma_\pm}.
\end{align*}
Plugging this into the  second equation in \eqref{system:fpm} we obtain 
\[
\Vert \frac{1}{\mu^-}T_\lambda f^-+\frac{1}{\mu^+}T_\lambda f^+\Vert_{H^{\sigma-1+\delta}}\le \mathcal{F}(\Vert \eta\Vert_{H^s})\Vert f^\pm\Vert_{\widetilde{H}^\sigma_\pm}.
\]
But $f^+=f^--\lb \rho\rb$ from the first equation in \eqref{system:fpm}, hence 
\bq\label{boostrap:Tf-}
\Vert T_\lambda f^-\Vert_{H^{\sigma-1+\delta}}\le \mathcal{F}(\Vert \eta\Vert_{H^s})\big(\Vert f^\pm\Vert_{\widetilde{H}^\sigma_\pm}+\| \eta\|_{H^{\sigma+\delta}}\big).
\eq
On the other hand, by virtue of Theorem \ref{theo:sc} (ii) we have 
\[
\| \Psi(D)g\|_{H^\nu}=\| T_1g\|_{H^\nu}\le \mathcal{F}(\Vert \eta\Vert_{H^s})(\| T_\ld g\|_{H^{\nu-1}}+\| g\|_{H^{1, \nu -\delta}}) \quad\forall \nu \in \Rr.
\]
Combining this with  the inequality 
\[
\| g\|_{\wt H^\nu_\pm}\les \| g\|_{\wt H^\mez_\pm}+\| \psi(D) g\|_{H^\nu}\quad\forall\nu \ge \mez
\]
yields 
\[
\| g\|_{\wt H^\nu_\pm}\les \| g\|_{\wt H^\mez_\pm}+\mathcal{F}(\Vert \eta\Vert_{H^s})(\| T_\ld g\|_{H^{\nu-1}}+\| g\|_{H^{1, \nu -\delta}})\quad\forall \nu\ge \mez.
\]
Applying this with $g=f^-$ and $\nu=\sigma+\delta\ge \mez+\delta$ we deduce in view of \eqref{boostrap:Tf-} and \eqref{variest:fpm} that
\[
\| f^-\|_{\wt H^{\sigma+\delta}_\pm}\le\mathcal{F}(\|\eta\|_{H^s})\big(\Vert f^\pm\Vert_{\widetilde{H}^\sigma_\pm}+\| \eta\|_{H^{\sigma+\delta}}\big)
\]
for all $\sigma\in [\mez, s-\delta]$. Clearly, this implies \eqref{BoostrapHRNew}.

Then, because \eqref{fpm:dotHs} holds for $r=\mez$, an induction argument using \eqref{BoostrapHRNew} shows that \eqref{fpm:dotHs} holds for any $r\in [\mez, s]$.
\end{proof}
\subsubsection{Paradifferential reduction} 
\begin{prop}\label{prop:reduce2p}
Let $s>1+\frac d2$ and let $\delta \in \big(0,  s-1-\frac d2\big)$ and $\delta \le 1$. If $\eta\in H^s(\Rr^d)$ and $f^\pm\in \wt H^s_\pm$ solve the system \eqref{eq:eta2p}--\eqref{system:fpm}, then we have
\bq\label{eta:reduce2p}
\begin{aligned}
\p_t\eta&=-\frac{1}{\mu^-}G^-(\eta)f^-\\
&=-\frac{1}{\mu^++\mu^-}T_\lambda\left(\lb\rho\rb\eta-T_{\lb B\rb}\eta\right)+\frac{1}{\mu^++\mu^-}T_{\lb V\rb}\cdot\nabla \eta+R(\eta)
\end{aligned}
\eq
where  $\lb B\rb=B^--B^+$, $\lb V\rb=V^--V^+$, $B^\pm$ and $V^\pm$ are given by 
\bq\label{BVpm}
B^\pm=\frac{\nabla \eta\cdot \nabla f^\pm+G^\pm(\eta)f^\pm}{1+|\nabla \eta|^2},\quad V^\pm=\nabla f^\pm-B^\pm\nabla \eta,
\eq
 and $R(\eta)$ obeys the bound   
\bq\label{bound:Retafpm}
\| R(\eta)\|_{H^{s-\mez}}\le \cF(\| \eta\|_{H^s})\| \eta\|_{H^s}\big(1+\| \eta\|_{H^{s+\mez-\delta}}\big)
\eq
for some $\cF:\Rr^+\to \Rr^+$ depending only on $(s, \mu^\pm, \lb\rho\rb, h)$.
\end{prop}
\begin{proof}
We first apply the paralinearization in Theorem \ref{tame:RDN} to $G^\pm(\eta)f^\pm$ with $\sigma=s$, to have 
\bq\label{linGpm:pr1}
\begin{aligned}
&G^-(\eta)f^-= T_\lambda(f^--T_{B^-}\eta)-T_{V^-}\cdot\nabla \eta+R^-(\eta)f^-,\\
&G^+(\eta)f^+= -T_\lambda(f^+-T_{B^+}\eta)+T_{V^+}\cdot\nabla \eta+R^+(\eta)f^+,
\end{aligned}
\eq
where, using Proposition \ref{prop:fpmdotHs},
\begin{equation}\label{est:Rpmfpm}
\| R^\pm(\eta)f^\pm\|_{H^{s-\mez}}\le \cF(\| \eta\|_{H^s})\Vert \eta\Vert_{H^s}\big(1+\| \eta\|_{H^{s+\mez-\delta}}\big).
\end{equation}
In view of the second equation in \eqref{system:fpm}, \eqref{linGpm:pr1}  yields
\bq\label{linG=G}
T_\ld\big(\frac{f^+}{\mu^+}+\frac{f^-}{\mu^-}\big)=T_\ld T_{\frac{B^+}{\mu^+}+\frac{B^-}{\mu^-}}\eta+T_{\frac{V^+}{\mu^+}+\frac{V^-}{\mu^-}}\cdot\na \eta+\frac{1}{\mu^+}R^+(\eta)f^+-\frac{1}{\mu^-}R^-(\eta)f^-.
\eq
Since $f^+=f^--\lb \rho\rb\eta$, \eqref{linG=G} implies
\[
\begin{aligned}
T_\ld f^-&=\frac{\lb \rho\rb\mu^-}{\mu^++\mu^-}T_\ld \eta+\frac{1}{\mu^++\mu^-}T_\ld T_{\mu^+B^-+\mu^-B^+}\eta+\frac{1}{\mu^++\mu^-}T_{\mu^+V^-+\mu^-V^+}\cdot\na \eta\\
&\quad +\frac{1}{\mu^++\mu^-}\big(\mu^- R^+(\eta)f^+-\mu^+R^-(\eta)f^-\big).
\end{aligned}
\]
Plugging this into the first equation in \eqref{linGpm:pr1}, we arrive at \eqref{eta:reduce2p}
with
\begin{equation*}
\begin{split}
R(\eta)&=-\frac{1}{\mu^++\mu^-}\big( R^+(\eta)f^+-\frac{\mu^+}{\mu^-}R^-(\eta)f^-\big).
\end{split}
\end{equation*}
In view of \eqref{est:Rpmfpm}, this finishes the proof.
%
\end{proof}
When the top fluid is vacuum, equation \eqref{eta:reduce2p} reduces to equation \eqref{eta:reduce} previously obtained for the one-phase problem. Remarkably, \eqref{eta:reduce2p} together with the fact that
\[
T_\lambda\left(\lb\rho\rb\eta-T_{\lb B\rb}\eta\right)\sim T_{\lambda(\lb\rho\rb-\lb B\rb)}\eta
\]
 shows that the two-phase Muskat problem is parabolic so long as  the Rayleigh-Taylor condition $\text{RT}=\sqrt{1+|\na \eta|^2}(\lb\rho\rb-\lb B\rb)>0$ holds. In addition, for constant viscosity, $\lb \mu\rb=0$, by using \eqref{RT} and \eqref{RT2} we find that the parabolic term becomes explicit
 \bq
 T_{\lambda\left(\lb\rho\rb-\lb B\rb\right)}\eta=T_{\ld\lb \rho\rb(1+|\na_x\eta|^2)^{-1}}\eta.
 \eq
\subsubsection{Proof of Theorem \ref{theo:wp2}}
We observe that the paradifferential equation \eqref{eta:reduce2p} has the same form as equation \eqref{eta:reduce} for the one-phase problem. In particular,  an $H^s$ energy estimate on $[0, T]$ can be obtained as in Section \ref{section:energyest} provided that  $\text{RT}(x, t)>0$ for all $(x, t)\in \Rr^d\times [0, T]$. This stability condition can be propagated as in Section \ref{section:propaRT} if it is assumed to hold at initial time. In the rest of this subsection, we only sketch the approximation scheme that preserves the aforementioned a priori estimates. 

For each $\eps\in (0, 1)$, consider the approximate problem 
\bq\label{eq:eta2pe}
\p_t\eta_\eps=-\frac{1}{\mu^-}G^{-}(\eta_\eps)f^-_\eps+\eps \Delta \eta_\eps,
\eq
where $f^\pm_\eps$ solves \eqref{system:fpm}:
\bq\label{system:fpme}
\begin{cases}
 f^+_\eps-f^-_\eps= (\rho^+-\rho^-)\eta_\eps,\\
\frac{1}{\mu^+}G^+(\eta_\eps)f^+_\eps-\frac{1}{\mu^-}G^-(\eta_\eps)f^-_\eps=0.
\end{cases}
\eq
\begin{prop}
Let $s>1+\frac d2$ with $d\ge 1$. For each $\eta_0\in H^s(\Rr^d)$ with $\dist(\eta_0, \Gamma^\pm)>2h>0$, there exist  $T_\eps=T_\eps(\| \eta_0\|_{H^s}, h, s, \mu^\pm, \lb\rho\rb)>0$ and a unique solution $\eta_\eps$ to \eqref{eq:eta2pe}-\eqref{system:fpme} on $[0, T_\eps]$ such that $\eta_\eps\vert_{t=0}=\eta_0$,
 \[
  \eta_\eps\in C([0, T_\eps]; H^s(\Rr^d))\cap L^2([0, T_\eps]; H^{s+1}(\Rr^d)),
  \]
  and
 \[
\dist(\eta_\eps(t), \Gamma^\pm)\ge h\quad\forall t\in [0, T_\eps].
\]
\end{prop}
\begin{proof}
The unique existence of $\eta_\eps$ can be obtained via a contraction mapping argument in the space
\bq\label{def:E'h}
E'_{h}(T_\eps)=\{ v\in L^\infty([0, T_\eps]; H^s)\cap L^2([0, T_\eps]; H^{s+1}): \dist(v(t), \Gamma^\pm)\ge h~ \text{a.e.}~ t\in [0, T_\eps]\}
\eq
provided that $T_\eps$ is sufficiently small. Note the similarity between  $E'_h$ and $E_h$ defined by \eqref{def:E}. 
\end{proof}

\appendix
\section{Traces for homogeneous Sobolev spaces}\label{appendix:trace}
\subsection{Infinite strip-like domains}
Let $\eta_1$ and $\eta_2$ be two Lipchitz functions on $\Rr^d$, $\eta_j\in \dot W^{1, \infty}(\Rr^d)$, such that $\eta_1>\eta_2$.  Set 
\[
L=\| \na\eta_1\|_{L^\infty}+\| \na\eta_2\|_{L^\infty}, \quad\Theta(x)=\frac{\eta_1(x)-\eta_2(x)}{2L}.
\]
 Consider the infinite strip-like domain
\bq
U=\{(x, y)\in \Rr^{d+1}: \eta_2(x)<y<\eta_1(x)\}.
\eq
We record in this Appendix the trace theory in \cite{LeoTice} (see also \cite{Stri}) for $\dot H^1(U)$ where 
\[
\dot H^1(U)=\{u\in L^2_{loc}(U): \na u\in L^2(U)\}/~\Rr.
\]
\begin{theo}[\protect{\cite[Theorem~5.1]{LeoTice}}]\label{theo:trace}
 There exists a unique linear operator 
\[
\text{Tr}: \dot H^1(U)\to L^2_{loc}(\Rr^d)
\]
such that the following hold. 
\begin{itemize}
\item[1)] $\text{Tr}(u)=u\vert_{\p U}$ for all $u\in \dot H^1(U)\cap C(\overline U)$.
\item[2)] There exists a positive constant $C=C(d)$ such that for all $u\in \dot H^1(U)$, the functions $g_j=\text{Tr}(u)(\cdot, \eta_j(\cdot))$ are in $\wt H^\mez_\Theta(\Rr^d)$ and satisfy
\begin{align}\label{trace:ineq}
 & \| g_j\|_{ \wt H^\mez_\Theta(\Rr^d)}\le C(1+L) \| u\|_{\dot H^1(U)},\\
&\int_{\Rr^d}\frac{|g_1(x)-g_2(x)|^2}{\eta_1(x)-\eta_2(x)}dx\le C\int_U |\p_y u(x, y)|^2dxdy.
\end{align}
\end{itemize}
Recall that the space $\wt H^\mez_\Theta(\Rr^d)$ is defined by \eqref{def:wtHmez}.
\end{theo}
\begin{theo}[\protect{\cite[Theorem~5.4]{LeoTice}}]\label{theo:lift}
 Suppose that $g_1$ and $g_2$ are in $\wt H^\mez_{a(\eta_1-\eta_2)}(\Rr^d)$ for some $a\in (0, 1)$ such that
 \[
 \int_{\Rr^d}\frac{|g_1(x)-g_2(x)|^2}{\eta_1(x)-\eta_2(x)}dx<\infty.
 \] 
Then there exists $u\in \dot H^1(U)$ such that $\text{Tr}(u)(\cdot, \eta_j(\cdot))=g_j$ and 
\bq
\| u\|^2_{\dot H^1(U)}\le C(1+L)^2\int_{\Rr^d}\frac{|g_1(x)-g_2(x)|^2}{\eta_1(x)-\eta_2(x)}dx +C(1+L)^2\Big(\| g_1\|^2_{\wt H^\mez_{a(\eta_1-\eta_2)}(\Rr^d)}+\| g_2\|^2_{\wt H^\mez_{a(\eta_1-\eta_2)}(\Rr^d)}\Big)
\eq
where $C=C(d)$. 
\end{theo}
\subsection{Lipschitz half spaces}
For a Lipschitz function $\eta$ on $\Rr^d$ we consider the associated half-space 
\[
U=\{(x, y)\in \Rr^{n+1}: y<\eta(x)\}.
\]
\begin{theo}\label{theo:trace2}
 There exists a unique linear operator 
\[
\text{Tr}: \dot H^1(U)\to L^2_{loc}(\Rr^d)
\]
such that the following hold. 
\begin{itemize}
\item[1)] $\text{Tr}(u)=u\vert_{\p U}$ for all $u\in \dot H^1(U)\cap C(\overline U)$.
\item[2)] There exists a positive constant $C=C(d)$ such that for all $u\in \dot H^1(U)$, the function $g=\text{Tr}(u)(\cdot, \eta(\cdot))$ is in $\wt H^\mez_\infty(\Rr^d)$ and satisfies
\bq\label{trace:ineq2}
  \| g\|_{ \wt H^\mez_\infty(\Rr^d)}\le C(1+\|\na \eta\|_{L^\infty(\Rr^d)}) \|  u\|_{\dot H^1(U)}.
\eq
\end{itemize}
\end{theo}
\begin{theo}\label{theo:lift2}
For each $g\in \wt H^\mez_\infty(\Rr^d)$, there exists $u\in \dot H^1(U)$ such that $\text{Tr}(u)(\cdot, \eta(\cdot))=g$ and 
\bq
\| u\|_{\dot H^1(U)}\le C(1+\|\na \eta\|_{L^\infty(\Rr^d)})\| g_1\|_{\wt H^\mez_\infty(\Rr^d)}
\eq
where $C=C(d)$. 
\end{theo}

\subsection{Trace of normalized normal derivative}
 We consider the infinite strip-like domain
\[
 U_h=\{(x, y)\in \Rr^d\times \Rr:~\eta(x)-h<y<\eta(x)\}
\]  
of width $h$ underneath the graph of $\eta$
\[
\Sigma=\{ (x, \eta(x)):~x\in \Rr^d\}.
\]
Note that $n=\frac{1}{\sqrt{1+|\nabla \eta|}}(-\nabla \eta, 1)$ is the upward pointing unit normal to $\Sigma$.

Our goal is to show that for any function $u$ in the homogeneous maximal domain of the Laplace operator $\Delta$ 
 \[
E(U_h)=\{u\in \dot H^1(U_h):~\Delta_{x, y} u\in L^2(U_h)\},
 \]
 the trace 
 \[
\sqrt{1+|\na \eta|^2}\frac{\p u}{\p n}\vert_\Sigma=\p_y u(x, \eta(x))-(\na_xu)(x ,\eta(x))\cdot \na \eta(x)
 \]
  makes sense in $H^{-\mez}(\Rr^d)$. 
    \begin{theo}\label{theo:normaltrace}
Assume that $\eta\in \dot W^{1, \infty}(\Rr^d)$. There exists a unique linear operator $\cN: E(U_h)\to H^{-\mez}(\Rr^d)$ such that the following hold.
\begin{itemize}
\item[1)] $\cN(u)=\sqrt{1+|\na \eta|^2}\frac{\p u}{\p n}\vert_\Sigma$ if $u\in \dot H^1(U_h)\cap C^1(\overline U_h)$.
\item[2)] There exists an absolute constant $C>0$ such that 
\[
\| \cN(u)\|_{H^{-\mez}(\Rr^d)}\le  Ch^{-\mez}(1+\|\na\eta\|_{L^\infty})\| \na_{x, y}u\|_{L^2(U_h)}+Ch^{-\mez}\|\Delta_{x, y}u\|_{L^2(U_h)}
\]
for all $u\in E(U_h)$.
\end{itemize}
\end{theo}
  \begin{proof}
  Set $S=\Rr^d\times (-1, 0)$ and introduce  $\tt(x, z)=\eta(x)+zh$ for $(x, z)\in S$. It is clear that $(x, z)\mapsto (x, \tt(x, z))$ is a diffeomorphism from $S$ onto $U_h$. If $f:U_h\to \Rr$ we denote $\wt f(x, z)=f(x, \tt(x, z))$. It follows that 
 \[
 \begin{aligned}
 (\p_yf)(x, \tt(x, z))=\frac{1}{h}\p_z\wt f(x, z),\\
 (\na_xf)(x, \tt(x, z))=(\na_x-\frac{\na \eta}{h}\p_z)\wt f(x, z):=\Lambda \wt f.
 \end{aligned}
 \]
 For $u\in E(U)$ we have that $f=\Delta u$ satisfies 
 \bq\label{eq:wtf}
 \wt f =\cnx_{x,z}(A \na_{x,z}\wt u)
 \eq
with
\bq
A=
\begin{bmatrix}
Id & -\frac{\na \eta}{h}\\
-\frac{(\na \eta)^T}{h} &\frac{1+|\na\eta|^2}{h^2}.
\end{bmatrix}
\eq 
Equivalently, we have
\bq\label{eq0:dzXi}
\cnx_x\Lambda \wt u+\p_z(-\frac{\na \eta}{h}\na_x\wt u+\frac{1+|\na\eta|^2}{h^2}\p_z\wt u)=\wt f. 
\eq
Let us consider the quantity 
 \[
 \begin{aligned}
&(\p_y u)(x, \tt(x, z)) -(\na_xu)(x, \tt(x, z))\cdot \na \eta(x)\\
 &=\frac{1}{h}\p_z\wt u(x, z)-\na \eta(x, z)\cdot (\na_x-\frac{\na \eta}{h}\p_z)\wt u(x, z)\\
 &=-\frac{\na \eta(x)}{h}\cdot\na_x\wt u(x, z)+\frac{1+|\na\eta(x, z)|^2}{h^2}\p_z\wt u(x, z)\\
 &:= \Xi(x, z).
 \end{aligned}
 \]
 Using the first expression we deduce easily that
   \bq\label{L2:Xi}
 \| \Xi\|_{L^2((-1, 0); L^2(\Rr^d))}\le h^{-\mez}(1+\|\na\eta\|_{L^\infty})\| \na_{x, y}u\|_{L^2(U_h)}.
 \eq
  Moreover, \eqref{eq0:dzXi} implies
 \bq\label{eq:dzXi}
 \p_z \Xi=\wt f-\cnx_x\Lambda \wt u.
 \eq
  Note that if $u\in C^1(\overline U)$ then 
 \[
\p_y u(x, \eta(x)) -(\na_xu)(x ,\eta(x))\cdot \na \eta(x)=\Xi(x, 0).
 \]
 We shall appeal to Theorem \ref{trace:Lions} below to prove that the trace $\Xi\vert_{z=0}$ is well-defined in $H^{-\mez}(\Rr^d)$ for $u\in \dot H^1(U_h)$.  To this end, we use \eqref{eq:dzXi} to have
  \[
  \| \p_z\Xi\|_{L^2((-1, 0); H^{-1}(\Rr^d))}\le \|\wt f\|_{L^2((-1, 0); H^{-1}(\Rr^d))}+\| \Lambda \wt u\|_{L^2((-1, 0); L^2(\Rr^d))}
 \]
 where 
 \[
 \| \Lambda \wt u\|_{L^2((-1, 0); L^2(\Rr^d))}\le h^{-\mez}\| \na_xu\|_{L^2(U_h)}.
 \]
  On the other hand, 
\[
\|\wt f\|_{L^2((-1, 0); H^{-1}(\Rr^d))}\le \|\wt f\|_{L^2(S)}\le h^{-\mez}\| f\|_{L^2(U_h)},
\]
hence
\bq\label{L2:dzXi}
 \| \p_z\Xi\|_{L^2((-1, 0); H^{-1}(\Rr^d))}\le h^{-\mez}\big(\|f\|_{L^2(U_h)}+\| \na_xu\|_{L^2(U_h)}\big).
 \eq
 Combining \eqref{L2:Xi} and \eqref{L2:dzXi} we conclude by virtue of Theorem \ref{trace:Lions} that $\Xi\in C([-1, 0]; L^2(\Rr^d))$ and 
  \bq\label{C0:Xi}
\| \Xi \|_{ C([-1, 0]; L^2(\Rr^d))}\le  Ch^{-\mez}(1+\|\na\eta\|_{L^\infty})\| \na_{x, y}u\|_{L^2(U_h)}+Ch^{-\mez}\|f\|_{L^2(U_h)}
\eq
for some absolute constant $C>0$.
\end{proof}
\begin{theo}[\protect{\cite[Theorem 3.1]{Lions}}]\label{trace:Lions}
Let  $s \in \Rr$ and $I$ be a closed (bounded or unbounded) interval in $\Rr$. Let  $u\in L^2_z(I, H^{s+ \mez}(\Rr^d))$ such that 
$\partial_z u \in L_z^2(I, H^{s-\mez}(\Rr^d)).$ Then $u \in BC(I, H^{s}(\Rr^d))$ 
and there exists an absolute constant $C>0$ such that
$$
\sup_{z\in I}\| u(z, \cdot) \|_{H^{s}(\Rr^d)} 
\leq C \big(\|u\|_{L^2(I,H^{s+ \mez}(\Rr^d))}+  \|\partial_z u\|_{L^2(I,H^{s- \mez}(\Rr^d))} \big).
$$
\end{theo}
\subsection{Proof of Proposition \ref{prop:wtH}}\label{appendix:wtH}
Assuming \eqref{wt:cd1}, let us prove \eqref{equi:wtH}. By comparing $\sigma_1$ with $\min\{\sigma_1, \sigma_2\}+M$ it suffices to prove the following claim. If $\sigma\ge 2h>0$ for some $h>0$, then for any $M>0$, there exists $C=C(d, h, M)$ such that 
\bq\label{claim:wtH}
\| f\|_{\wt H^\mez_{\sigma+M}}\le C\| f\|_{\wt H^\mez_{\sigma}}.
\eq 
By iteration, \eqref{claim:wtH} will follow from the same estimate with $M$ replaced by $\delta=\frac{h}{10}$. To prove this, we first note that
 \begin{equation*}
\begin{split}
\Vert f\Vert_{\wt H^\mez_{\sigma+\delta}}^2
&=\Vert f\Vert_{\wt H^\mez_{\sigma}}^2+ \int_{x\in\mathbb{R}^d}\int_{\{\sigma(x)\le \vert k\vert\le \sigma(x)+\delta\}}\frac{\vert f(x+k)-f(x)\vert^2}{\vert k\vert^{d+1}}dkdx:=\Vert f\Vert_{\wt H^\mez_{\sigma}}^2+ J.
\end{split}
\end{equation*}
Letting $\theta(x)=1-\frac{h}{2\sigma(x)}$, we have $\tt\in [\mez, 1]$ and 
\bq\label{equi:wth:10}
\vert y\theta(x)\vert\le \sigma(x)-\frac{h}{4}\quad\text{when}\quad \vert y\vert\le\sigma(x)+\delta.
\eq
 Then, for $\vert u\vert\le h/4$, we decompose $J\le 2J_1+2J_2$ where
\begin{equation*}
\begin{split}
&J_1= \int_{x\in\mathbb{R}^d}\int_{\{\sigma(x)\le \vert k\vert\le \sigma(x)+\delta\}}\frac{\vert f(x+k)-f(x+\theta k+u)\vert^2}{\vert k\vert^{d+1}}dkdx\\
&J_2=\int_{x\in\mathbb{R}^d}\int_{\{\sigma(x)\le \vert k\vert\le \sigma(x)+\delta\}}\frac{\vert f(x+\theta k+u)-f(x)\vert^2}{\vert k\vert^{d+1}}dkdx.
\end{split}
\end{equation*}
We can easily dispense with $J_2$ by changing variable $k\mapsto z=\theta k+u$ and using \eqref{equi:wth:10}:
\begin{equation*}
\begin{split}
J_2&\lesssim \int_{x\in\mathbb{R}^d}\int_{\{\vert z\vert\le \sigma(x)\}}\frac{\vert f(x+z)-f(x)\vert^2}{\vert z\vert^{d+1}}dxdz=\Vert f\Vert_{H^s_{\sigma}}^2\\
\end{split}
\end{equation*}
uniformly in $\vert u\vert\le h/4$. To estimate $J_1$, we  average over $\vert u\vert\le h/4$:
\[
J_1\lesssim h^{-d} \int_{u\in B(0,h/4)}\int_{x\in\mathbb{R}^d}\int_{\{\sigma(x)\le \vert k\vert\le \sigma(x)+\delta\}}\frac{\vert f(x+k)-f(x+ k+(\tt-1)k+u)\vert^2}{\vert k\vert^{d+1}}dkdxdu.
\]
Since $|(\tt-1)k+u|\le h$, by the changes  of variables $u\mapsto (\tt-1)k+u$ and then $k\mapsto k+x$, we get
\begin{equation*}
\begin{split}
J_1&\lesssim h^{-d} \int_{u\in B(0,h)}\int_{x,k\in\mathbb{R}^d}\frac{\vert f(x+k)-f(x+k+u)\vert^2}{\vert k\vert^{d+1}}\mathfrak{1}_{\{\sigma(x)\le\vert k\vert\le\sigma(x)+\delta\}}dkdxdu\\
&\lesssim h^{-d}\int_{z\in\mathbb{R}^d} \int_{u\in B(0,h)}\vert f(z)-f(z+u)\vert^2 \left\{\int_{x\in\mathbb{R}^d}\frac{1}{\vert z-x\vert^{d+1}}\mathfrak{1}_{\{\sigma(x)\le\vert z-x\vert\le\sigma(x)+\delta\}}dx\right\}dudz\\
&\lesssim \int_{z\in\mathbb{R}^d} \int_{u\in B(0,h)}\frac{\vert f(z)-f(z+u)\vert^2}{\vert u\vert^{d+1}} dudz\\
&\lesssim \Vert f\Vert_{H^\mez_{\sigma}}^2
\end{split}
\end{equation*}
which finishes the proof.
\section{Proof of \eqref{RT} and \eqref{RT2}}\label{appendixB}
Since $p^+(x, \eta(x))=p^-(x, \eta(x))$ we have
\[
\na_xp^+-\na_xp^-=(\p_yp^--\p_yp^+)\na \eta\quad\text{on}~\Sigma=\{y=\eta(x)\}.
\]
Then 
\[
\sqrt{1+|\na \eta|^2}RT=-(\na_x p^+-\na_xp^-)\vert_\Sigma\cdot \na\eta+(\p_y p^+-\p_yp^-)\vert_\Sigma=(\p_y p^+-\p_yp^-)\vert_\Sigma(1+|\na \eta|^2).
\]
Finally, using the fact that
\[
\p_yp^\pm\vert_\Sigma=\p_yq^\pm\vert_\Sigma-\rho^\pm=B^\pm-\rho^\pm,
\]
we obtain
\[
RT=\sqrt{1+|\na \eta|^2}[(\rho^--\rho^+)-(B^--B^+)]
\]
which proves \eqref{RT}. As for \eqref{RT2}, we use the Darcy law \eqref{Darcy:pm} and the continuity \eqref{u.n:pm} of $u\cdot n$ to have
\[
\mu^\pm u\cdot n+\na p^\pm=-(0, \rho^\pm)\cdot n=-\rho^\pm (1+|\na\eta|^2)^{-\mez},
\]
yielding
\[
RT=(\na p^+-\na p^-)\cdot n=(\mu^--\mu^+)u\cdot n+(\rho^--\rho^+)\rho^\pm (1+|\na\eta|^2)^{-\mez}.
\]


\section{Paradifferential calculus}\label{appendix:para}
This section is devoted to a review of basic features of Bony's paradifferential calculus (see e.g. \cite{Bony, Hormander, MePise,ABZ3}). 
\begin{defi}\label{defi:para}
1. (Symbols) Given~$\rho\in [0, \infty)$ and~$m\in\Rr$,~$\Gamma_{\rho}^{m}(\Rr^d)$ denotes the space of
locally bounded functions~$a(x,\xi)$
on~$\Rr^d\times(\Rr^d\setminus 0)$,
which are~$C^\infty$ with respect to~$\xi$ for~$\xi\neq 0$ and
such that, for all~$\alpha\in\Nn^d$ and all~$\xi\neq 0$, the function
$x\mapsto \partial_\xi^\alpha a(x,\xi)$ belongs to~$W^{\rho,\infty}(\Rr^d)$ and there exists a constant
$C_\alpha$ such that,
\begin{equation*}
\forall |\xi|\ge \mez,\quad 
\Vert \partial_\xi^\alpha a(\cdot,\xi)\Vert_{W^{\rho,\infty}(\Rr^d)}\le C_\alpha
(1+|\xi|)^{m-|\alpha|}.
\end{equation*}
Let $a\in \Gamma_{\rho}^{m}(\Rr^d)$, we define the semi-norm
\begin{equation}\label{defi:norms}
M_{\rho}^{m}(a)= 
\sup_{|\alpha|\le 2(d+2) +\rho ~}\sup_{|\xi| \ge \mez~}
\Vert (1+|\xi|)^{|\alpha|-m}\partial_\xi^\alpha a(\cdot,\xi)\Vert_{W^{\rho,\infty}(\Rr^d)}.
\end{equation}
2. (Paradifferential operators) Given a symbol~$a$, we define
the paradifferential operator~$T_a$ by
\begin{equation}\label{eq.para}
\widehat{T_a u}(\xi)=(2\pi)^{-d}\int \chi(\xi-\eta,\eta)\widehat{a}(\xi-\eta,\eta)\Psi(\eta)\widehat{u}(\eta)
\, d\eta,
\end{equation}
where
$\widehat{a}(\theta,\xi)=\int e^{-ix\cdot\theta}a(x,\xi)\, dx$
is the Fourier transform of~$a$ with respect to the first variable; 
$\chi$ and~$\Psi$ are two fixed~$C^\infty$ functions such that:
\begin{equation}\label{cond.psi}
\Psi(\eta)=0\quad \text{for } |\eta|\le \frac{1}{5},\qquad
\Psi(\eta)=1\quad \text{for }|\eta|\geq \frac{1}{4},
\end{equation}
and~$\chi(\theta,\eta)$ 
satisfies, for~$0<\eps_1<\eps_2$ small enough,
$$
\chi(\theta,\eta)=1 \quad \text{if}\quad |\theta|\le \eps_1| \eta|,\qquad
\chi(\theta,\eta)=0 \quad \text{if}\quad |\theta|\geq \eps_2|\eta|,
$$
and such that
$$
\forall (\theta,\eta), \qquad | \partial_\theta^\alpha \partial_\eta^\beta \chi(\theta,\eta)|\le 
C_{\alpha,\beta}(1+| \eta|)^{-|\alpha|-|\beta|}.
$$
\end{defi}
\begin{rema}
The cut-off $\chi$ can be appropriately chosen so that when $a=a(x)$, the paradifferential operator $T_au$ becomes the usual paraproduct.
\end{rema}
\begin{defi}\label{defi:order}
Let~$m\in\Rr$.
An operator~$T$ is said to be of  order~$m$ if, for all~$\mu\in\Rr$,
it is bounded from~$H^{\mu}$ to~$H^{\mu-m}$.
\end{defi}
Symbolic calculus for paradifferential operators is summarized in the following theorem.
\begin{theo}\label{theo:sc}(Symbolic calculus)
Let~$m\in\Rr$ and~$\rho\in [0,1)$. \\
$(i)$ If~$a \in \Gamma^m_0(\Rr^d)$, then~$T_a$ is of order~$m$. 
Moreover, for all~$\mu\in\Rr$ there exists a constant~$K$ such that
\begin{equation}\label{esti:quant1}
\Vert T_a \Vert_{H^{\mu}\rightarrow H^{\mu-m}}\le K M_{0}^{m}(a).
\end{equation}
$(ii)$ If~$a\in \Gamma^{m}_{\rho}(\Rr^d), b\in \Gamma^{m'}_{\rho}(\Rr^d)$ then 
$T_a T_b -T_{ab}$ is of order~$m+m'-\rho$. 
Moreover, for all~$\mu\in\Rr$ there exists a constant~$K$ such that
\begin{equation}\label{esti:quant2}
\begin{aligned}
\Vert T_a T_b  - T_{a  b} \Vert_{H^{\mu}\rightarrow H^{\mu-m-m'+\rho}}
&\le 
K (M_{\rho}^{m}(a)M_{0}^{m'}(b)+M_{0}^{m}(a)M_{\rho}^{m'}(b)).
\end{aligned}
\end{equation}
$(iii)$ Let~$a\in \Gamma^{m}_{\rho}(\Rr^d)$. Denote by 
$(T_a)^*$ the adjoint operator of~$T_a$ and by~$\overline{a}$ the complex conjugate of~$a$. Then $(T_a)^* -T_{\overline{a}}$ is of order~$m-\rho$  where
Moreover, for all~$\mu$ there exists a constant~$K$ such that
\begin{equation}\label{esti:quant3}
\Vert (T_a)^*   - T_{\overline{a}}   \Vert_{H^{\mu}\rightarrow H^{\mu-m+\rho}}
\le 
K M_{\rho}^{m}(a).
\end{equation}
\end{theo}
\begin{rema}\label{rema:low}
In the definition \eqref{eq.para} of paradifferential operators, the cut-off $\Psi$ removes the low frequency part of $u$. In particular, when $a\in \Gamma^m_0$  we have
\[
\Vert T_a u\Vert_{H^\sigma}\le CM_0^m(a)\Vert \nabla u\Vert_{H^{\sigma+m-1}}. 
\]
\end{rema}
To handle symbols of negative Zygmund regularity, we shall appeal to the following. 
\begin{prop}[\protect{\cite[Proposition 2.12]{ABZ3}}]\label{prop:negOp}
Let  $m\in \Rr$ and $\rho<0$. We denote by $\dot \Gamma^m_\rho(\Rr^d)$ the class of symbols $a(x, \xi)$ that are homogeneous of order $m$ in $\xi$, smooth in $\xi\in \Rr^d\setminus \{0\}$ and such that  
\[
M_{\rho}^{m}(a)= 
\sup_{|\alpha|\le 2(d+2) +\rho ~}\sup_{|\xi| \ge \mez~}
\Vert |\xi|^{|\alpha|-m}\partial_\xi^\alpha a(\cdot,\xi)\Vert_{C^\rho_*(\Rr^d)}<\infty.
\]
 If $a\in \dot \Gamma^m_\rho$ then the operator $T_a$ defined by  \eqref{eq.para} is of order~$m-\rho$.
\end{prop}
\begin{nota}
If $a$ and $u$ depend on a parameter $z\in J\subset \Rr$ we denote 
\[
(T_au)(z)=T_{a(z)}u(z),\quad M^m_\rho(a; J)=\sup_{z\in J}M^m_\rho(a(z)).
\]
If $M^m_\rho(a; J)$ is finite we write $a\in \Gamma^m_\rho(\Rr^d\times J)$.
\end{nota}
\begin{defi}
Given two functions~$a,~u$ defined on~$\Rr^d$  the Bony's remainder is defined by
$$
R(a,u)=au-T_a u-T_u a.
$$
\end{defi}
We gather here several useful product and paraproduct rules.
\begin{theo}\label{pproduct}
Let $s_0$, $s_1$ and $s_2$ be real numbers.
\begin{enumerate}
\item For any $s\in \Rr$, 
\bq\label{pp:Linfty}
\| T_a u\|_{H^s}\le C\| a\|_{L^\infty}\| u\|_{H^s}.
\eq
\item If
$s_0\le s_2$ and~$s_0 < s_1 +s_2 -\frac{d}{2}$, 
then
\begin{equation}\label{boundpara}
\Vert T_a u\Vert_{H^{s_0}}\le C \Vert a\Vert_{H^{s_1}}\Vert u\Vert_{H^{s_2}}.
\end{equation}
\item  If~$s_1+s_2>0$ then
\begin{align}
&\Vert R(a,u) \Vert _{H^{s_1 + s_2-\frac{d}{2}}(\Rr^d)}
\leq C \Vert a \Vert _{H^{s_1}(\Rr^d)}\Vert u\Vert _{H^{s_2}(\Rr^d)},\label{Bony1}\\
&\Vert R(a,u) \Vert _{H^{s_1+s_2}(\Rr^d)} \leq C \Vert a \Vert _{C^{s_1}_*(\Rr^d)}\Vert u\Vert _{H^{s_2}(\Rr^d)}.\label{Bony2}
\end{align}
\item If $s_1+s_2> 0$, $s_0\le s_1$ and $s_0< s_1+s_2-\frac{d}{2}$ then  
\bq\label{paralin:product}
\Vert au - T_a u\Vert_{H^{s_0}}\le C \Vert a\Vert_{H^{s_1}}\Vert u\Vert_{H^{s_2}}.
\eq
\item \label{it.1} If $s_1+s_2> 0$, $s_0\le s_1$, $s_0\le s_2$ and  $s_0< s_1+s_2-\frac d2
$ then 
\begin{equation}\label{pr}
\Vert u_1 u_2 \Vert_{H^{s_0}}\le C \Vert u_1\Vert_{H^{s_1}}\Vert u_2\Vert_{H^{s_2}}.
\end{equation}
\end{enumerate}
\end{theo}
\begin{theo}\label{est:nonl}
  Consider~$F\in C^\infty(\Cc^N)$ such that~$F(0)=0$. 
  
  (i) For $s>\frac{d}{2}$, there exists a non-decreasing function~$\mathcal{F}\colon\Rr_+\rightarrow\Rr_+$ 
such that, for any~$U\in H^s(\Rr^d)^N$,
\begin{equation}\label{est:F(u):S}
\Vert F(U)\Vert_{H^s}\le \mathcal{F}\bigl(\Vert U\Vert_{L^\infty}\bigr)\Vert U\Vert_{H^s}.
\end{equation}
(ii) For $s>0$,  there exists an increasing function~$\mathcal{F}\colon\Rr_+\rightarrow\Rr_+$ 
such that, for any~$U\in C_*^s(\Rr^d)^N$,
\begin{equation}\label{est:F(u):Z}
\Vert F(U)\Vert_{C_*^s}\le \mathcal{F}\bigl(\Vert U\Vert_{L^\infty}\bigr)\Vert U\Vert_{C_*^s}.
\end{equation}
\end{theo}
\begin{theo}[\protect{\cite[Theorem~2.92]{BCD} and \cite[Theorem 5.2.4]{MePise}}]\label{paralin:nonl}(Paralinearization for nonlinear functions)
Let $\mu,~\tau$ be positive real numbers and let $F\in C^{\infty}(\mathbb{C}^N)$ be a scalar function satisfying $F(0)=0$. If $U=(u_j)_{j=1}^N$ with $u_j\in H^{\mu}(\Rr^d)\cap C_*^{\tau}(\Rr^d)$ then we have
\bq\label{def:RF}
F(U)=\Sigma_{j=1}^NT_{\p_jF(U)}u_j+R_F(U)
\eq
with
\[
\| R_F(U)\|_{ H^{\mu+\tau}}\le \cF(\| U\|_{L^{\infty}})\| U\|_{C_*^{\tau}}\| U\|_{H^{\mu}}.
\]
\end{theo}

Recall the definitions \eqref{XY}. The next proposition provides parabolic estimates for elliptic paradifferential operators.
\begin{prop}[\protect{\cite[Proposition 2.18]{ABZ3}}]\label{prop:parabolic}
Let~$r\in \Rr$,~$\varrho\in (0,1)$,~$J=[z_0,z_1]\subset\Rr$ and let 
$p\in \Gamma^{1}_{\varrho}(\Rr^d\times J)$ 
satisfying$$
\RE p(z;x,\xi) \geq c |\xi|,
$$
for some positive constant~$c$. Then for any ~$f\in Y^r(J)$ 
and~$w_0\in H^{r}(\Rr^d)$, there exists~$w\in X^{r}(J)$ solution of  the parabolic evolution equation
\begin{equation}\label{eqW}
\partial_z w + T_p w =f,\quad w\arrowvert_{z=z_0}=w_0,
\end{equation}
satisfying 
\begin{equation*}
\| w \|_{X^{r}(J)}\le \cF(\mathcal{M}^1_\varrho(p), \frac{1}{c})\big(\| w_0\|_{H^{r}}+ \| f\|_{Y^{r}(J)}\big),
\end{equation*}
for some increasing function $\cF$ depending only on $r$ and $\varrho$. Furthermore, this solution is unique in ~$X^s(J)$ for any~$s\in \Rr$. 
 \end{prop}
 \section{Proof of Lemma \ref{lemm:XY}}\label{appendix:XY}
 1) Assuming \eqref{cd:XY1}, we prove \eqref{XY1}. We decompose $u_1u_2=T_{u_1}u_2+T_{u_2}u_1+R(u_1, u_2)$. Since $s_0\le s_2+1$ and $s_1+s_2>s_0+\frac{d}{2}-1$,  \eqref{boundpara} gives 
 \[
 \| T_{u_1}u_2\|_{L^2 H^{s_0-\mez}}\les \| u_1\|_{L^\infty H^{s_1}}\|u_2\|_{L^2 H^{s_2+\mez}}\les \|u_1\|_{X^{s_1}}\|u_2\|_{X^{s_2}}.
 \]
 The paraproduct $T_{u_2}u_1$ can be estimated similarly. As for $R(u_1, u_2)$ we use \eqref{Bony1} and the conditions $s_1+s_2+1>0$ and $s_1+s_2>s_0+\frac{d}{2}-1$:
 \[
 \| R(u_1, u_2)\|_{L^1 H^{s_0}}\les \|u_1\|_{L^2H^{s_1+\mez}}\| u_2\|_{L^2 H^{s_2+\mez}}\les \|u_1\|_{X^{s_1}}\|u_2\|_{X^{s_2}}.
 \]
 
 2) Assuming \eqref{cd:XY2}, we prove \eqref{XY2}. Again, we decompose $u_1u_2=T_{u_1}u_2+T_{u_2}u_1+R(u_1, u_2)$. If $u_1=a+b$ with $a\in L^1 H^{s_1}$ and $b\in L^2 H^{s_1-\mez}$ then 
 \[
 u_1u_2=T_{a}u_2+T_bu_2+T_{u_2}a+T_{u_2}b+R(a, u_2)+R(b, u_2).
 \]
  Using the conditions  $s_0\le s_2$ and $s_1+s_2>s_0+\frac{d}{2}$, we can apply \eqref{boundpara} to have
 \begin{align*}
& \| T_{a}u_2\|_{L^1 H^{s_0}}\les \| a\|_{L^1 H^{s_1}}\|u_2\|_{L^\infty H^{s_2}},\\
& \| T_{b}u_2\|_{L^2 H^{s_0-\mez}}\les \| b\|_{L^2 H^{s_1-\mez}}\|u_2\|_{L^\infty H^{s_2}}.
  \end{align*}
 Next from the conditions  $s_0\le s_1$ and $s_1+s_2>s_0+\frac{d}{2}$, \eqref{boundpara} yields \begin{align*}
& \| T_{u_2}a\|_{L^1 H^{s_0}}\les\|u_2\|_{L^\infty H^{s_2}} \| a\|_{L^1 H^{s_1}},\\
& \| T_{u_2}b\|_{L^2 H^{s_0-\mez}}\les \|u_2\|_{L^\infty H^{s_2}}\| b \|_{L^2 H^{s_1-\mez}}.
  \end{align*}
Finally, under the conditions  $s_1+s_2>0$ and $s_1+s_2>s_0+\frac{d}{2}$, using \eqref{Bony1} we obtain
\begin{align*}
&\|R(a, u_2)\|_{L^1H^{s_0}}\les \|a\|_{L^1 H^{s_1}}\| u_2\|_{L^\infty H^{s_2}},\\
&\|R(b, u_2)\|_{L^1H^{s_0}}\les \|b\|_{L^2 H^{s_1-\mez}}\| u_2\|_{L^2 H^{s_2+\mez}}.
\end{align*}
We have proved that 
\[
\| u_1u_2\|_{Y^{s_0}}\les (\|a\|_{L^1 H^{s_1}}+\|b\|_{L^2 H^{s_1-\mez}})\| u_2\|_{X^{s_2}}
\]
for any decomposition $u_1=a+b$. Therefore, \eqref{XY2} follows. We have also proved \eqref{est:XY22}, \eqref{est:XY3} and \eqref{est:XY4}.

\vspace{.1in}
\noindent{\bf{Acknowledgment.}} 
The work of HQN was partially supported by NSF grant DMS-1907776. BP was partially supported by NSF grant DMS-1700282. We would like to thank the reviewer for his/her positive and insightful comments on the manuscript.


\begin{thebibliography}{10}
\small

\bibitem{Ai} A. Ai. Low regularity solutions for gravity water waves, preprint, arXiv:1712.07821.

\bibitem{ABZ1}
T. Alazard, N. Burq, and C. Zuily.
\newblock On the water waves equations with surface tension.
\newblock {\em Duke Math. J.}, 158(3):413--499, 2011.

\bibitem{ABZ3}
T. Alazard, N. Burq, and C. Zuily.
\newblock On the Cauchy problem for gravity water waves.
\newblock {\em Invent. Math.}, 198(1): 71--163, 2014.

\bibitem{ABZ4} T. Alazard, N. Burq, and C. Zuily. Strichartz estimates and the Cauchy problem for the gravity water waves equations. {\it Memoirs of the AMS}. Volume 256, 2018.

\bibitem{AlaLaz} T. Alazard and O. Lazar, Paralinearization of the Muskat equation and application to the Cauchy problem, preprint arXiv:1907.02138.

\bibitem{AlaMeuSme} T. Alazard, N. Meunier and D. Smets, Lyapounov functions, Identities and the Cauchy problem for the Hele-Shaw equation, preprint arXiv:1907.03691.

\bibitem{ThoGuy} T. Alazard and G. M\'etivier. Paralinearization of the Dirichlet to Neumann operator, and regularity of diamond waves.
{\it Comm. Partial Differential Equations}, 34 (2009), no. 10-12, 1632--1704.

\bibitem{Amb0} D. M. Ambrose. Well-posedness of two-phase Hele-Shaw flow without surface tension. {\it European J. Appl. Math.}, 15(5):597--607, 2004.

\bibitem{Amb} D. M. Ambrose, Well-posedness of two-phase Darcy flow in 3D. {\it Q. Appl. Math.} 65(1), 189--203 (2007)

\bibitem{BCD}
H. Bahouri, J-Y Chemin, and R. Danchin.
\newblock {\em Fourier analysis and nonlinear partial differential equations},
  volume 343 of {\em Grundlehren der Mathematischen Wissenschaften [Fundamental
  Principles of Mathematical Sciences]}.
\newblock Springer, Heidelberg, 2011.

\bibitem{BerCorGra} L. C. Berselli, D. C\'ordoba, and R. Granero-Belinch\'on. Local solvability and turning for the inhomogeneous Muskat problem. {\it Interfaces Free Bound.}, 16(2):175--213, 2014.

\bibitem{Bony}
J-M. Bony.
\newblock Calcul symbolique et propagation des singularit\'es pour les
  \'equations aux d\'eriv\'ees partielles non lin\'eaires.
\newblock {\em Ann. Sci. \'Ecole Norm. Sup. (4)}, 14(2):209--246, 1981.

\bibitem{Cam} S. Cameron. Global well-posedness for the two-dimensional Muskat problem with slope less than 1. {\it Anal. PDE} Vol12, Number 4 (2019), 997-1022.
\bibitem{CasCorFefGan} A. Castro, D.  C\'ordoba, C. Fefferman and F. Gancedo, Breakdown of smoothness for the Muskat problem. {\it Arch. Ration. Mech. Anal.} 208(3), 805--909 (2013).

\bibitem{CasCorFefGanMar} A. Castro, D. C\'ordoba, C. L. Fefferman, F. Gancedo and Mar\'ia L\'opez-Fern\'andez. Rayleigh Taylor breakdown
for the Muskat problem with applications to water waves. {\it Annals of Math} 175, no. 2, 909--948, 2012.

\bibitem{ChlGuiSch} H. A. Chang-Lara, N. Guillen, R. W. Schwab, Some free boundary problems recast as nonlocal parabolic equations, preprint:arXiv:1807.02714.

\bibitem{Chen} X. Chen. The hele-shaw problem and area-preserving curve-shortening motions.
{\it Arch. Ration. Mech. Anal.}, 123(2):117--151, 1993.

\bibitem{CheGraShk} C.H. Cheng, R. Granero-Belinch\'on and S. Shkoller, Well-posedness of the Muskat problem with $H^2$ initial data. {\it Adv. Math.} 286, 32--104 (2016).

\bibitem{ConCorGanRodPStr} P. Constantin, D. C\'ordoba, F. Gancedo, L. Rodriguez-Piazza and R.M. Strain, On the Muskat problem: global in time results in 2D and 3D, {\it Amer. J. Math.}, 138 (6) (2016), pp. 1455-1494.

\bibitem{ConCorGanStr} P. Constantin, D. C\'ordoba, F. Gancedo and R. M. Strain. On the global existence for the Muskat problem. {\it J. Eur. Math. Soc.}, 15, 201-227, 2013.

\bibitem{ConGanShvVic} P. Constantin, F. Gancedo, R. Shvydkoy and V. Vicol, Global regularity for 2D Muskat equations with finite slope.  {\it Ann. Inst. H. Poincar\'e Anal. Non Lin\'eaire} 34 (2017), no. 4, 1041--1074.

\bibitem{ConPug} P. Constantin, M. Pugh. Global solutions for small data to the Hele-Shaw problem.  {\it Nonlinearity} 6 (1993), no. 3, 393--415.

\bibitem{CorCorGan} A. C\'ordoba, D. C\'ordoba, and F. Gancedo, Interface evolution: the Hele-Shaw and Muskat problems. {\it Ann. Math.} 173(1), 477--542 (2011).

\bibitem{CorCorGan2} A. C\'ordoba, D. C\'ordoba, and F. Gancedo. Porous media: the Muskat problem in three dimensions. {\it Anal. \& PDE}, 6(2):447--497, 2013.

\bibitem{CorGan07} D. C\'ordoba, F. Gancedo
Contour dynamics of incompressible 3-D fluids in a porous medium with different densities
{\it Comm. Math. Phys.}, 273 (2) (2007), pp. 445--471

 \bibitem{CorGan09}
 D. C\'ordoba, and F. Gancedo. A maximum principle for the Muskat problem for fluids with different densities. {\it Comm. Math. Phys.}, 286(2): 681--696, 2009.
 
 \bibitem{CorGraOri} D. C\'ordoba, R. Granero-Belinch\'on and R. Orive, The confined Muskat problem: differences with the deep water regime. {\it Commun. Math. Sci.} 12(3), 423--455 (2014).
 
 \bibitem{CorLaz} D. C\'ordoba and O. Lazar. Global well-posedness for the 2d stable Muskat problem in H3/2, preprint (2018), arXiv:1803.07528.
 
 \bibitem{DenLeiLin} F. Deng, Z. Lei, and F. Lin, On the Two-Dimensional Muskat Problem with Monotone Large Initial Data. {\it Comm. Pure Appl. Math.}, 70 (2017), no. 6, 1115?1145.
 
 \bibitem{Poy} T. de Poyferr\'e, A Priori Estimates for Water Waves with Emerging Bottom, {\it Arch. Ration. Mech. Anal.}, 2019, Volume 232, Issue 2, pp 763--812.
 
 \bibitem{PoyNg1}  T. de Poyferr\'e,  H. Q. Nguyen. A paradifferential reduction for the gravity-capillary waves system at low regularity and applications. {\it Bull. Soc. Math. France}, 145(4): 643--710, 2017.
  
   \bibitem{PoyNg2} T. de Poyferre, H. Q. Nguyen. Strichartz estimates and local existence for the gravity-capillary water waves with non-Lipschitz initial velocity. {\it J.  Differ. Equ.}, 261(1) 396--438, 2016.
 
 \bibitem{EscSim} J. Escher and G. Simonett. Classical solutions for Hele-Shaw models with surface tension. {\it Adv. Differential Equations}, no. 2, 619-642, 1997.
 
  \bibitem{EscMat} J. Escher and B.V. Matioc, On the parabolicity of the Muskat problem: well-posedness, fingering, and stability results. {\it Z. Anal. Anwend.} 30(2), 193--218 (2011).
  
  \bibitem{FlyNgu} P. Flynn and H. Q. Nguyen. The vanishing surface tension limit of the Muskat problem, {\it in preparation}. 
  
\bibitem{GanSur} F. Gancedo, A survey for the Muskat problem and a new estimate, {\it SeMA} (2017), Volume 74, Issue 1, pp 21--35.

\bibitem{GomGra} J. G\'omez-Serrano and R. Granero-Belinch\'on, On turning waves for the inhomogeneous Muskat problem: a computer-assisted proof. {\it Nonlinearity} 27(6), 1471--1498 (2014).

\bibitem{Gra} R. Granero-Belinch\'on, Global existence for the confined Muskat problem. {\it SIAM J. Math. Anal. }46(2), 1651--1680 (2014)

\bibitem{GanGarPatStr} F. Gancedo,  E. Garc\'ia-Ju\'arez, N. Patel and R. M. Strain, On the Muskat problem with viscosity jump: global in time results. {\it Adv. Math}. 345 (2019), 552--597.

\bibitem{GraLaz} R. Granero-Belinch\'on and O. Lazar, Growth in the Muskat problem, preprint (2019), arXiv:1904.00294.

\bibitem{GraShk} R. Granero-Belinch\'on and S. Shkoller. Well-posedness and decay to equilibrium for the muskat problem with discontinuous permeability, {\it Transactions of the American Mathematical Society}, to appear. 

\bibitem{GuoHalSpi} Y. Guo, C. Hallstrom, D. Spirn, Dynamics near unstable, interfacial fluids, {\it Comm. Math. Phys}. 270 (3) (2007) 635--689.

\bibitem{Hel1}H. S. Hele-Shaw. The flow of water. {\it Nature}, 58:34--36, 1898.

\bibitem{Hel2} H. S. Hele-Shaw. On the motion of a viscous fluid between two parallel plates. {\it Trans. Royal Inst. Nav. Archit.}, 40:218, 1898.
 
\bibitem{Hormander}
L. H{\"o}rmander.
\newblock {\em Lectures on nonlinear hyperbolic differential equations},
  volume~26 of {\em Math\'ematiques \& Applications (Berlin) [Mathematics \&
  Applications]}.
\newblock Springer-Verlag, Berlin, 1997.

\bibitem{HunIfrTat} J.K. Hunter, M. Ifrim and D. Tataru. Two Dimensional Water Waves in Holomorphic Coordinates, {\it Commun. Math. Phys.} (2016) 346--483. 


\bibitem{LannesJAMS}
D. Lannes.
\newblock Well-posedness of the water waves equations.
\newblock {\em J. Amer. Math. Soc.}, 18(3):605--654 (electronic), 2005.

\bibitem{Lions}
J. L Lions and E. Magenes.
\newblock Probl\`emes aux limites non homog\`enes et applications.
\newblock Vol.1 Dunod, 1968.


\bibitem{LeoTice}
G. Leoni and I. Tice.
Traces for homogeneous Sobolev spaces in infinite strip-like domains. {\it J. Funct. Anal.} 277(7), 2288--2380, 2019.

\bibitem{Mat} B-V Matioc. The muskat problem in 2d: equivalence of formulations, well-posedness, and regularity
results. {\it Analysis \& PDE}, 12(2), 281-332, 2018.

\bibitem{Mat2} B-V Matioc, Viscous displacement in porous media: the Muskat problem in 2D,
{\it Trans. Amer. Math. Soc.}, 370(10):7511-7556, 2018.

\bibitem{MePise}
G. M{\'e}tivier.
\newblock {\em Para-differential calculus and applications to the {C}auchy
  problem for nonlinear systems}, volume~5 of {\em Centro di Ricerca Matematica
  Ennio De Giorgi (CRM) Series}.
\newblock Edizioni della Normale, Pisa, 2008.

\bibitem{Mus} M. Muskat, Two Fluid systems in porous media. The encroachment of water into an oil sand. {\it Physics} 5, 250--264 (1934)

\bibitem{Ng}H. Q. Nguyen. A sharp Cauchy theory for 2D gravity-capillary water waves. {\it Ann. Inst. H. Poincar\'e Anal. Non Lin\'eaire}, 34(7), 1793--1836, 2017.

\bibitem{Ng1} H. Q. Nguyen. On well-posedness of the Muskat problem with surface tension. arXiv:1907.11552 [math.AP], 2019.

\bibitem{Per} T. Pernas-Casta\~no. Local-existence for the inhomogeneous Muskat problem. {\it Nonlinearity}, 30(5):2063, 2017.

\bibitem{SafTay} P. G. Saffman and G. Taylor. The penetration of a fluid into a porous medium or Hele-Shaw cell containing a
more viscous liquid. {\it Proc. Roy. Soc. London.} Ser. A, 245:312--329. (2 plates), 1958.

\bibitem{Safo}
M.V. Safonov. Boundary estimates for positive solutions to second order elliptic equations. arXiv:0810.0522.

\bibitem{SieCafHow} M. Siegel, R. Caflisch and S. Howison, Global existence, singular solutions, and Ill-posedness for the Muskat problem. {\it Commun. Pure Appl. Math.} 57, 1374--1411 (2004).

\bibitem{Stri}
 R. S. Strichartz. ``Graph paper'' trace characterizations of functions of finite energy. {\it J. Anal. Math.}, 128: 239--260,
2016.

\bibitem{Wu2d} S. Wu. Well-posedness in Sobolev spaces of the full water wave problem in 2-D. {\it Invent. Math.} 130, 39--72 (1997).

\bibitem{Wu3d} S. Wu. Well-posedness in Sobolev spaces of the full water wave problem in 3-D. {\it Journal of the American Mathematical Society}, 12(2):445--495, 1999.

\bibitem{Yi} F. Yi, Local classical solution of Muskat free boundary problem. {\it J. Partial Differ. Equ.} 9, 84--96 (1996).


\end{thebibliography}
\end{document}